\documentclass[11pt, twoside]{article}

\usepackage{enumerate}
\usepackage{graphics}
\usepackage{graphicx}
\usepackage{amssymb}
\usepackage{amsmath}
\usepackage{amsthm}

\usepackage[title,titletoc]{appendix}

\usepackage{color}
\usepackage{mathrsfs}

\usepackage{indentfirst}

\usepackage{txfonts}

\usepackage{anysize}

\allowdisplaybreaks

\newtheorem{theorem}{Theorem}[section]
\newtheorem{lemma}[theorem]{Lemma}
\newtheorem{corollary}[theorem]{Corollary}
\newtheorem{proposition}[theorem]{Proposition}
\theoremstyle{definition}
\newcounter{assum}

\newtheorem{assumption}[assum]{Assumption}
\newtheorem{remark}[theorem]{Remark}
\newtheorem{example}[theorem]{Example}
\newtheorem{definition}[theorem]{Definition}
\renewcommand{\appendix}{\par
	\setcounter{section}{0}%
	\setcounter{subsection}{0}%
	\setcounter{subsubsection}{0}%
	\gdef\thesection{\@Alph\c@section}%
	\gdef\thesubsection{\@Alph\c@section.\@arabic\c@subsection}%
	\gdef\theHsection{\@Alph\c@section.}%
	\gd\theHsubsection{\@Alph\c@section.\@arabic\c@subsection}%
	\csname appendixmore\endcsname
}

\newtheorem{letterthm}{Theorem}

\numberwithin{equation}{section}

\begin{document}

\title{\vspace{-0cm}
\bf\Large Difference and Wavelet Characterizations
of Distances from
Functions in Lipschitz Spaces to Their Subspaces
\footnotetext{\hspace{-0.35cm} Dedicated to Hans 
Triebel on the occasion of his 90th birthday.
\endgraf 2020 {\it
Mathematics Subject Classification}. Primary 42B35;
Secondary 26A16, 42B25, 42C40, 46E35.
\endgraf {\it Key words and phrases.} Daubechies wavelet,
difference, Lipschitz Space, Besov--Triebel–-Lizorkin space,
Whitney-type inequality, distance, closure.
\endgraf This project is supported by the National Key Research
and Development Program of China
(Grant No. 2020YFA0712900), the National Natural Science
Foundation of China
(Grant Nos. 12431006 and 12371093),
and
the Fundamental Research Funds for the Central
Universities (Grant No. 2233300008).
The first author is also supported by
NSERC of Canada Discovery grant RGPIN-2020-03909.}}
\author{Feng Dai, Eero Saksman,
Dachun Yang, Wen Yuan and Yangyang Zhang}
\date{}
\maketitle

\vspace{-0.8cm}

\begin{center}
\begin{minipage}{13.8cm}
{\small {\bf Abstract}\quad
Let $\Lambda_s$ denote
the Lipschitz space of  order $s\in(0,\infty)$
on $\mathbb{R}^n$, which consists of all
$f\in\mathfrak{C}\cap L^\infty$ such that, for some constant $L\in(0,\infty)$
and some integer $r\in(s,\infty)$,
\begin{equation*}
\label{0-1}\Delta_r f(x,y):
=\sup_{|h|\leq y} |\Delta_h^r f(x)|\leq L y^s,
\ x\in\mathbb{R}^n,
\ y \in(0, 1].
\end{equation*}
Here (and throughout the article) $\mathfrak{C}$ refers to continuous functions, and
 $\Delta_h^r$ is the usual $r$-th
order difference operator
with step $h\in\mathbb{R}^n$. For each
$f\in \Lambda_s$ and $\varepsilon\in(0,L)$,
let $ S(f,\varepsilon):=
\{ (x,y)\in\mathbb{R}^n\times [0,1]:
\frac {\Delta_r f(x,y)}{y^s}>\varepsilon\}$, and let
$\mu: \mathcal{B}(\mathbb{R}_+^{n+1})\to
[0,\infty]$
be a suitably defined nonnegative
extended real-valued function on the Borel
$\sigma$-algebra of subsets of $\mathbb{R}_+^{n+1}$.
Let $\varepsilon(f)$ be the infimum
of all $\varepsilon\in(0,\infty)$ such that
$\mu(S(f,\varepsilon))<\infty$.
The main target of this article is to characterize 
the distance from $f$ to a subspace
$V\cap \Lambda_s$ of $\Lambda_s$
for various function spaces $V$ (including  Sobolev, Besov--Triebel--Lizorkin, 
and Besov--Triebel--Lizorkin-type spaces)
in terms of $\varepsilon(f)$,
showing that
\begin{equation*} \varepsilon(f)\sim \mathrm{dist}
(f, V\cap \Lambda_s)_{\Lambda_s}: =
\inf_{g\in \Lambda_s\cap V} \|f-g\|_{\Lambda_s}.\end{equation*}
These results extend the characterization of the least constant in the
John-Nirenberg inequality via the distance in BMO
to $L^\infty(\mathbb{R}^n)$ by Garnett and Jones
in [Ann. of Math. (2) 108 (1978)]
as well as the recent work of Saksman and Soler i Gibert
in [Canad. J. Math. 74 (2022)] on the distance in $\Lambda_s$ to the subspace
$\mathrm{J}_s(\mathop\mathrm{bmo})$ for any $s \in(0, 1]$.
Moreover, we present our results in a general framework based on quasi-normed
lattices of function sequences $X$ and
Daubechies $s$-Lipschitz $X$-based spaces.
}
\end{minipage}
\end{center}

\tableofcontents

\section{Introduction}

The analytical characterizations of both
the closures of subspaces of a
given function space, under the topology
of that function space,
and the related distance from any element
of that function space to its subspaces
play an important role in many applications
in analysis.
Here is a brief summary of some notable results
in this area. (Since we work on $\mathbb{R}^n$,
we omit the underlying space $\mathbb{R}^n$
in all the related symbols whenever
there is no risk of ambiguity.)
\begin{itemize}
\item
Anderson et al. \cite{acp}
provided a characterization of the closure of
polynomials via derivatives
in the Bloch space, which they used to explore
the dual properties of
the Bloch space and the zeros of Bloch functions.
They also posed an open problem regarding the
closure of bounded
functions in the Bloch space in \cite{acp},
which is still open.
\item  Sarason \cite{s75}
established  a  characterization
of the closure of smooth functions via finite
differences  under the norm of the space
$\mathop\mathrm{\,BMO\,}$
(introduced by John and Nirenberg  \cite{JN}),
identifying
this closure as the well-known
vanishing mean oscillation space
$\mathop\mathrm{\,VMO\,}$.
This  characterization was subsequently applied
to the  study of Riesz
transforms and the regularity for solutions
of certain equations
(see, for example, \cite{cfl,s}).
Let $L^1_{\mathrm{loc}}$
denote the set of all locally integrable
functions on $\mathbb{R}^n$. Recall that
the space $\mathop\mathrm{\,BMO\,}$ is defined
to be the set of all $f\in L^1_{\mathrm{loc}}$
such that $\|f\|_{\mathop\mathrm{\,BMO\,}}:=
\sup_{B}\frac{1}{|B|}
\int_{B} |f(x)-\frac{1}{|B|}
\int_{B}f(y)\,dy|\, dx
<\infty$, where the supremum is taken over
all
balls $B$ in $\mathbb{R}^n$.

\item  Neri \cite{n} provided a difference
characterization of
the closure of smooth functions with compact
support under the
norm of  $\mathop\mathrm{\,BMO\,}$.
This characterization has been used to study
the compactness of
Calder\'on--Zygmund commutators, the Beltrami
equation, Riesz transforms,
and the regularity of the Jacobian of certain
quasiconformal mappings (see, for instance, \cite{cc,i,is,u}).
\item
Ghatage and Zheng \cite{gz93} characterized
the distance from  a function
to the space of analytic functions of bounded
mean oscillation
under the norm of  the Bloch space,  addressing
an open problem
raised by  Axler \cite{a} in 1988 concerning
the Bloch space and
providing a derivative characterization of
the closure of
the space of analytic functions
of bounded mean oscillation
in the Bloch space. Furthermore, they
highlighted the importance
of distance characterizations of function spaces
in harmonic analysis and complex analysis, applying
their  characterizations to the study of Hankel
operators and operator algebras.

\item  In 1978,
Garnett and Jones \cite{gj78} established a
celebrated formula
characterizing  the distance from  any
$\mathop\mathrm{\,BMO\,}$
function on $\mathbb{R}^n$
to the subspace $L^\infty$ (the Lebesgue
space) under the
$\mathop\mathrm{\,BMO\,}$ norm, which
in particular
induces a difference characterization of the closure of
$L^\infty$ under the topology of
$\mathop\mathrm{\,BMO\,}$. Moreover,
as applications,
Garnett and Jones \cite{gj78} also established the
higher dimensional
Helson--Szeg\"o theorem and, in 1980, Jones
\cite{j} gave the estimates
for solutions of the corona problem for
$H^\infty$ (the bounded analytic function
space)
of the upper half
plane or unit disk (the former was further refined
by Uchiyama
\cite{u82}) and Jones \cite{j80} established the
factorization
of $A_p$ weights. In 1983, Bourgain \cite{b} used
this distance characterization
to study the embedding problem of the space of
integrable
functions to the predual of $H^\infty$. Two
different proofs of
\cite[Corollary 1.1]{gj78} (the formula of the
distance from
$\mathop\mathrm{\,BMO\,}$
to $L^\infty$) can be found,
respectively, in \cite{gj82,v}. This formula was
recently
generalized by Chen et al. \cite{cdlsy} to the
function space,
on metric spaces with doubling measure, associated
with
nonnegative self-adjoint operators with heat
kernels
having the Gaussian upper bound.
\end{itemize}

The work in this article  is largely  inspired
by the recent research of Saksman and Soler i Gibert
\cite{ss}. They studied an approximation
problem in the Lipschitz space
$\Lambda_s=(\Lambda_s, \|\cdot\|_{\Lambda_s})$ of
smooth order $s \in(0,1]$ on $\mathbb{R}^n$, seeking  to
determine  the  distance (up to a positive constant
multiple) in $\Lambda_s$ to the subspace $\mathrm{J}_s(\mathop\mathrm{\,bmo\,})$
for all $f\in \Lambda_s$ and $s \in(0,1]$.
Here,   $\mathrm{J}_s(\mathop\mathrm{\,bmo\,})$ is
the image  of the space $\mathop\mathrm{\,bmo\,}$ under
the Bessel potential  $\mathrm{J}_s:=(I-\Delta)^{-\frac s2}$,
$\Delta=\sum_{j=1}^n \partial_j^2$ is the Laplace operator,
and  $\mathop\mathrm{b m o}$
is  an  inhomogeneous version of the space
$\mathop\mathrm{\,BMO\,}$ consisting   of all functions $f\in
\mathop\mathrm{\,BMO\,}$ with
\[ \|f\|_{\mathop\mathrm{\,bmo\,}} :=\|f\|_{\mathop\mathrm{\,BMO\,}}
+\sup_{k\in\mathbb{Z}^n}
\left[ \int_{k+[0,1)^n} |f(x)|^2\, dx \right]^{\frac 12}
<\infty.
\]
It is well known that  $\mathrm{J}^s(\mathop\mathrm{\,bmo\,})$ is a
Banach space with norm
$\|f\|_{\mathrm{J}^s(\mathop\mathrm{\,bmo\,})}:=\|\mathrm{J}^{-s}
f\|_{\mathop\mathrm{\,bmo\,}},$
and  the strict inclusion $\mathrm{J}^s(\mathop\mathrm{\,bmo\,})
\subsetneqq \Lambda_s$ holds.

To describe  the results in \cite{ss}
more precisely,
we need to  introduce some necessary notation.
First, for any $k\in\mathbb N$ and $h\in \mathbb{R}^n$,
the $k$-th order  \emph{symmetric
difference operator} $\Delta_h^k$ is defined
recursively   as follows: for any
$f:\mathbb{R}^n\to\mathbb{R}$ and $x\in\mathbb{R}^n$,
$$
\Delta_h(f)(x) :=
f\left(x+\frac h2\right)-f\left(x-\frac h2\right),\
\Delta^{i+1}_h:
= \Delta^i_h(\Delta_h),\   i\in\mathbb{N}.
$$
Given  $s\in(0,\infty)$,
the \emph{Lipschitz space}
$\Lambda_s$ of order $s$ on $\mathbb{R}^n$
is the Banach space of  all continuous
functions $f$ on $\mathbb{R}^n$ with norm
\begin{align*}
\|f\|_{\Lambda_s}:=
\|f\|_{L^\infty(\mathbb{R}^n)}+
\sup_{x\in\mathbb{R}^n,\ y \in(0, 1]}
\frac {\Delta_r f(x,y)}{y^s}<\infty,
\end{align*}
where $r=\lfloor s\rfloor+1$, $\lfloor s\rfloor$
denotes the \emph{largest integer not greater than} $s$,
and
\[ \Delta_k f(x,y)=\sup_{h\in\mathbb{R}^n, |h|=y}
|\Delta_h^kf(x)|,\   x\in \mathbb{R}^n,\
\ y\in(0,\infty),\  k\in\mathbb{N}.\]
It is well known that replacing
$r=\lfloor s\rfloor+1$ with any positive
integer $r>s$ in the definition of
$\|\cdot\|_{\Lambda_s}$  results in an
equivalent norm.

Next, given a nonempty subset $E\subset \Lambda^s$,
we denote
the closure of $E$ under the topology
of $\Lambda_s$ by  $ \overline{E}^{\Lambda_s}$.
For any
$f\in \Lambda_s$ and any nonempty set
$E\subset \Lambda^s$, let
\begin{align*}
\mathop\mathrm{\,dist\,}\left(f,
E\right)_{\Lambda_s}:=&
\inf_{g\in E} \|f-g\|_{\Lambda_s}.
\end{align*}
Clearly, for any  $E\subset \Lambda^s(\mathbb{R}^n)$,
$$\overline{E}^{\Lambda_s}
=\left\{ f\in\Lambda_s:
\mathop\mathrm{\,dist\,}(f, E)_{\Lambda_s}
=0\right\}.$$
In the case when $E$ does not lie entirely
in the space $\Lambda_s$, we define
$$ 	\mathop\mathrm{\,dist\,}
\left(f, E\right)_{\Lambda_s}:=	
\mathop\mathrm{\,dist\,}\left(f, E\cap
\Lambda_s\right)_{\Lambda_s}\   \text{and}\
\overline{E}^{\Lambda_s}
:=\overline{E\cap \Lambda_s}^{\Lambda_s}. $$

One  of the main results of \cite{ss}
can now be described  as follows.
For any integer
$r>s>0$, any $f\in\Lambda_s$, and any $\varepsilon\in(0,\infty)$,
define
\begin{align*}
S_r(s,f, \varepsilon):&=\left\{ (x, y)\in\mathbb{R}_+^{n+1}:
\frac{\Delta_r f(x,y)}
 { y^s}>\varepsilon,\   \   x\in\mathbb{R}^n,\
y\in(0,1] \right\}.
\end{align*}
Clearly, $S_r(s,f,\varepsilon)=\emptyset$ if
$\varepsilon\ge \|f\|_{\Lambda_s}$. For convenience,
let us define $\varepsilon_{r,s}(f)$ to be the infimum
of all $\varepsilon\in(0,\infty)$ such that

\begin{equation}
	\sup_{I\in\mathcal D} \frac 1 {|I|} \int_0^{\ell(I)}
 \left|\left\{ x\in I: \frac{\Delta_r f(x,y)}
 { y^s}>\varepsilon
\right\}\right|\frac {dy} y<\infty.\label{1-1-2}
\end{equation}

Here and throughout the article,
$\mathcal D$ denotes the collection of all dyadic
cubes $I$ in $\mathbb{R}^n$ of edge length
$\ell(I)\leq 1$, and $|E|$ denotes the
$n$-dimensional  Lebesgue measure of
$E\subset \mathbb{R}^n$.
It is worth noting that condition
\eqref{1-1-2} is equivalent to requiring that,
for any $(x,y)\in\mathbb{R}_+^{n+1}$,
\begin{equation*} \label{1-2a}
d\gamma_{f,\varepsilon}(x,y)
:= \frac {\mathbf{1}_{S_r(s,f, \varepsilon)}
(x,y)\mathbf{1}_{[0, 1]}(y)\,dx\,dy}y
\end{equation*}
is a Carleson measure on $\mathbb{R}^{n+1}_+$.
Thus, $\varepsilon(f)$ is the \emph{critical index} for
which $\gamma_{f,\varepsilon}$ is a Carleson measure on
$\mathbb{R}^{n+1}_+$ for all $\varepsilon
\in(\varepsilon_{r,s}(f),\infty)$.
Very remarkably,   Saksman and Soler i Gibert
\cite{ss}  proved that, under certain conditions,
the  critical index $\varepsilon_{r,s}((f)$ characterizes
the distance   $\mathop\mathrm{\,dist\,}(f,
\mathrm{J}^s(\mathop\mathrm{\,bmo\,}))_{\Lambda_s}$
for any $f\in\Lambda_s$. More precisely, they showed
the following theorem.

\begin{letterthm}(\cite{ss})\label{fwfw}
If $s\in(0,1)$ and $r = 1$ or if $s\in(0,1]$
and $r = 2$, then,
for any $f\in\Lambda_s$,
\begin{equation}\label{1-3}
\mathop\mathrm{\,dist\,}(f, \mathrm{J}^s(\mathop\mathrm{\,bmo\,}))_{\Lambda_s}\sim
\varepsilon_{r,s}(f):= \inf\left\{ \varepsilon\in(0,\infty):\text{\eqref{1-1-2} holds} \right\}
\end{equation}
with the positive equivalence constants
depending only on $s$ and $n$.
\end{letterthm}

We point out  that the above
theorem was in turn influenced by the earlier work of Nicolau 
and Soler i Gibert \cite{ns}, which established the result
for $n=s=1$ with a different proof
that does not extend to higher-dimensional situation.
Several variants of these questions have
also been studied in \cite{ln,mn}.

As a direct consequence of Theorem A,
we obtain the following characterization
of  the closure of
$\mathrm{J}_s(\mathop\mathrm{\,bmo\,})$
under the topology of $\Lambda_s$:
\[ f\in\overline{\mathrm{J}_s
(\mathop\mathrm{\,bmo\,})}^{\Lambda_s}\iff
f\in\Lambda_s\  \text{and}\  \varepsilon_{r,s}(f)=0.\]
It is worthwhile to point out that the
following strict inclusions hold for any $s\in(0,\infty)$
(see Section \ref{2.1} for the proof):
$\mathrm{J}_s(\mathop\mathrm{\,bmo\,})
\subsetneqq\overline{\mathrm{J}_s(\mathop
\mathrm{\,bmo\,})}^{\Lambda_s}
\subsetneqq\Lambda_s.$

One of the fundamental tools used in the article \cite{ss} is the first-order
and the second-order difference operators. The starting point of the discussion
in \cite{ss} is the usual characterization of the Lipschitz space
$\Lambda_s$ for any $s\in(0,1]$ in terms of differences
(the first-order for any $s\in(0,1)$ and the second-order for $s=1$).
However, the proof of the equivalence \eqref{1-3} in \cite{ss}
is much  more technical and involved than the standard
proofs of the  difference characterization of many
classical function spaces. It  requires many delicate
pointwise estimates,  establishing  connections  between
differences, wavelet expansions, and Poisson integrals.

Note that  the space
$\mathrm{J}_s(\mathop\mathrm{\,bmo\,})$
coincides with the Triebel--Lizorkin space
$F^s_{\infty,2}$ (see, for instance,
\cite[Section 2.3]{T83}).
The main purpose of this article is to essentially
extend  the results of  \cite{ss} by
establishing similar characterizations of
the distances $\mathop\mathrm{\,dist\,}(f, V)_{\Lambda_s}$
in $\Lambda_s$  for  many different  subspaces
$V$ of $\Lambda_s$ that are not dense in $\Lambda_s$ and
for all $s\in(0,\infty)$ and  $f\in \Lambda_s$. In particular,
we show that Theorem A holds for the full range
of $r>s>0$. The spaces  $V$ considered
in this article  include the space
$\mathrm{J}^s(\mathop\mathrm{\,bmo\,})$ for  all $s\in(0,\infty)$,
the Besov spaces $B_{p,q}^s$  and the
Triebel--Lizorkin spaces $F_{p,q}^s$
for all $s\in(0,\infty)$ and $p, q\in (0, \infty]$,
as well as the Besov-type spaces
and  the Triebel--Lizorkin-type spaces.
The precise definitions of these spaces
are given in   Section \ref{sea8}. We also
refer to  \cite{T20,T83,ysy10,b5,b6,b7,b8,b9}
for various  results related to these spaces.
The precise formulations of our main results
are given  in Section \ref{2222}.
The proofs of these results, including Theorem A, as mentioned
above, strongly depend on
the wavelet and the difference characterizations of the function space under consideration.
It is well known that the wavelet characterization and the difference
characterization of function spaces have been an important topic in
the theory of function spaces for a long time and have found wide applications. We refer, for example,
to \cite{Me92,t06,t08,t10} for the wavelet characterization of Besov and Triebel--Lizorkin spaces,
to \cite{T83,t92,T20} for the difference characterization of Besov  and Triebel--Lizorkin spaces,
to \cite{ysy10,lsuyy12,s011,ht23,syy,FSyy} for the wavelet characterization of Besov-type  and Triebel--Lizorkin-type spaces,
and
to \cite{s08,ysy10,s011,hs20,h22} for the difference characterization of Besov-type  and Triebel--Lizorkin-type spaces;
see also \cite{b5,b6,b7,b8,b9,HS13,HMS16,hl,hms23,hlms}
for some further extensions of  these spaces.
We also point out that our results are
of wide generality and applicability.
In addition to being applicable to the aforementioned examples,
it can certainly be applied to more function spaces.

We conclude this section with  some
conventions on notation. Let $\mathbb{N}:=\{1,2,\ldots\}$
and $\mathbb{Z}_+:=\mathbb{N}\cup\{0\}$.
We use $\mathbf{0}$ to denote the
\emph{origin} of $\mathbb{R}^n$. All
functions on $\mathbb{R}^n$ and all
subsets of $\mathbb{R}^n$ are
assumed to be  Lebesgue measurable.
Given a set $E\subset \mathbb{R}^n$,
we use the notation $|E|$ to denote its
Lebesgue measure in $\mathbb{R}^n$ and
$\mathbf{1}_E$ to denote its characteristic
function.
For any  $E \subset \mathbb{R}^n$ with
$|E| < \infty$ and any
$f \in L^1_{{\mathop\mathrm{\,loc\,}}}$, let
$f_E=\fint_Ef(x)\,dx:=
\frac1{|E|}\int_{E}f(x)\,dx$.
We denote the upper half-space
$\mathbb{R}^n \times (0, \infty)$
by $\mathbb{R}_+^{n+1}$.
For any $x \in \mathbb{R}^n$ and
$r \in (0, \infty)$, let
$B(x, r) := \{y \in
\mathbb{R}^n :|x - y| < r\}.$
We denote by
$\mathbb{B}$ the collection of all
Euclidean balls $B(x,r)$
with $x\in\mathbb{R}^n$ and $r\in(0,\infty)$.
For any $\alpha \in (0, \infty)$,
we write $\alpha B := B(x_B, \alpha r_B)$
for any ball $B := B(x_B, r_B)$.
We use $C$ to  denote a  general
{positive constant} independent of the main
parameters involved, which may vary from
line to line. We also use $C_{(\alpha,\beta,\ldots)}$
to denote a positive constant depending
on indicated parameters $\alpha, \beta, \ldots$.
The symbol $f \lesssim g$ means $f \leq Cg$.
If $f \lesssim g$ and $g \lesssim f$, we
write $f \sim g$. If $f \leq Cg$ and
$g = h$ or $g \leq h$, we write
$f \lesssim g = h$ or $f \lesssim g \leq h$.
The symbol $\mathfrak{C}$ denotes the space
of all continuous functions
on $\mathbb{R}^n$, and the symbols $\mathfrak{C}_0$
denotes the space of all $f\in\mathfrak{C}$
such that $\lim_{|x|\to\infty} f(x)=0$.
Also, for any given set $A$, we use $\sharp A$
to denote its \emph{cardinality}.
Finally, when we prove a theorem (or the like),
in its proof we always use the same symbols as in the
statement itself of that theorem (or the like).

\section{Main results }\label{2222}

By definition,   a function
$f\in \mathfrak{C}\cap L^\infty$
belongs to the space $\Lambda_s$ for some $s\in(0,\infty)$
if and only if  there exist a  constant
$\alpha\in(0,\infty)$ and an integer $r>s$ such that
\begin{equation}\label{2-1b}
\Delta_r
f(x,y):=\sup_{|h|\leq y} |\Delta_h^r f(x)|
\leq \alpha y^s\    \ \text{for any}\  \
x\in\mathbb{R}^n\  \ \text{and}\  \
y\in(0,1],
\end{equation}
where the \emph{smallest constant $\alpha$} as above
is denoted by $\alpha_{r,s}(f)$.
For each $f\in \Lambda_s$ and $\varepsilon \in(0,
\alpha_{r,s}(f))$,  consider the set
$S_{r}(s,f,\varepsilon)$ of ``bad'' points
$(x, y)\in \mathbb{R}^n\times (0, 1]$ at
which \eqref{2-1b} with $\varepsilon$ in place of
$\alpha$  fails; that is,
$$ S_{r}(s,f,\varepsilon):=\left\{ (x,y)\in
\mathbb{R}^n\times (0,1]:
\frac{\Delta_r f(x,y)}
 { y^s}>\varepsilon\right\}.$$
To measure the ``size'' of such a set,
we consider a suitably defined nonnegative
extended real-valued function
$\gamma:\mathcal{B}(\mathbb{R}^n\times (0, 1])
\to [0,\infty]$ on the Borel $\sigma$-algebra
$\mathcal{B}(\mathbb{R}^n\times (0, 1])$ of
subsets of $\mathbb{R}^n\times (0, 1]$.
We call  $\gamma$ an \emph{admissible function}
if it satisfies that $\gamma(\emptyset)=0$ and
$\gamma(E_1)\leq \gamma(E_2)$ whenever $E_1,
E_2\in \mathcal{B}(\mathbb{R}^n\times
(0, 1])$ and $E_1\subset E_2$.

For a given admissible function $\gamma$, we
then define
\[ \varepsilon_{r,s,\gamma}(f):=\inf\{\varepsilon\in(0,\infty):
\gamma(S_{r}(s,f,\varepsilon))<\infty\}.\]
In other words, $\varepsilon_{r,s,\gamma}(f)$ is the
\emph{critical index} for which $\gamma(S_{r}(s, f,\varepsilon))
<\infty$ for any $\varepsilon\in(\varepsilon_{r,s,\gamma}(f),\infty)$.
For convenience, we refer to  $\varepsilon_{r,s,\gamma}(f)$
as the \emph{Lipschitz deviation constant}
of $f$ with constants $r,s$ and admissible
function $\gamma$. Clearly, $0\leq \varepsilon_{r,s,\gamma}(f)
\leq \alpha_{r,s}(f)$.  In general, however,
the constant $\varepsilon_{r,s,\gamma}(f)$ can be
significantly smaller than  the  constant
$\alpha_{r,s}(f)$.

Let $s\in(0,\infty)$ and let $r$ be an given integer
with $r>s$.
We are interested in the following question:
\begin{center}
\begin{minipage}{13.8cm}
\emph{What conditions should be imposed on   a
subspace $V$ of $ \Lambda_s$  to guarantee the
existence of  an admissible function $\gamma$ such
that, for each $f\in \Lambda_s$,  the distance
$\mathrm{dist}\,(f, V)_{\Lambda_s}$ can be characterized
by  the Lipschitz deviation constant
$\varepsilon_{r,s,\gamma}(f)$  of $f$?}
\end{minipage}
\end{center}
We will prove a more general result in a unified
framework, showing that under certain conditions
on a given subspace $V\subset \Lambda_s$, there always
exists  an admissible function $\gamma$ (depending on $V$)
such that
$\varepsilon_{r,s,\gamma}(f)\sim \mathrm{dist}\,(f,  \Lambda_s)_{\Lambda_s}$.

Our general result applies to  the case when $V$ is
the intersection of $\Lambda_s$ with any of the following
classical function spaces:  the space $\mathrm{J}^s(\mathop\mathrm{\,bmo\,})$
for  all $s\in(0,\infty)$,
the Besov spaces $B_{p,q}^s$  and the Triebel--Lizorkin
spaces $F_{p,q}^s$  for all $s\in(0,\infty)$ and $p, q\in (0, \infty]$,
as well as the Besov-type
and  the Triebel--Lizorkin-type spaces.  In particular, our
result extends Theorem A to  higher orders of
smoothness in the case when  $V=J_s(\mathop\mathrm{\,bmo\,})$, where
the admissible function $\gamma$ takes the form
\[\gamma(E):=\sup_{I\in\mathcal D} \frac 1 {|I|} \int_0^{\ell(I)}
\left|E_y\cap I\right|\frac {dy} y<\infty, \  E\subset
\mathbb{R}^n\times [0, 1]
\]
with $E_y:=\{x\in \mathbb{R}^n:(x,y)\in E\}$
for any $y\in [0, 1]$.

To highlight the spirit of our main results,
let us first list
the corresponding results  for the
aforementioned  specific well-known
function spaces. Assume that $s\in(0,\infty)$ and $r>s$ is a
given integer. Let $p\in(1,\infty)$ and $q \in(0,\infty]$.
Then the main result of this article implies the
following ones on the aforementioned classical function spaces
(see Section \ref{sea8} for   precise formulations
and  proofs): Here, all the positive equivalence constants
are independent of $f$.
\begin{itemize}
\item   For the space
$\mathrm{J}_s(\mathop\mathrm{\,bmo\,})$ and $f\in\Lambda_s$,
\begin{equation} \label{1-8b}	
\mathop\mathrm{\,dist\,}(f,
\mathrm{J}_s(\mathop\mathrm{\,bmo\,}))_{\Lambda_s} \sim
\varepsilon_{r,s}(f):=\inf\left\{ \varepsilon\in(0,\infty):	
J_{r,s,\varepsilon}(\mathop\mathrm{\,bmo\,})(f) <\infty\right\},
\end{equation}
where
\begin{equation*}
J_{r,s,\varepsilon}(\mathop\mathrm{\,bmo\,})(f):=
\sup_{I\in\mathcal{D}} \frac 1 {|I|} \int_0^{\ell(I)}
\left|\left\{ x\in I:  \frac{\Delta_r f(x, y)}{y^{s} }
>\varepsilon\right\}\right|\,\frac {dy} y.
\end{equation*}

\item For the Sobolev space $W^{1,p}$
and
$f\in \mathfrak{C}_0\cap
\Lambda_1$,
\begin{align*}
\mathop\mathrm{\,dist\,}(f, W^{1,p})_{\Lambda_s}
\sim
\varepsilon_p(f):=\inf \left\{ \varepsilon\in(0,\infty):\left\| \int_{\left\{y\in (0,1]:\frac
{\Delta_2 f(\cdot, y)}{y}>\varepsilon \right\}}
\frac {dy} y\right\|_{L^{\frac p2}}
<\infty \right\}.
\end{align*}
In particular, this implies
$W^{1,p}\cap\Lambda_s\subsetneqq
\overline{W^{1,p}\cap\Lambda_1}^{\Lambda_1}
\subsetneqq\Lambda_1$.

\item   For  the Besov space $B_{p,q}^s$
 and
$f\in \mathfrak{C}_0\cap
\Lambda_s$,
\begin{align*}
&	
\mathop\mathrm{\,dist\,}
(f, B_{p, q}^s)_{\Lambda_s}
\sim \varepsilon_{r,s,p,q}(f):=
\inf \left\{ \varepsilon\in(0,\infty):B_{p,q,\varepsilon}^s(f) <\infty\right\},
\end{align*}
where
\begin{equation*}
B_{p,q,\varepsilon}^s(f):= \left\|
\left\{  \int_{2^{-j-1}}^{2^{-j}} \left|
\left\{ x\in{\mathbb R}^n:\frac{\Delta_r f(x, y)}{y^{s} }>\varepsilon
\right\} \right|\,\frac {dy} y \right\}_{j=0}^\infty
\right\|_{\ell^{\frac qp}({\mathbb Z}_+)}.
\end{equation*}
\item  For  the Triebel--Lizorkin space
$F_{p,q}^s$ and
$f\in \mathfrak{C}_0\cap
\Lambda_s$,
\begin{align*}
\mathop\mathrm{\,dist\,}(f, F_{p, q}^s)_{\Lambda_s}
\sim \varepsilon_{r,s,p,q}(f):=
\inf \left\{ \varepsilon\in(0,\infty):F_{p, q,\varepsilon}^s(f) <\infty \right\},
\end{align*}
where
\[  F_{p, q,\varepsilon}^s(f):=
\left\| \int_{\left\{y\in [0,1]:\frac{\Delta_r f(x, y)}{y^{s}}>\varepsilon \right\}}
\frac {dy} y\right\|_{L^{\frac pq}}.\]
\end{itemize}

All of the above results can, in fact,
be proved within a unified framework,
where the function space $V$ of smooth
order $s\in(0,\infty)$, called  the
\emph{Daubechies $s$-Lipschitz $X$-based space},
is  defined via the Daubechies wavelet
expansions and  a given  quasi-normed
lattice $X$ of function sequences.

Before proceeding, let us briefly review
the Daubechies wavelet system on
$\mathbb{R}^n$ and related facts.
We denote by $\mathfrak{C}_{\mathrm{c}}^r$
the set of all $r$ times  continuously
differentiable  real-valued  functions on
$\mathbb{R}^n$ with compact support.  For any
$\gamma:=(\gamma_1,\ldots,\gamma_n)\in\mathbb Z_+^n$ and
$x:=(x_1,\ldots,x_n)\in\mathbb{R}^n$, let
$|\gamma|:=\gamma_1+\cdots+\gamma_n$,
$x^\gamma:=x_1^{\gamma_1}\cdots x_n^{\gamma_n}$, and
$\partial^\gamma:=(\frac{\partial}{\partial
x_1})^{\gamma_1}\cdots(\frac{\partial}{\partial x_n})^{\gamma_n}$.
According to \cite[Sections~3.8 and 3.9]{Me92},
for any integer  $L>1$
there exist functions
$\varphi,\psi_\ell\in \mathfrak{C}_{\mathrm{c}}^L$ with
$\ell\in\{1,\ldots, 2^{n}-1\}=:\mathcal{N}$
such that,
for any $\gamma\in\mathbb{Z}_+^n$ with
$|\gamma|\le L$,
\begin{align*}
\int_{\mathbb{R}^n}\psi_{\ell}(x)x^\gamma\,dx=0,\
\forall\, \ell\in \mathcal{N}
\end{align*}
and the union of both
\begin{align}\label{fanofn}
\Phi:=\left\{\varphi_k(x):=\varphi(x-k):k\in \mathbb Z^n\right\}
\end{align}
and $$
\Psi:=\left\{ \psi_{\ell,j,k}(x):=2^{\frac{jn}2}
\psi_\ell (2^jx-k):\ell\in \mathcal{N},\
j\in \mathbb Z_+,\  \ k\in\mathbb Z^n\right\}
$$
forms  an orthonormal basis of $L^2$.
The system  $\Phi\cup \Psi$
is called the \emph{Daubechies wavelet system
of regularity $L$}
on $\mathbb{R}^n$.
We index this system as usual   via dyadic
cubes as follows.
For any $j\in \mathbb{Z}_+$, let $\mathcal{D}_j$
denote the class of all
dyadic cubes in $\mathcal{D}$ with edge lengths
$2^{-j}$.
Let $\Omega:=\Omega_0\cup\Omega_1$, where
$\Omega_0:=\{0\}\times \mathcal{D}_0$ and
$\Omega_1:= \mathcal{N}\times \mathcal D$.
For any $\omega=(0, I)\in\Omega_0$ with
$I=k+[0,1)^n\in \mathcal{D}_0$, let
$\psi_\omega(\cdot):=\varphi(\cdot-k),$
while, for  any  $\omega=(\ell_\omega,
I_\omega)\in\Omega_1$
with $\ell_\omega\in\mathcal{N}$
and $I_\omega=2^{-j} ( k+[0,1)^n) \in\mathcal
D_j$, let
$\psi_{\omega}(\cdot): =2^{\frac{jn}2}
\psi_{\ell_\omega}(2^j \cdot-k).$
For any $\omega\in\Omega$ and
$f\in L^1_{\mathrm{loc}}$, let
$\langle f,\psi_{\omega}\rangle:=
\int_{\mathbb{R}^n} f(x) \psi_\omega(x)\, dx.$
We fix a Daubechies wavelet system
$\Phi\cup \Psi$  with regularity $L$ in the
remainder of this article.

The definition of the function space $V$ in
our framework also requires the concept of
quasi-normed lattice of function sequences,
which we now  introduce.
For any given $i,j\in\mathbb{Z}_+$,
let $\delta_{i,j}:=1$ or $0$ depending on
whether $i=j$ or $i\neq j$.
We denote by $\mathscr M$ the
set of all measurable functions  $f:
\mathbb{R}^n\to \mathbb{R}\cup\{\pm\infty\}$
that are finite almost everywhere  on
$\mathbb{R}^n$ and by  $\mathscr{M}_{\mathbb{Z}_+}$
the set of all sequences $	\{f_j\}_{j\in\mathbb{Z}_+}$
of functions in $\mathscr M$.  We use
boldface letters $\mathbf G, \mathbf F,\ldots$
to denote elements in $\mathscr{M}_{\mathbb{Z}_+}$.
For any
 $\alpha\in \mathbb{C}$
and $\mathbf F:=\{f_j\}_{j\in\mathbb{Z}_+},\mathbf G:=\{g_j\}_{j\in\mathbb{Z}_+}
\in \mathscr{M}_{\mathbb{Z}_+}$, define
$|\mathbf{F}|:=\{| f_j|\}_{j\in\mathbb{Z}_+},\
\alpha \mathbf F:=\{\alpha f_j\}_{j\in\mathbb{Z}_+}$,
and
$\mathbf F\pm \mathbf G:=
\{f_j\pm g_j
\}_{j\in\mathbb{Z}_+}.$
Also, we write $\mathbf{F}\leq \mathbf {G}$
if, for  any $j\in\mathbb{Z}_+$,
$f_j\leq g_j$  almost everywhere on $\mathbb{R}^n$.
We define the \emph{left shift $S_L$} and
the \emph{right shift} $S_R$ on the space
$\mathscr{M}_{\mathbb{Z}_+}$
as follows: for any $\mathbf{F}:
=\{f_j\}_{j\in\mathbb{Z}_+}
\in\mathscr{M}_{\mathbb{Z}_+}$,
$
S_L\left(\mathbf{F}\right)
:=\{f_{j+1}\}_{j\in\mathbb{Z}_+}
\ \mathrm{and}\
S_R\left(\mathbf{F}\right)
:=\{f_{j-1}\}_{j\in\mathbb{Z}_+}, \
\text{where}\  f_{-1}:=0.
$

\begin{definition}\label{Debqfs}
Let $\|\cdot\|_X:
\mathscr{M}_{\mathbb{Z}_+}\to [0,\infty]$ be
an extended-valued  quasi-norm
defined on the entire space
$\mathscr{M}_{\mathbb{Z}_+}$ satisfying
that  $\|\mathbf F\|_X\leq \|\mathbf{G}\|_X$
whenever $\mathbf{F},\mathbf{G}\in
\mathscr{M}_{\mathbb{Z}_+}$ and
$|\mathbf{F}|\leq |\mathbf{G}|$.
Let $X$ denote the  space of all function
sequences $ \mathbf{F}\in
\mathscr{M}_{\mathbb{Z}_+}$ such that
$\|\mathbf{F}\|_X<\infty$.
We call
$X=(X,\|\cdot\|_X)$  a \emph{quasi-normed
lattice
of function
sequences} if it satisfies   the following
two conditions:
\begin{itemize}
\item[\rm (i)]
There exists a positive constant $C$ such
that,
for any $\mathbf{F}\in X$,
$\|	S_L(\mathbf{F})\|_X+\|S_R(\mathbf{F})
\|_X\leq C\|\mathbf{F}\|_X. $
\item[\rm (ii)]
For any bounded function $f:\mathbb{R}^n\to
\mathbb{R}$ with compact support,
$\{f, 0,\ldots\}\in X$.
\end{itemize}
\end{definition}

The function space with smooth order $s\in(0,\infty)$ in
our general framework is called the
Daubechies $s$-Lipschitz $X$-based space.
It is defined below through the Daubechies
wavelet expansions and a given quasi-normed
lattice $X$ of function sequences.

\begin{definition}\label{asdfs}
Let $X$ be a quasi-normed lattice
of function sequences. Let $s\in(0,\infty)$. Assume that
the regularity $L\in\mathbb{N}$ of the \emph{Daubechies wavelet
 system}  is strictly  bigger than $s$. Then the
\emph{Daubechies $s$-Lipschitz $X$-based space}
$\Lambda_X^{s}$ is defined
to be the set of
all functions $f\in \Lambda_{s}$ such that
$$
\|f\|_{\Lambda_X^{s}}:=\left\|\left\{
\sum_{\{\omega\in \Omega:I_\omega\in\mathcal{D}_j\}}
|I_\omega|^{-\frac sn-\frac12}\left|\langle f,\psi_{\omega}\rangle\right|\mathbf{1}_{I_\omega}
\right\}_{j\in\mathbb{Z}_+}
\right\|_{X}<\infty.
$$
\end{definition}

It is easily seen that
many classical function spaces, including those
aforementioned spaces, can be considered  as the
\emph{Daubechies s-Lipschitz $X$-based spaces}
associated with  some  quasi-normed lattice
$X$ of function sequences, according to
the usual  wavelet characterizations of these spaces.

To formulate our first result, we
assume that the  given quasi-normed
lattice $X$ of function sequences
satisfies the following condition
for some positive constant $u$.

\begin{assumption}\label{a1}
Let $X$ be a quasi-normed lattice
of function sequences and $u\in(0,\infty)$.
Assume that,
for any $\beta\in(1,\infty)$,
there exists a positive constant $C$,
which may depend on $u$ and $\beta$,
such that,
for any sequence $\{B_{k,j}\}_{k\in
\mathbb N,j\in\mathbb Z_+}$ in $\mathbb{B}$
of Euclidean balls in $\mathbb{R}^n$, one has
\begin{align*}
\left\|\left\{ \left|\sum_{k\in\mathbb N}
\mathbf{1}_{\beta B_{k,j}}\right|^u
\right\}_{j\in\mathbb Z_+}\right
\|_{X}^{\frac 1u}
\le C\left\|\left\{ \left|
\sum_{k\in\mathbb N}\mathbf{1}_{B_{k,j}}\right|^u
\right\}_{j\in\mathbb Z_+}\right\|_{X}^{\frac 1u}.
\end{align*}
\end{assumption}

With Assumption I, we can define the
Lipschitz deviation constant
associated with $X$  as follows.

\begin{definition}\label{asdfs34}
Let $X$ be a quasi-normed lattice
of function
sequences satisfying Assumption \ref{a1}
for some $u\in(0,\infty)$.
Assume that  $s\in(0,\infty)$ and
$r\in\mathbb N\cap(s,\infty)$.
Given $f\in\Lambda_s$, define the
\emph{Lipschitz deviation constant} $\varepsilon_Xf$ of $f$
associated with $X$ by
\begin{equation}\label{ef}
\varepsilon_Xf:=
\inf\left\{\varepsilon\in(0,\infty):\left\|\left\{ \left|\int_{0}^1
\mathbf{1}_{S_{r,j}(s,f, \varepsilon)}
(\cdot,y)\,\frac{dy}{y}
\right|^u\right\}_{j\in\mathbb Z_+}
\right\|_{X}<\infty\right\},
\end{equation}
where, for any $j\in\mathbb{Z}_+$,
\begin{align}\label{dfmawlmfp}
S_{r,j}(s,f, \varepsilon):&=\left\{
(x, y)\in\mathbb{R}_+^{n+1}:\frac{\Delta_r f(x,y)}
 { y^s}>\varepsilon,\
y\in(2^{-j-1},2^{-j}]\right\},
\end{align}
\end{definition}

With the above definitions, we can now
state our first result as follows.

\begin{theorem}\label{pppz23}
Let $s\in(0,\infty)$ and  $r\in
\mathbb N\cap(s,\infty)$. Assume
the regularity $L\in\mathbb{N}$ of the Daubechies wavelet
system $\{\psi_\omega\}_{\omega\in\Omega}$ satisfies
$L>s$ and $L\geq r-1$.
Let $X$ be a quasi-normed lattice
of function
sequences satisfying Assumption \ref{a1}
for some $u\in(0,\infty)$.
Then,
for any $f\in\Lambda_s$,
\begin{align}\label{pofwj}
\mathop\mathrm{\,dist\,}\left(f,
\Lambda_X^{s}\right)_{\Lambda_s}
\sim \varepsilon_Xf+
\inf\left\{\varepsilon\in(0,\infty):\left\|\left\{
\sum_{I\in V_0(s, f, \varepsilon)}
\mathbf{1}_{I}\delta_{0,j}
\right\}_{j\in\mathbb{Z}_+}
\right\|_{X}<\infty\right\},
\end{align}
where \begin{align*}
V_0(s, f, \varepsilon):=
\left\{ I\in\mathcal D_0:\left|\int_{\mathbb{R}^n} f(x)
\psi_{(0,I)}(x)\, dx\right|
>\varepsilon \right\}
\end{align*}
and the positive equivalence constants
are independent of $f$.
\end{theorem}

The proof of Theorem \ref{pppz23}
is given in Subsection \ref{sec:3-3}.

It should be noted  that Assumption
\ref{a1} may not hold in  certain
endpoint cases, such as  when $\Lambda_X^s$
is  the Triebel--Lizorkin space
$F^s_{\infty,q}$ or the Besov space
$B^s_{\infty,q}$ (see Definitions
\ref{df-Triebel} and \ref{sanfnlk}).
Theorem \ref{pppz23} can not be applied
to these endpoint cases.  To address
these endpoint cases, we need a
different theorem, namely Theorem
\ref{thm-7-11} below.

The formulation of Theorem \ref{thm-7-11}
requires an additional assumption and some
additional notation, which we  now introduce.
For each  $I\in \mathcal D$, let
$T(I): =I \times (2^{-1} \ell(I), \ell(I)].$
We denote by  $\mathscr P({\mathbb R}^{n+1}_+)$
the
collection  of all measurable subsets of
${\mathbb R}^{n+1}_+$.
Let
\begin{align*}
\mathscr{P}_{\mathbb{Z}_+}(\mathbb{R}^{n+1}_+)
:=
\left\{
\{A_j\}_{j\in\mathbb{Z}_+}:\forall\, j\in\mathbb{Z}_+,\
A_j\in\mathscr P({\mathbb R}^{n+1}_+)\right\}.
\end{align*}For any $\{A_j\}_{j\in\mathbb{Z}_+},
\{B_j\}_{j\in\mathbb{Z}_+}\in
\mathscr{P}_{\mathbb{Z}_+}(\mathbb{R}^{n+1}_+)$,
if $A_j\subset B_j$
for any $j\in\mathbb{Z}_+$,
we  write
$\{A_j\}_{j\in\mathbb{Z}_+}\subset
\{B_j\}_{j\in\mathbb{Z}_+}.$

\begin{definition}
The \emph{left shift $S_L$} and
the \emph{right shift} $S_R$ on
$\mathscr{P}_{\mathbb{Z}_+}
(\mathbb{R}^{n+1}_+)$
are defined, respectively, by setting, for any
$\mathbf{F}:=\{A_j\}_{j\in\mathbb{Z}_+}
\in\mathscr{P}_{\mathbb{Z}_+}
(\mathbb{R}^{n+1}_+)$,
$S_L(\mathbf{F})
:=\{A_{j+1}\}_{j\in\mathbb{Z}_+}$ and $
S_R(\mathbf{F})
:=\{A_{j-1}\}_{j\in\mathbb{Z}_+},$
where $A_{-1}:=\emptyset$.
\end{definition}

We also  use the \emph{Poincar\'e
hyperbolic metric}  $\rho:\mathbb{R}_+^{n+1}\times \mathbb{R}_+^{n+1}
\to \mathbb{R}$ in the upper-half space
$\mathbb{R}_+^{n+1}$, which is defined
by setting, for any $\mathbf{x}:=( x, x_{n+1})$,
$\mathbf{y}:=( y, y_{n+1}) \in \mathbb{R}_+^{n+1}$,
\begin{align}\label{metric}
\rho(\mathbf{x}, \mathbf{y}) &:
=\mathop\mathrm{\,arccosh\,}
\left( 1+\frac{|\mathbf{x}-\mathbf{y}|^2}{2x_{n+1} y_{n+1}}\right),
\end{align}
For  convenience, we collect
some useful properties of this metric
in Appendix.
Given $R\in(0,\infty)$, we define
the \emph{hyperbolic $R$-neighborhood} $A_R$
of a subset $A\subset \mathbb{R}_+^{n+1}$
by setting
$$A_R:=\{\mathbf x\in\mathbb{R}^{n+1}_+:
\rho(\mathbf x, A)< R\},$$
where
$\rho\left(\textbf{x}, A\right):=
\inf_{\textbf{y}\in A} \rho\left(\textbf{x}, \textbf{y}\right)$.

\begin{definition}\label{Debqf2s}
A map
$\nu:\mathscr{P}_{\mathbb{Z}_+}
(\mathbb{R}^{n+1}_+)\rightarrow [0,\infty]$
is called a \emph{Carleson-type measure}
if it satisfies the following conditions:
\begin{itemize}
\item[\rm(i)]
For any $\{A_j\}_{j\in\mathbb{Z}_+},
\{B_j\}_{j\in\mathbb{Z}_+}\in
\mathscr{P}_{\mathbb{Z}_+}
(\mathbb{R}^{n+1}_+)$,
$\{A_j\}_{j\in\mathbb{Z}_+}\subset
\{B_j\}_{j\in\mathbb{Z}_+}$ implies
that $\nu(\{A_j\}_{j\in\mathbb{Z}_+})
\le\nu(\{B_j\}_{j\in\mathbb{Z}_+})$.
\item[\rm(ii)]
There exists a positive constant $C$
such that,
for any $\{A_j\}_{j\in\mathbb{Z}_+},
\{B_j\}_{j\in\mathbb{Z}_+}\in
\mathscr{P}_{\mathbb{Z}_+}(\mathbb{R}^{n+1}_+)$,
$$\nu(\{A_j\cup B_j\}_{j\in\mathbb{Z}_+})
\le C\left[\nu(\{A_j\}_{j\in\mathbb{Z}_+})
+\nu(\{B_j\}_{j\in\mathbb{Z}_+})\right].$$
\item[\rm(iii)]
For any $j\in\mathbb Z_+$, let $A_j\subset
\mathbb{R}^{n+1}_+$ be a measurable set
such that there  exist constants $\delta,
\delta'\in (0, \frac{1}{10})$, independent of $j$,
such that,
for any $\mathbf{z}\in A_j$,
\begin{equation}\label{5-1-00}
|B_\rho(\mathbf{z},\delta) \cap A_j |\ge
\delta' |B_\rho(\mathbf{z},\delta)|.
\end{equation}
Then,
for any sequence $\{A_j\}_{j\in\mathbb
Z_+}$ as above,
$\nu(\{(A_j)_R\}_{j\in\mathbb Z_+})$ and
$\nu(\{A_j\}_{j\in\mathbb Z_+})$
are finite or infinite simultaneously.
\item[\rm(iv)]
There exists a positive constant $C$ such
that,
for any $\{A_j\}_{j\in\mathbb{Z}_+}\in
\mathscr{P}_{\mathbb{Z}_+}(\mathbb{R}^{n+1}_+)$,
$\nu(S_L\{A_j\}_{j\in\mathbb{Z}_+})
\le\nu(\{A_j\}_{j\in\mathbb{Z}_+})$
and $\nu(S_R\{A_j\}_{j\in\mathbb{Z}_+})
\le\nu(\{A_j\}_{j\in\mathbb{Z}_+})$.
\end{itemize}
\end{definition}
Given a function
$\nu:\mathscr{P}_{\mathbb{Z}_+}
(\mathbb{R}^{n+1}_+)\rightarrow
[0,\infty]$, we also define,
for any $A\in \mathscr
P({\mathbb R}^{n+1}_+)$, $$\nu(A)
:=\nu\left(\left\{A
\right\}_{j=0}\cup\left\{
\emptyset
\right\}_{j\in\mathbb N}\right),$$
here and thereafter, $\{A
\}_{j=0}\cup\{
\emptyset
\}_{j\in\mathbb N}$ means the
sequence of sets in $\mathscr
P({\mathbb R}^{n+1}_+)$ whose first entry
is $A$ and whose other entries are all the empty set.

We make the following additional
assumption on the quasi-normed lattice $X$.

\begin{assumption}\label{pplp}
Let $X$ be a quasi-normed lattice
of function
sequences. Assume that
there exists a Carleson-type measure
$\nu$ such that,
for any sequence $A$ in  $\mathcal D$,
\begin{align*}
\left\|\left\{
\sum_{I\in A\cap \mathcal D_j}
\mathbf{1}_{I}
\right\}_{j\in\mathbb{Z}_+}
\right\|_{X}\ \
\text{and}\ \ \nu\left(\left\{
\bigcup_{I\in A\cap \mathcal D_j}
T(I)\right\}_{j\in\mathbb{Z}_+}\right)
\end{align*}
are finite or infinite simultaneously.
\end{assumption}

It can be easily verified that
Assumption \ref{pplp} holds
in the endpoint case where
$\Lambda_X^{s}:=F^s_{\infty,q}$
with $s\in(0,\infty)$ and $q\in(0,\infty]$
or where $\Lambda_X^{s}:=B^s_{\infty,q}$
with $s\in(0,\infty)$
and $q\in(0,\infty]$.
\begin{example}
\begin{itemize}
\item[\rm (i)]
Let $\Lambda_X^{s}:=F^s_{\infty,q}$
with $s\in(0,\infty)$ and $q\in(0,\infty]$.
For any $\{A_j\}_{j\in\mathbb{Z}_+}
\in\mathscr{P}_{\mathbb{Z}_+}
(\mathbb{R}^{n+1}_+)$, let
$\nu(\{A_j\}_{j\in\mathbb Z_+})
:=M(\bigcup_{j\in\mathbb Z_+}A_j)$,
where
$F^s_{\infty,q}$ is  the Triebel--Lizorkin
space defined
in Definition \ref{sanfnlk}(ii)
and $M$ is the same as in \eqref{asdf}.
Then, by Lemma \ref{as12312},
we find that in this case Assumption
\ref{pplp} holds.
\item[\rm (ii)] Let $\Lambda_X^{s}
:=B^s_{\infty,q}$ with $s\in(0,\infty)$
and $q\in(0,\infty]$. For any
$\{A_j\}_{j\in\mathbb{Z}_+}
\in\mathscr{P}_{\mathbb{Z}_+}
(\mathbb{R}^{n+1}_+)$, let
$\nu(\{A_j\}_{j\in\mathbb Z_+})
:=\sharp\{j\in\mathbb Z_+:A_j\neq\emptyset\}$,
where
$B^s_{\infty,q}$ is  the Besov space
defined in Definition \ref{df-Triebel}.
Then, by Lemma \ref{as12k312},
we find that in this case Assumption
\ref{pplp} also holds.
\end{itemize}
\end{example}

For any $A\subset \mathscr
P({\mathbb R}^{n+1}_+)$,
let $\nu(A):=\nu(\left\{A
\right\}_{j=0}\cup\left\{
\emptyset
\right\}_{j\in\mathbb N})$.

\begin{definition}\label{asdfs345}
Let $X$ be a quasi-normed lattice
of function
sequences satisfying Assumption
\ref{pplp} for some  Carleson-type
measure $\nu:
\mathscr{P}_{\mathbb{Z}_+}
(\mathbb{R}^{n+1}_+)\rightarrow
[0,\infty]$.
Given  $s\in(0,\infty)$ and an integer $r>s$,
we define the  \emph{Lipschitz
deviation constant} $\varepsilon_{X,\nu}f$ of $f\in
\Lambda_s$  associated with $X$ and $\nu$ by setting
\begin{equation*}
\varepsilon_{X,\nu}f
:=\inf\left\{\varepsilon\in
(0,\infty):\nu\left(\left\{S_{r,j}(s,f,
\varepsilon)\right
\}_{j\in\mathbb Z_+}\right)<\infty\right\},
\end{equation*}
where $S_{r,j}(s,f, \varepsilon)$
is as in \eqref{dfmawlmfp}.
\end{definition}

\begin{theorem}\label{thm-7-11}
Let $X$ be a quasi-normed lattice
of function
sequences satisfying Assumption
\ref{pplp} for some  Carleson-type
measure $\nu:
\mathscr{P}_{\mathbb{Z}_+}
(\mathbb{R}^{n+1}_+)\rightarrow [0,\infty]$.
Let $s\in(0,\infty)$ and let $r>s$ be an integer.
Assume that
the regularity $L\in\mathbb N$ of the Daubechies wavelet
system $\{\psi_\omega\}_{\omega\in\Omega}$
satisfies that
$L>s$ and $L\geq r-1$. Let $\Lambda_X^{s}$ be
the Daubechies $s$-Lipschitz $X$-based space.
Then, for any $f\in\Lambda_s$,
 \begin{align}\label{mop}
\mathop\mathrm{\,dist\,}
\left(f, \Lambda_X^{s}
\right)_{\Lambda_s}&\sim
\varepsilon_{X,\nu}f+
\inf\left\{\varepsilon\in(0,\infty):\nu\left(\bigcup_{I\in V_0(s, f,
\varepsilon)}
T(I)\right)<\infty\right\}
\end{align}
with the positive equivalence constants
independent of $f$, where
$S_{r,j}(s,f, \varepsilon)$ is as in \eqref{dfmawlmfp}.
\end{theorem}

The proof of Theorem \ref{thm-7-11}
is given in
Subsection \ref{sec:3-4}.

The remainder of this article is
organized as follows.

In Section \ref{oijjl}, we establish
the wavelet characterization
of the distance
$\mathop\mathrm{\,dist\,}
(f, \Lambda_X^{s})_{\Lambda_s}$
from any function $f \in \Lambda_s$
to the subspace $\Lambda_X^{s}$.
In Section \ref{s3}, we use differences
to characterize
the distance from any function in
$\Lambda_s$ to $\Lambda_X^{s}$. As
applications,
we also characterize the closure $\overline{\Lambda_X^s}^{\Lambda_s}$ of
the Daubechies $s$-Lipschitz $X$-based space
$\Lambda_X^{s}$ in $\Lambda_s$
by using wavelet coefficients and
differences, respectively.
Additionally, we explore the connection
between the hyperbolic metric
and the Carleson measure via the
difference and discuss
the relations between wavelets and
differences
(see Theorems \ref{cor-6-2-0} and
\ref{cor-6-7-0}).
In Section \ref{sea8}, we apply the
results from Section \ref{2222}
to Besov spaces, Triebel--Lizorkin
spaces, Besov-type spaces,
and Triebel--Lizorkin-type spaces,
yielding entirely new results
for these function spaces. Finally,
in Appendix, we include some
basic properties of hyperbolic
metrics related to this article,
which may also be useful for other
applications in function spaces.

\section{Wavelet characterizations}\label{oijjl}

The main purpose of this section
is to characterize the distance
from  a function of the Lipschitz space
$\Lambda_s$ to its
subspace $\Lambda_X^s$
via wavelet coefficients and give its application to characterizing
the closure  $\overline{\Lambda_X^s}^{\Lambda_s}$ of
the Daubechies $s$-Lipschitz $X$-based space
$\Lambda_X^{s}$ in $\Lambda_s$
by wavelets.
For the formulation of our main
result in this section,
Theorem \ref{thm-2-6}, we need
to introduce some additional notation.

Let $Y$ be a quasi-normed space
equipped with a quasi-norm $\|\cdot\|_Y$
and let $Z\subset Y$.
Then, for any $f\in Y$, let
\begin{align*}
\mathop\mathrm{\,dist\,}\left(f, Z\right)_{Y}:=&
\inf_{g\in Z} \|f-g\|_{Y}.
\end{align*}

The following conclusion is the main result
of this section.

\begin{theorem}\label{thm-2-6}
Let $X$ be a quasi-normed lattice
of function
sequences and  $s\in(0,\infty)$.
Assume that
the regularity parameter $L\in\mathbb N$ of
the Daubechies wavelet system
$\{\psi_\omega\}_{\omega\in\Omega}$
satisfies that
$L>s$. Let $\Lambda_X^{s}$ be
the Daubechies $s$-Lipschitz $X$-based space.
Then, for any $f\in\Lambda_s$,
\begin{align}\label{afpmfp}
\mathop\mathrm{\,dist\,}\left(f,
\Lambda_X^{s}\right)_{\Lambda_s}\sim
\inf\left\{\varepsilon\in(0,\infty):\left\|\left\{
\sum_{I\in W_j^0(s, f, \varepsilon)}
\mathbf{1}_{I}
\right\}_{j\in\mathbb{Z}_+}
\right\|_{X}<\infty\right\}
\end{align}
with positive equivalence constants
independent of $f$, where, for any
$j\in\mathbb{Z}_+$,
\begin{align*}
W^0_j(s, f, \varepsilon):=
\left\{ I\in\mathcal D_j:\max_{\{\omega\in\Omega:I_{\omega} =I\}}
\left|\int_{\mathbb{R}^n} f(x) \psi_{\omega}(x)\,
dx\right|
>\varepsilon |I|^{\frac sn+\frac 12}\right\}.
\end{align*}
\end{theorem}

\begin{proof}
For any $\varepsilon\in(0,\infty)$, let
$
G_\varepsilon(f):=
\sum_{\omega\in \Omega(\varepsilon)}\langle f,\psi_{\omega}\rangle\psi_{\omega},
$
where $\Omega(\varepsilon)
:=\{\omega\in\Omega:
I_{\omega}\in W^0(s,f,\varepsilon)\}$.
Then, by Lemma \ref{asqw}, we have
$G_\varepsilon(f)\in \Lambda_s$
and
\begin{align}\label{aaa25}
\|f-G_\varepsilon(f)\|_{\Lambda_s}
\lesssim\sup_{\omega\in\Omega\setminus
\Omega(\varepsilon)}|\langle f,\psi_{\omega}
\rangle||I_{\omega}|^{-\frac sn-\frac 12}
\le\varepsilon.
\end{align}
We use $d_0$ to denote the
right-hand side of $\eqref{afpmfp}$.
We first show that
\begin{align}\label{zxczz2}
\inf_{g\in \Lambda_X^{s}} \|f-g\|_{\Lambda_s}
\lesssim
d_0.
\end{align}
Assume that $\varepsilon\in(0,\infty)$
satisfies that
$\|\{
\sum_{I\in W_j^0(s, f, \varepsilon)}
\mathbf{1}_{I}
\}_{j\in\mathbb{Z}_+}
\|_{X}<\infty.$
From this and Lemma \ref{asqw}, it follows
that
\begin{align*}
\|G_\varepsilon(f)\|_{\Lambda_X^{s}}&\sim
\left\|\left\{
\sum_{\{\omega\in\Omega:I_\omega\in W_j^0(s, f, \varepsilon)\}}|I_\omega|^{-\frac{s}{n}-
\frac 12}\left|\langle G_\varepsilon(f),
\psi_{\omega}\rangle\right|
\mathbf{1}_{I_\omega}
\right\}_{j\in\mathbb{Z}_+}
\right\|_{X}\\
&\lesssim\sup_{\omega\in\Omega}\left|
|I_\omega|^{-\frac{s}{n}-\frac 12}
\langle G_\varepsilon(f),\psi_{\omega}
\rangle\right|\left\|\left\{
\sum_{I\in W_j^0(s, f, \varepsilon)}
\mathbf{1}_{I}
\right\}_{j\in\mathbb{Z}_+}
\right\|_{X}<\infty,
\end{align*}
which further implies that $G_\varepsilon(f)
\in \Lambda_X^{s}$.
By this and \eqref{aaa25}, we conclude that,
for any
$\varepsilon\in(0,\infty)$ with
$\|\{
\sum_{I\in W_j^0(s, f, \varepsilon)}
\mathbf{1}_{I}\}_{j\in\mathbb{Z}_+}
\|_{X}<\infty,$
$$\inf_{g\in \Lambda_X^{s}}
\|f-g\|_{\Lambda_s}\le
\|f-G_\varepsilon(f)\|_{\Lambda_s}
\lesssim\varepsilon,
$$
which further implies that \eqref{zxczz2}
holds.
Next, we prove that
$d_0\lesssim\inf_{g\in \Lambda_X^{s}}
\|f-g\|_{\Lambda_s}.$
To this end, by Lemma \ref{asqw},
we conclude that
it suffices to show, for any
$f\in\Lambda_s$,
\begin{align}\label{zxscgh}
d_0
\le
\inf_{g\in \Lambda_X^{s}}
\sup_{\omega\in\Omega}|\langle
f-g,\psi_{\omega}\rangle|
|I_{\omega}|^{-\frac sn-\frac12}.
\end{align}
Assume otherwise that
\begin{align}\label{zxscgh2}
\inf_{g\in \Lambda_X^{s}}
\sup_{\omega\in\Omega}|\langle f-g,
\psi_{\omega}\rangle||I_{\omega}
|^{-\frac sn-\frac12}< d_0.
\end{align}
Let $\varepsilon\in(0,\infty)$ and
$\delta\in(0,\infty)$  be such that
$$\inf_{g\in \Lambda_X^{s}}
\sup_{\omega\in\Omega}|\langle f-g,
\psi_{\omega}\rangle||I_{\omega}|^{-\frac sn-\frac12}
<\varepsilon< \varepsilon+\delta<d_0.$$
Then there exists a function $g\in \Lambda_X^{s}$
such that
\begin{align}\label{zxc3ass}
\sup_{\omega\in\Omega}|\langle f-g,
\psi_{\omega}\rangle||I_{\omega}
|^{-\frac sn-\frac12}<\varepsilon.
\end{align}
From this and Lemma \ref{asqw},
it follows that
\begin{align}\label{zxc4ass}
\delta \left\|\left\{
\sum_{I\in W_j^0(s, g, \delta)}
\mathbf{1}_{I}
\right\}_{j\in\mathbb{Z}_+}
\right\|_{X}\nonumber
&\lesssim
\left\|\left\{
\sum_{\{\omega\in\Omega:I_\omega\in
W_j^0(s, g, \delta)\}}|I_\omega
|^{-\frac{s}{n}-\frac 12}\left|
\langle g,\psi_{\omega}\rangle\right|
\mathbf{1}_{I_\omega}
\right\}_{j\in\mathbb{Z}_+}
\right\|_{X}\\
&\lesssim\; \|g\|_{\Lambda_X^{s}}
<\infty.
\end{align}
On one hand, since $\varepsilon+\delta<d_0$,
we deduce that
$\|\{
\sum_{I\in W_j^0(s, f, \varepsilon+\delta)}
\mathbf{1}_{I}
\}_{j\in\mathbb{Z}_+}
\|_{X}=\infty.$
On the other hand, however, by \eqref{zxc3ass},
we conclude that,
for any $\omega\in\Omega$ with $I_\omega
\in W_0(s, f, \varepsilon+\delta)$,
\begin{align*}
& |\langle g,\psi_{\omega}\rangle|\ge
|\langle f,\psi_{\omega}\rangle|-
\varepsilon |I_\omega|^{\frac sn+\frac 12}>
\delta  |I_\omega|^{\frac sn+\frac 12},
\end{align*}
which further implies that, for any
$j\in\mathbb Z_+$,
$W_j^0(s, f, \varepsilon+\delta)\subset
W_j^0(s, g, \delta).$
From this and \eqref{zxc4ass}, we infer
that
\begin{align*}
\left\|\left\{
\sum_{I\in W_j^0(s, f, \varepsilon+\delta)}
\mathbf{1}_{I}
\right\}_{j\in\mathbb{Z}_+}
\right\|_{X}\leq\left\|\left\{
\sum_{I\in W_j^0(s, g, \delta)}
\mathbf{1}_{I}
\right\}_{j\in\mathbb{Z}_+}
\right\|_{X}
<\infty,
\end{align*}
which further implies $d_0\le\varepsilon+\delta$.
This contradicts $\varepsilon+\delta<d_0$.
Thus, assumption \eqref{zxscgh2} is not
true and hence \eqref{zxscgh} holds.
This, combined with \eqref{zxczz2}, then
finishes the proof of  Theorem \ref{thm-2-6}.
\end{proof}

To apply Theorem \ref{thm-2-6} to Besov-type spaces in
Subsection \ref{fwefm} and to Triebel--Lizorkin-type
spaces in Subsection \ref{fwefm2}, we
also need the following proposition to
separate the Daubechies wavelet
system into the vanishing part and the
non-vanishing part.

\begin{proposition}\label{pp}
Let $X$ be a quasi-normed lattice
of function
sequences and $s\in(0,\infty)$.
Assume that
the regularity parameter $L\in\mathbb N$
of the Daubechies wavelet system
$\{\psi_\omega\}_{\omega\in\Omega}$
satisfies that
$L>s$. Let $\Lambda_X^{s}$ be
the Daubechies $s$-Lipschitz $X$-based
space and $\varphi$
 the same as in \eqref{fanofn}.
Let
$$\mathcal{A}:=\left\{\{a_k
\}_{k\in\mathbb N}\ in\ \mathbb Z^n: \lim_{k\to\infty}|a_k|=\infty,\
 \left\|\left\{
\sum_{k\in\mathbb N}
\mathbf{1}_{I_{0,a_k}}
\right\}_{j\in\mathbb{Z}_+}\right
\|_{X}=\infty
\right\}.$$
Then, for any $f\in\Lambda_s$,
\begin{align*}
\varepsilon_0&=\sup_{\{a_k\}_{k\in\mathbb N}
\in\mathcal{A}}
\liminf_{k\rightarrow\infty}
\left|\int_{\mathbb{R}^n}\varphi(x-a_k)f(x)
\,dx \right|,
\end{align*}
where
$$\varepsilon_0:=\inf\left\{\varepsilon
\in(0,\infty):\left\|\left\{
\sum_{I\in V_0(s, f, \varepsilon)}
\mathbf{1}_{I}\delta_{0,j}
\right\}_{j\in\mathbb{Z}_+}
\right\|_{X}<\infty\right\}.$$
\end{proposition}
\begin{proof}
We first show
\begin{align}\label{fnqo}
\varepsilon_0
\le&\sup_{\{a_k\}_{k\in\mathbb N}
\in\mathcal{A}}
\liminf_{k\rightarrow\infty} \left|
\int_{\mathbb{R}^n}\varphi(x-a_k)f(x)
\,dx \right|=:\varepsilon_1.
\end{align}
Let $\varepsilon\in(0,\varepsilon_1)$.
Then there exists $\{a_k
\}_{k\in\mathbb N}\in\mathcal{A}$
such that
$$
\liminf_{k\rightarrow\infty}
\left|\int_{\mathbb{R}^n}\varphi(x-a_k
)f(x)\,dx \right|>\varepsilon,
$$
which further implies that there exists
$N\in\mathbb N$
such that, for any $k\in[N,\infty)$, $$ \left|\int_{\mathbb{R}^n}\varphi(x-a_k)
f(x)\,dx \right|>\varepsilon.$$
From this, we deduce that
\begin{align*}
\left\|\left\{
\sum_{I\in V_0(s, f, \varepsilon)}
\mathbf{1}_{I}\delta_{0,j}
\right\}_{j\in\mathbb{Z}_+}
\right\|_{X}&\geq\left\|\left\{
\sum_{\{k\in\mathbb N:k\geq N\}}
\mathbf{1}_{I_{0,a_k}}\delta_{0,j}
\right\}_{j\in\mathbb{Z}_+}
\right\|_{X}=\infty,
\end{align*}
which further implies $\varepsilon
\le\varepsilon_0$ and hence
$\varepsilon_1\le\varepsilon_0$.
Thus, \eqref{fnqo} holds.
Now, we prove
\begin{align}\label{fnqo2}
\varepsilon_0
\geq\varepsilon_1.
\end{align}
Let $\varepsilon\in(0,\varepsilon_0)$.
Then $$
\left\|\left\{
\sum_{I\in V_0(s, f, \varepsilon)}
\mathbf{1}_{I}\delta_{0,j}
\right\}_{j\in\mathbb{Z}_+}
\right\|_{X}=\infty.
$$
Let $V_0(s, f, \varepsilon)=:
\{I_{0,a_k}\}_{k\in\mathbb N}$.
Then we have $\{a_k\}_{k\in
\mathbb N}\in\mathcal{A}$,
which further implies that
$$\varepsilon_1\geq
\liminf_{k\rightarrow\infty}
\left|\int_{\mathbb{R}^n}
\varphi(x-a_k)f(x)\,dx \right|\geq\varepsilon$$
and
hence $\varepsilon_1\geq\varepsilon_0$.
Thus, \eqref{fnqo2} holds.
This finishes the proof of Proposition \ref{pp}.
\end{proof}

Applying Theorem \ref{thm-2-6}, we now give the
wavelet characterization of the closure $\overline{\Lambda_X^{s}}^{\Lambda_s}$
of the Daubechies $s$-Lipschitz $X$-based space
$\Lambda_X^{s}$ in $\Lambda_s$.

\begin{theorem}\label{thm-2-622a}
Let $X$ be a quasi-normed lattice
of function
sequences and  $s\in(0,\infty)$.
Assume that
the regularity parameter $L\in\mathbb N$
of the Daubechies wavelet system
$\{\psi_\omega\}_{\omega\in\Omega}$ satisfies that
$L>s$. Let $\Lambda_X^{s}$ be
the Daubechies $s$-Lipschitz $X$-based space.
Then
$f\in\overline{\Lambda_X^{s}}^{\Lambda_s}$
if and only if,  for any $\varepsilon\in(0,\infty)$,
$$
\left\|\left\{
\sum_{I\in W_j^0(s, f, \varepsilon)}
\mathbf{1}_{I}
\right\}_{j\in\mathbb{Z}_+}
\right\|_{X}<\infty.$$
\end{theorem}
\begin{proof}
Observe that $f\in
\overline{\Lambda_X^{s}}^{\Lambda_s}$
if and only if
$f\in \Lambda_s$
and $\mathop\mathrm{\,dist\,}(f,
\Lambda_X^{s})_{\Lambda_s}=0$.
From this and  Theorem \ref{thm-2-6},
it follows that, when $f\in \Lambda_s$,
$\mathop\mathrm{\,dist\,}(f,
\Lambda_X^{s})_{\Lambda_s}=0$
if and only if, for any
$\varepsilon\in(0,\infty)$,
$\|\{\sum_{I\in W_j^0(s, f,
\varepsilon)}
\mathbf{1}_{I}\}_{j\in
\mathbb{Z}_+}\|_{X}<\infty,$
which completes the proof of
Theorem \ref{thm-2-622a}.
\end{proof}

\section{Difference characterizations}\label{s3}

In this section, for any $s\in(0,\infty)$,
we establish the characterization
of the distance from any function of
the Lipschitz space
$\Lambda_s$ to  the
Daubechies $s$-Lipschitz $X$-based space
$\Lambda_X^{s}$  in terms of differences, that is,
we prove Theorems \ref{pppz23} and \ref{thm-7-11}.

Recall that, for any $h\in \mathbb{R}^n$,
the \emph{symmetric difference operator}
$\Delta_h$
is defined
by setting, for any function $f:
\mathbb{R}^n \to \mathbb{C}$ and
$x\in\mathbb{R}^n$,
$
\Delta_h(f)(x) := f(x+\frac h2)-
f(x-\frac h2)
$ and, for any $k\in\mathbb N$,
the \emph{$(k+1)$-order symmetric
difference operator
$\Delta^{k+1}_h$}
is defined by setting, for any
function $f$,
$\Delta^{k+1}_h (f): =
\Delta^k_h(\Delta_h(f)).$
For any $y\in(0,\infty)$ and
$x\in\mathbb{R}^n$, let
\begin{align}\label{nnfpaw}
\Delta_{k} f(x,y): =
\sup_{\{h\in\mathbb{R}^n:|h|=y\}}
\left|\Delta^k_h(f)(x)\right|.
\end{align}
Define the measure
$d\mu (x,x_{n+1}):=
\frac {dx\,dx_{n+1}} {x_{n+1}}$
on $\mathbb{R}_+^{n+1}$.
Assume that $r\in\mathbb N$ and $r>s$.
Then, for any $\varepsilon\in(0,\infty)$,
let
\begin{align*}
S_r(s,f, \varepsilon):&=\left\{
(x, y)\in\mathbb{R}_+^{n+1}:\frac{\Delta_r f(x,y)}
 { y^s}
 >\varepsilon\right\}
\end{align*}
and, for any $j\in\mathbb Z_+$,
\begin{align*}
S_{r,j}(s,f, \varepsilon):&=\left\{
(x, y)\in\mathbb{R}_+^{n+1}:
\frac{\Delta_r f(x,y)}
 { y^s}>\varepsilon,\
y\in\left(2^{-j-1},2^{-j}\right]\right\}.
\end{align*}

When Theorem \ref{pppz23} is applied
to a specific function space,
the function space under consideration
often has some additional nice properties, which
lead to a more concise version of Theorem \ref{pppz23}; indeed,
applying Theorem \ref{pppz23}, we can
easily obtain the following concise
corollary; we omit the details.

\begin{corollary}\label{pp2}
Let $X$ be a quasi-normed lattice
of function
sequences satisfying Assumption \ref{a1}
for some $u\in(0,\infty)$.
Let $s\in(0,\infty)$ and $r\in\mathbb N$
with $r>s$.
Assume that
the regularity parameter $L\in\mathbb N$ of the
Daubechies wavelet system
$\{\psi_\omega\}_{\omega\in\Omega}$ satisfies that
$L>s$ and $L\geq r-1$. Let $\Lambda_X^{s}$ be
the Daubechies $s$-Lipschitz $X$-based space and $\varphi$
be the same as in \eqref{fanofn}.
Assume that $E\subset \mathbb{R}^n$ with $|E|=\infty$
implies that $\|\{
\mathbf{1}_{E}\delta_{0,j}
\}_{j\in\mathbb{Z}_+}\|_X=\infty$.
Then, for any $f\in\Lambda_s$,
\begin{align*}
\mathop\mathrm{\,dist\,}\left(f, \Lambda_X^{s}\right)_{\Lambda_s}&\sim\varepsilon_f+
\limsup_{k\in\mathbb{Z}^n,\;|k|\rightarrow\infty}
\left|\int_{\mathbb{R}^n}\varphi(x-k)f(x)\,dx\right|
\end{align*}
with positive equivalence constants independent of $f$,
where
$\varepsilon_Xf$ is the same as in \eqref{ef}.
\end{corollary}

Before we give the proofs of Theorems \ref{pppz23} an \ref{thm-7-11},
we first present their applications to characterizing
the closure $\overline{\Lambda_X^{s}}^{\Lambda_s}$
of the Daubechies $s$-Lipschitz $X$-based space
$\Lambda_X^{s}$ in $\Lambda_s$ by differences, respectively,
in Theorems \ref{pppz2a} and \ref{thm-7-117}.

\begin{theorem}\label{pppz2a}
Let $X$ be a quasi-normed lattice
of function
sequences satisfying Assumption \ref{a1}
for some $u\in(0,\infty)$.   Assume that
$s\in(0,\infty)$,  $r\in\mathbb N\cap(s,\infty)$, and
the regularity parameter $L\in\mathbb N$ of the
Daubechies wavelet system $\{\psi_\omega\}_{\omega\in\Omega}$
satisfies $L>s$ and $L\geq r-1$. Let
$f\in\Lambda_s$.  Then the following
two statements are
mutually equivalent:
\begin{itemize}
\item[\rm (i)] $f\in\overline{
\Lambda_X^{s}}^{\Lambda_s}$.

\item[\rm (ii)]
For any $\varepsilon\in(0,\infty)$,
$$
\left\|\left\{ \left|\int_{0}^\infty
\mathbf{1}_{S_{r,j}(s,f, \varepsilon)}
(\cdot,y)\,\frac{dy}{y} \right|^u
\right\}_{j\in\mathbb Z_+}
\right\|_{X}^{\frac 1u} +\left\|\left\{
\sum_{I\in V_0(s, f, \varepsilon)}
\mathbf{1}_{I}\delta_{0,j}
\right\}_{j\in\mathbb{Z}_+}
\right\|_{X}
<\infty.
$$
\end{itemize}
\end{theorem}

\begin{proof}
Observe that $f\in\overline{\Lambda_X^{s}
}^{\Lambda_s}$ if and only if
$f\in \Lambda_s$
and $\mathop\mathrm{\,dist\,}(f,
\Lambda_X^{s})_{\Lambda_s}=0.$
From this and  Theorem \ref{pppz23},
it follows that, when $f\in \Lambda_s$,
$\mathop\mathrm{\,dist\,}(f,
\Lambda_X^{s})_{\Lambda_s}=0$
if and only if, for any $\varepsilon\in(0,
\infty)$,
$\|\{\sum_{I\in W_j^0(s, f, \varepsilon)}
\mathbf{1}_{I}\}_{j\in\mathbb{Z}_+}\|_{X}<\infty.$
Thus, (i) $\Longleftrightarrow$ (ii),
which completes the proof of Theorem \ref{pppz2a}.
\end{proof}

\begin{theorem}\label{thm-7-117}
Let $X$ be a quasi-normed lattice
of function
sequences satisfying Assumption \ref{pplp}.
Let $s\in(0,\infty)$ and $r\in\mathbb N$ with
$r>s$.
Assume that
the regularity parameter $L\in\mathbb N$ of
the Daubechies wavelet
system $\{\psi_\omega\}_{\omega\in\Omega}$
satisfies that
$L>s$ and $L\geq r-1$. Let $\Lambda_X^{s}$ be
the Daubechies $s$-Lipschitz $X$-based space.
Then
$f\in\overline{\Lambda_X^{s}}^{\Lambda_s}$
if and only if,  for any $\varepsilon\in(0,
\infty)$,
$$\nu\left(\left\{S_{r,j}(s,f, \varepsilon)
\right\}_{j\in\mathbb Z_+}\right)+
\nu\left(\bigcup_{I\in V_0(s, f, \varepsilon)}
T(I)\right)<\infty.$$
\end{theorem}

\begin{proof}
Repeating the proof of Theorem \ref{pppz2a}
with Theorem \ref{pppz23} replaced
by  Theorem \ref{thm-7-11},
we then obtain the desired conclusion,
which completes the proof of Theorem \ref{thm-7-117}.
\end{proof}

The proofs of Theorems \ref{pppz23} and \ref{thm-7-11}
are given,
respectively,
in
Subsections \ref{sec:3-3} and \ref{sec:3-4}.
To this end, we first establish connections
between the hyperbolic metric and
the difference in
Subsection \ref{sec:3-1} and
then connections between the wavelet and the
difference in Subsection  \ref{sec:3-2},
which play a
key role in the proofs of Theorems \ref{pppz23}
and \ref{thm-7-11}.

\subsection{Connections between the hyperbolic metric and
the difference}\label{sec:3-1}

In this subsection, we establish a connection
between
the hyperbolic metric and the Carleson measure
in terms of differences.
It should be pointed out that Theorem
\ref{thm-3adf} plays
a key role in the proofs of Theorems
\ref{pppz23} and \ref{thm-7-11}.

\begin{theorem}\label{thm-3adf}
Let  $s\in(0,\infty)$ and $r\in\mathbb N$ with
$r>s$. Assume that $f\in\Lambda_s$ and $\varepsilon,\varepsilon_1\in(0,\infty)$ with
$\varepsilon>\varepsilon_1$.
Then there exists a positive constant $\delta$,
depending only on $f,\varepsilon$, and $\varepsilon_1$,
such
that, for any $j\in\mathbb N$,
$$
[S_{r,j}(s, f, \varepsilon)]_\delta\subset
S_{r,j-1}(s, f, \varepsilon_1)\cup
S_{r,j}(s, f, \varepsilon_1)\cup S_{r,j+1}
(s, f, \varepsilon_1).
$$
\end{theorem}

In what follows, we denote by $\mathcal S$ the \emph{space of all Schwartz functions},
equipped with the well-known topology determined by a countable family of
norms.
Let $\eta\in\mathcal S$ satisfy $0\leq\eta\leq1$,
\begin{equation*}
\eta\equiv1\ \text{on}\
\left\{x\in\mathbb{R}^n:|x|\leq1\right\},\
\text{and}\
\eta\equiv0\  \text{on}\ \left\{x\in
\mathbb{R}^n:|x|\geq\frac 32\right\}.
\end{equation*}
For any $x\in\mathbb{R}^n$, let
$
\theta(x):=\eta\left(x\right)-\eta(2x).
$
Then $\mathop\mathrm{\,supp\,}\theta\subset
\left\{x\in\mathbb{R}^n:1/2 \leq|x|\leq
3/2\right\}$ and, for any $x\in\mathbb{R}^n$,
\begin{align*}
1=\lim_{k\to \infty} \eta\left(\frac x {2^k}\right)
=\eta(x) +\sum_{j=1}^\infty \left[\eta
\left(\frac x {2^j}\right) -\eta
\left(\frac x {2^{j-1}}\right)\right]
=\eta(x) +\sum_{j=1}^\infty \theta
\left(\frac x {2^j}\right).
\end{align*}

In what follows, for any $\phi\in\mathcal{S} $,
we use $\mathcal{F} \phi$ (or $\widehat{\phi}$)
and $\mathcal{F}^{-1}\phi$ to denote its
\emph{Fourier transform} and its \emph{inverse Fourier transform},
respectively.
Recall that, for any $\phi\in\mathcal{S} $
and $\xi\in\mathbb{R}^n$,
\begin{align*}
\widehat{\phi}(\xi):=\frac 1{(2\pi)^n}\int_{\mathbb{R}^n}
f(x)e^{-ix\cdot \xi}\,dx
\end{align*}
and
$\mathcal{F}^{-1}\phi(\xi):=\mathcal{F}\phi(-\xi)$.
Define $\{\phi_k\}_{k\in\mathbb N}$ by
setting, for any $x\in\mathbb{R}^n$,
\begin{equation}\label{eq-phi1}
\phi_0(x):=\mathcal F^{-1}\eta(x)
\end{equation}
and, for any $k\in\mathbb N$,
\begin{equation}\label{eq-phik}
\phi_k(x):=2^{kn}\mathcal F^{-1}\theta(2^kx).
\end{equation}
Clearly, $\sum_{k\in\mathbb Z_+}
\widehat\phi_k=1$ on $\mathbb{R}^n$.
To prove Theorem \ref{thm-3adf}, we
first need the following conclusion
about difference operators.

\begin{lemma}\label{lem-4-5-0}
Let $s\in(0,\infty)$, $r\in\mathbb N$ with $r>s$,
$y, y'\in (0, 1]$, and
$x, x'\in\mathbb{R}^n$. Assume in
addition that $|x-x'|\leq \max\{y,y'\} $
if $s\in[1,\infty)$.
Then there exists a positive constant
$C_{(r)}$, depending on $r$, such
that, for any $f\in\Lambda_s$,
\begin{align*}
&\left|\Delta_rf(x, y) -\Delta_r f(x', y')
\right|
\\\nonumber
&\quad\leq C_{(r)} \|f\|_{\Lambda_s}\left\{
\begin{aligned}
& |x-x'|^s+|y-y'|^s\hspace{5.6cm}\text{if
\ $s\in(0,1)$,}\\
&|x-x'| \log \left( e+\frac {y+y'} {|x-x'|}
\right)+ |y-y'|\log \left( e+\frac {y+y'} {|y-y'|}\right)\hspace{.5cm}\text{if\ $s=1$,}\\
&(y+y')^{s-1}( |y-y'|+|x-x'|)\hspace{4cm}
\text{if\ $s\in(1,\infty)$.}
\end{aligned}
\right.
\end{align*}
\end{lemma}

\begin{proof}
Observe that, for any $y, y'\in (0, 1]$
and $x, x'\in\mathbb{R}^n$,
\begin{align*}
&\left|\Delta_rf(x,y)-\Delta_r f(x', y')
\right|\\
&\quad=\left|\sup_{\xi\in\mathbb{S}^{n-1}}
\left|\Delta_{y\xi}^r f(x)\right|-\sup_{\xi\in\mathbb{S}^{n-1}} \left|\Delta_{y'\xi}^r f(x')\right|\right|
\leq \sup_{\xi\in\mathbb{S}^{n-1}}
\left|\Delta_{y\xi}^r f(x)-\Delta_{y'\xi}^r f(x')\right|\\
&\quad\leq \sup_{\xi\in\mathbb{S}^{n-1}}
\left|\Delta_{y\xi}^r f(x)-\Delta_{y\xi}^r f(x')\right|+
\sup_{\xi\in\mathbb{S}^{n-1}}
\left|\Delta_{y\xi}^r f(x')-\Delta_{y'\xi}^r
f(x')\right|,
\end{align*}
here and thereafter, $\mathbb{S}^{n-1}$
denotes the \emph{unit sphere}
of $\mathbb{R}^n$.
Let
$$B_1:= \left|\Delta_{y\xi}^r f(x)
-\Delta_{y\xi}^rf(x')\right|\ \text{and}\
B_2:=\left| \Delta_{y\xi}^r f(x')
-\Delta_{y'\xi} ^r f(x')\right|.$$
Without loss of generality, we may
assume that $\|f\|_{\Lambda_s}=1$.
To show the present lemma,
we consider the following two cases
for $s$.

If $s\in(0,1)$,
then
$$
B_1\lesssim\sup_{\{h\in\mathbb{R}^n:|h|=y\}}\max_{\{j\in\mathbb{Z}_+:0\leq
j\leq r\}} \left|f\left(x+\left[
\frac r2-j\right]h\right)-
f\left(x'+\left[\frac r2-j\right]h\right)
\right| \lesssim |x-x'|^s$$
and
$$
B_2\lesssim\sup_{ \xi\in\mathbb{S}^{n-1} }
\max_{\{j\in\mathbb{Z}_+:0\leq j\leq r\}}
\left|f\left(x'+\left[\frac r2-j\right]y\xi\right)
-f\left(x'+\left[\frac r2-j\right]y'\xi
\right)\right|\lesssim|y-y'|^s,
$$
which further imply the desired estimate
in this case.

If $s\in[1,\infty)$ and $|x-x'|\leq y$
(which we may assume by symmetry,
without loss of generality),
let $\xi\in\mathbb{S}^{n-1}$, $h:=y\xi$,
and $\phi_j$ be the same as in \eqref{eq-phi1}
and \eqref{eq-phik}.
For $B_1$, using \cite[Theorem 1.4.9]{g14},
we obtain
\begin{align*}
B_1&=\left|\Delta_h^r f(x)-\Delta_h^rf(x')
\right|\leq   \sum_{j=0}^\infty \left|\Delta_h^r
(\phi_j\ast f)(x)-\Delta_h^r(\phi_j\ast f )(x')
\right|=:\sum_{j=0}^\infty d_j.
\end{align*}
We claim that, for any $j\in\mathbb Z_+$ and
$\alpha\in\mathbb Z_+^n$,
\begin{align}\label{ljoi}
\left\|\partial^{\alpha}[\phi_j\ast f]
\right\|_{L^\infty}\lesssim
\|f\|_{\Lambda_s}2^{j(|\alpha|-s)}.
\end{align}
Let $\Phi\in\mathcal{S}$ be such that
$\mathbf{1}_{B(\mathbf{0},1) }\le
\mathcal F\Phi\le\mathbf{1
}_{B(\mathbf{0},2) }.$
Then we conclude that
$\mathcal F{\phi_0}\mathcal F{\Phi}
=\mathcal F{\phi_0}$,
which further implies that,
for any $j\in\mathbb Z_+$,
$\phi_j\ast\Phi_j=\phi_j$, where
$\Phi_j(\cdot):=2^{jn}\Phi(2^j\cdot)$.
From this and \cite[Theorem 1.4.9]{g14},
we deduce that, for any $x\in\mathbb{R}^n$
\begin{align}\label{lnjdf2}
\left|\partial^{\alpha}[\phi_j\ast f](x)
\right|&=
\left|\partial^{\alpha}[\phi_j\ast\Phi_j
\ast f](x)\right|
=\int_{\mathbb{R}^n}|\phi_j\ast f(x)|
|\partial^{\alpha}\Phi_j(x-y)|\,dy
\nonumber\\
&\lesssim\|f\|_{\Lambda_s}2^{j(|\alpha|-s)}
\int_{\mathbb{R}^n}|2^{jn}\partial^{\alpha}
\Phi(2^jy)|\,dy
\lesssim \|f\|_{\Lambda_s}2^{j(|\alpha|-s)},
\end{align}
which further implies \eqref{ljoi}.
By the mean value theorem and the
just aforementioned claim,
we have, for any $j\in\mathbb Z_+$,
\begin{align}\label{lnjdf}
d_j&\leq \left\|\nabla [\Delta_h^r
(\phi_j\ast f)]\right\|_{L^\infty}
|x-x'| =\left\| \Delta_h^r [\nabla
(\phi_j\ast f)]\right\|_{L^\infty}
|x-x'|\nonumber\\
&\lesssim |h|^r \left\|\nabla^{r+1}
(\phi_j\ast f)\right\|_{L^\infty}
|x-x'|\lesssim
(y2^j)^r 2^{-j(s-1)} |x-x'|.
\end{align}
Moreover, using the mean value
theorem and \cite[Theorem 1.4.9]{g14},
we find that, for any $j\in\mathbb Z_+$,
\begin{align*}
d_j& \lesssim \max_{\{i\in\mathbb{Z}_+:0\leq i\leq r\}}\left|\phi_j\ast
f\left(x+\left[\frac r2-i\right]h\right)-\phi_j\ast
f\left(x'+\left[\frac r2-i\right]h\right)\right|\\
&\lesssim |x-x'|  \|\nabla
(\phi_j\ast f)\|_{L^\infty}\lesssim
2^{-j(s-1)}|x-x'|
\end{align*}
and $d_j\lesssim\|\phi_j\ast f
\|_{L^\infty}\lesssim2^{-js}.$
From this, \eqref{lnjdf}, and the
assumption $|x-x'|\leq y$, it follows
that
\begin{align*}
B_1&\lesssim
|x-x'|\left[\sum_{\{j\in\mathbb Z_+:2^j \leq 1/y\}}
2^{-j(s-1)} (2^j y)^r +
\sum_{\{j\in\mathbb Z_+:y^{-1} <2^j
\leq |x-x'|^{-1}\}} 2^{-j(s-1)}\right]
+\sum_{\{j\in\mathbb Z_+:2^j >
|x-x'|^{-1}\}} 2^{-js}\\
&\lesssim |x-x'|y^{s-1}+  |x-x'|
\sum_{\{j\in\mathbb Z_+:y^{-1} <2^j
\leq |x-x'|^{-1}\}} 2^{-j(s-1)}.
\end{align*}
Next, we consider two subcases for $s$
in this case. If $s=1$,
we have
$$B_1\lesssim|x-x'|\log \left( e+\frac y
{|x-x'|}\right).$$
If $s\in(1,\infty)$, we have
$B_1\lesssim y^{s-1} |x-x'|$,
which is the desired estimate of $B_1$.

As for $B_2$, using
\cite[Theorem 1.4.9]{g14}, we obtain
$$ B_2= \left| \Delta_{y\xi}^r f(x')
-\Delta_{y'\xi} ^r f(x')\right|\leq
\sum_{j=0}^\infty r_j,
$$
where, for any $j\in\mathbb Z_+$,
$r_j:= |\Delta_{y\xi}^r
(\phi_j\ast f)(x') -\Delta_{y'\xi} ^r
(\phi_j\ast f)(x')|.
$
Fix $j\in\mathbb Z_+$ and $x'\in\mathbb{R}^n$,
and let $\Phi_{x'}(z):=\Delta_z^r
(\phi_j\ast f)(x')$ for any $z\in \mathbb{R}^n$.
Then
$$
\nabla \Phi_{x'} (z)=\sum_{i=0}^r
\binom{r}{i} (-1)^{i}
\left(\frac r2-i\right) \nabla
(\phi_j\ast f)\left(x'+\left[
\frac r2-i\right]z\right)=\Delta_1^{r}
g_{z,x'}(0),
$$
where $g_{z,x'}(t):=t \nabla (\phi_j
\ast f)(x'+tz)$ for any $t\in\mathbb{R}$.
This, combined with \eqref{lnjdf2},
further implies that,
for any $z\in\mathbb{R}^n$,
\begin{align*}
| \nabla \Phi_{x'} (z)|&
\lesssim  \max_{|t|\leq r/2}
\left|g^{(r)}_{z,x'} (t)\right|
\lesssim
|z|^{r-1}\|\nabla^{r} (\phi_j\ast f)
\|_{L^\infty} +  |z|^{r}\|\nabla^{r+1}
 (\phi_j\ast f)\|_{L^\infty}\\
&\lesssim \left[ (2^j |z|)^{r-1}+(2^j
|z|)^{r}\right]2^{-j(s-1)}.
\end{align*}
By this and the mean value theorem,
we find that
\begin{align}\label{lnjdf4}
r_j=|\Phi_{x'} (y\xi)-\Phi_{x'} (y'\xi)|
\leq \max_{t\in [y',y]}  |\nabla \Phi_{x'}
(t\xi)| |y-y'|
\lesssim2^{-j(s-1)} \left[ (2^{j} y)^{r-1}
+(2^j y)^{r}\right]  |y-y'|.
\end{align}
Moreover, from the mean value theorem and
\cite[Theorem 1.4.9]{g14},
we infer that
\begin{align*}
r_j& \lesssim \max_{\{i\in\mathbb{Z}_+:0\leq i\leq r\}}\left|
\phi_j\ast f\left(x'+\left[\frac r2-i\right]
y\xi\right)-\phi_j\ast
f\left(x'+\left[\frac r2-i\right]y'\xi\right)
\right|\\&\lesssim |y-y'|  \|\nabla (\phi_j
\ast f)\|_{L^\infty}\lesssim
\|f\|_{\Lambda_s}2^{-j(s-1)}|y-y'|
\end{align*}
and $r_j\lesssim\|\phi_j\ast f\|_{L^\infty}
\lesssim\|f\|_{\Lambda_s}2^{-js}.$
Using this and \eqref{lnjdf4}, we conclude that
\begin{align*}
B_2
&\lesssim
|y-y'
|\left [\sum_{\{j\in\mathbb Z_+:2^j
\leq 1/y\}}
2^{-j(s-1)} (2^j y)^{r-1} +
\sum_{\{j\in\mathbb Z_+:y^{-1} <2^j
\leq |y-y'|^{-1}\}} 2^{-j(s-1)}\right]
+\sum_{\{j\in\mathbb Z_+:2^j > |y-y'
|^{-1}\}} 2^{-js}\\
&\lesssim |y-y'|y^{s-1}+  |y-y'|
\sum_{\{j\in\mathbb Z_+:y^{-1} <2^j
\leq |y-y'|^{-1}\}} 2^{-j(s-1)}.
\end{align*}
Again, we consider two subcases for $s$
in this case. If $s=1$,
we have
$$B_2\lesssim|y-y'|\log \left( e+
\frac y {|y-y'|}\right).$$
If $s\in(1,\infty)$, we have
$B_2\lesssim y^{s-1} |y-y'|$,
which is the desired estimate of $B_2$.
This finishes the proof of Lemma \ref{lem-4-5-0}.
\end{proof}

\begin{lemma}\label{lem-4-545}
Let $s\in(0,\infty)$, $r\in\mathbb N$ with $r>s$,
$\max\{y, y'\}\in
(1, \infty)$, and $x, x'\in\mathbb{R}^n$. Assume
in addition that $|x-x'|\leq \max\{y,y'\}
$ if $s\in[1,\infty)$.
Then there exists a positive constant
$C_{(r)}$, depending on $r$, such
that, for any $f\in\Lambda_s$,
\begin{align*}
&\left|\Delta_rf(x, y) -\Delta_r f(x', y')
\right|\\\nonumber
&\quad\leq C_{(r)} \|f\|_{\Lambda_s}
\begin{cases}
|x-x'|^s+|y-y'|^s &\text{if $s\in(0,1)$,}\\
\displaystyle|x-x'| \log \left( e+\frac {1} {|x-x'|}
\right)+ |y-y'|\log \left( e+\frac {1}
{|y-y'|}\right)&\text{if $s=1$,}\\
(y+y')^{s-1}( |y-y'|+|x-x'|) &\text{if
$s\in(1,\infty)$.}
\end{cases}
\end{align*}
\end{lemma}

\begin{proof}
By the proof of Lemma \ref{lem-4-5-0},
we find that,
for any $y, y'\in (0, \infty)$ and $x,
x'\in\mathbb{R}^n$,
\begin{align*}
|\Delta_rf(x,y)-\Delta_r f(x', y') |
&\leq \sup_{\xi\in\mathbb{S}^{n-1}}
\left|\Delta_{y\xi}^r f(x)-\Delta_{y\xi}^r
f(x')\right|+
\sup_{\xi\in\mathbb{S}^{n-1}} \left|
\Delta_{y\xi}^r f(x')-\Delta_{y'\xi}^r
f(x')\right|.
\end{align*}
Let
$B_1:=|\Delta_{y\xi}^r f(x)-
\Delta_{y\xi}^rf(x')|$ and
$B_2:=|\Delta_{y\xi}^r f(x') -
\Delta_{y'\xi} ^r f(x')|.$
Without loss of generality, we may
assume that $\|f\|_{\Lambda_s}=1$.
To prove the present lemma,
we consider the following two cases for $s$.

If $s\in(0,1)$,
then the proof of Lemma \ref{lem-4-5-0}
implies the desired estimate.

If $s\in[1,\infty)$, $|x-x'|\leq y$,
and $y\geq y'$, let $\xi\in
\mathbb{S}^{n-1}$, $h:=y\xi$,
and $\phi_j$ for any $j\in\mathbb Z_+$
be the same as in \eqref{eq-phi1} and
\eqref{eq-phik}.
For $B_1$, using \cite[Theorem 1.4.9]{g14},
we obtain
\begin{align*}
B_1&=\left|\Delta_h^r f(x)-\Delta_h^rf(x')
\right|\leq   \sum_{j=0}^\infty \left|\Delta_h^r
(\phi_j\ast f)(x)-\Delta_h^r(\phi_j\ast f
)(x')\right|=:\sum_{j=0}^\infty d_j.
\end{align*}
From  the mean value theorem and
\cite[Theorem 1.4.9]{g14},
we deduce that, for any $j\in\mathbb Z_+$,
\begin{align}\label{jlkjln2}
d_j& \lesssim \max_{\{i\in\mathbb{Z}_+:0\leq i\leq r\}}
\left|\phi_j\ast f\left(x+\left[\frac r2-i
\right]h\right)-\phi_j\ast
f\left(x'+\left[\frac r2-i\right]h\right)
\right|\nonumber\\
&\lesssim |x-x'|  \|\nabla
(\phi_j\ast f)\|_{L^\infty}\lesssim
2^{-j(s-1)}|x-x'|
\end{align}
and
\begin{align}\label{jlkjln}
d_j& \lesssim\|\phi_j\ast f\|_{L^\infty}
\lesssim2^{-js}.
\end{align}
Now, we consider two subcases for $|x-x'|$.
If $|x-x'|>1$, by \eqref{jlkjln},
we conclude that
$B_1\lesssim1\lesssim|x-x'|,$
which is the desired estimate of $B_1$.
If $|x-x'|\le1$, by \eqref{jlkjln2},
\eqref{jlkjln}, and the
assumption that $|x-x'|\leq 1$, we find that
\begin{align*}
B_1&\lesssim
|x-x'|\sum_{\{j\in\mathbb Z_+:0 <2^j \leq
|x-x'|^{-1}\}} 2^{-j(s-1)}+
\sum_{\{j\in\mathbb Z_+:2^j >
|x-x'|^{-1}\}} 2^{-js}\\
&\lesssim |x-x'|^{s}+  |x-x'|
\sum_{\{j\in\mathbb Z_+:0 <2^j
\leq |x-x'|^{-1}\}} 2^{-j(s-1)}.
\end{align*}
Next, we consider two subcases for $s$
in this case. If $s=1$,
we have
$$B_1\lesssim|x-x'|\log \left( e+
\frac 1 {|x-x'|}\right).$$
If $s\in(1,\infty)$, we have
$B_1\lesssim y^{s-1}|x-x'|$,
which is the desired estimate of $B_1$.

As for $B_2$, using
\cite[Theorem 1.4.9]{g14}, we conclude that
$$ B_2= \left| \Delta_{y\xi}^r f(x') -
\Delta_{y'\xi} ^r f(x')\right|
\leq \sum_{j=0}^\infty r_j,$$
where, for any $j\in\mathbb Z_+$,
$$r_j:= \left|\Delta_{y\xi}^r (\phi_j
\ast f)(x') -\Delta_{y'\xi} ^r (\phi_j
\ast f)(x')\right|.$$
From the mean value theorem and
\cite[Theorem 1.4.9]{g14},
we infer that, for any $j\in\mathbb Z_+$,
\begin{align*}
r_j& \lesssim \max_{\{i\in\mathbb{Z}_+:0\leq i\leq r\}}\left|
\phi_j\ast f\left(x'+\left[\frac r2-i
\right]y\xi\right)
-\phi_j\ast  f\left(x'+\left[\frac r2-i
\right]y'\xi\right)\right|\\&\lesssim
|y-y'|
\|\nabla (\phi_j\ast f)\|_{L^\infty}
\lesssim 2^{-j(s-1)}|y-y'|
\end{align*}
and $r_j\lesssim\|\phi_j\ast f\|_{L^\infty}
\lesssim2^{-js}.$
By this and \eqref{lnjdf4}, we find that
\begin{align*}
B_2&\lesssim|y-y'|
\sum_{\{j\in\mathbb Z_+:0 <2^j \leq
|y-y'|^{-1}\}} 2^{-j(s-1)}
+\sum_{\{j\in\mathbb Z_+:2^j >
|y-y'|^{-1}\}} 2^{-js}\\
&\lesssim |y-y'|^{s}+  |y-y'|
\sum_{\{j\in\mathbb Z_+:0 <2^j
\leq |y-y'|^{-1}\}} 2^{-j(s-1)}.
\end{align*}
Now, we consider two subcases for $s$
in this case. If $s=1$,
we have
$$B_2\lesssim|y-y'|\log \left( e+
\frac 1 {|y-y'|}\right).$$
If $s\in(1,\infty)$, we have
$B_2\lesssim y^{s-1}|y-y'|$,
which is the desired estimate of
$B_2$.
This finishes the proof of
Lemma \ref{lem-4-5-0}.
\end{proof}

For any $t\in(0,\infty)$ and $\mathbf{z}
:=(z,z_{n+1})\in\mathbb{R}_+^{n+1}$, let
$B_\rho(\mathbf{z}, t):=\{
\mathbf{y}\in\mathbb{R}_+^{n+1}:\rho(\mathbf{z}, \mathbf{y})< t\},$
where $\rho$ is as in \eqref{metric}.

\begin{proof}[Proof of Theorem \ref{thm-3adf}]
We claim that
there exists $\delta\in (0, \frac1{10})$ such that
\begin{equation}\label{5-6}
(S_r(s, f,\varepsilon))_\delta:=
\bigcup_{\mathbf{z}\in S_r(s, f,\varepsilon)}
B_\rho(\mathbf{z}, \delta) \subset
S_r(s, f,\varepsilon_1).
\end{equation}
Assuming that this claim holds for
the moment, then, for
any $(z, z_{n+1})\in(S_{r,j}(s, f,
\varepsilon))_\delta$,
$(z, z_{n+1})\in S_{r}(s, f,\varepsilon_1)$
and there exists $(z', z_{n+1}')\in S_{r,j}(s,
f,\varepsilon)$ such
that $\rho((z, z_{n+1}),(z', z_{n+1}'))<\delta$.
From this and \eqref{dafg1}, we deduce that
\begin{equation*}
\frac{1}{(1+2\delta) 2^{j+1}}<\frac{z_{n+1}'}{1+2\delta}
\leq z_{n+1}\leq
(1+2\delta)  z_{n+1}'\leq\frac{1+2\delta}{2^{j}}.
\end{equation*}
Since $\delta<\frac 14$, it follows that
$2^{-j-2}<z_{n+1}\le2^{-j+1}$,
which further implies that
$$
(z, z_{n+1})\in S_{r,j-1}(s, f, \varepsilon_1)\cup
 S_{r,j}(s, f, \varepsilon_1)\cup S_{r,j+1}(s, f,
 \varepsilon_1).
$$
Thus, the present theorem holds.

Therefore, to complete the proof of the
present theorem, it remains to show the above
claim \eqref{5-6}.
Let  $\mathbf{z}:=(z, z_{n+1})\in S_r(s, f,
\varepsilon)$ and $\mathbf{z}':=(z', z_{n+1}')
\in B_\rho(\mathbf{z}, \delta)$.
From \eqref{dafg} and \eqref{dafg2}, we infer that
\begin{equation}\label{5-7}
|z-z'|\leq 2 z_{n+1} \delta
\end{equation}
and
\begin{equation}\label{5-8}
|z_{n+1}-z_{n+1}'|\leq 2z_{n+1}\delta.
\end{equation}

We next consider three cases based on the size of $s$.

If $s\in(0,1)$, then, from Lemmas
\ref{lem-4-5-0} and \ref{lem-4-545}, \eqref{5-7},
and \eqref{5-8}, it follows that
\begin{align*}
	|\Delta_rf(z,z_{n+1})-\Delta_r(z',
z_{n+1}')| &\lesssim
 |z-z'|^s +|z_{n+1}-z_{n+1}'|^s\lesssim
 z_{n+1}^s \delta^s,
\end{align*}
which, together with the assumption that
$\mathbf{z}:=(z, z_{n+1})\in S_r(s, f,\varepsilon)$
and
\eqref{dafg1},
further implies that there exists a
positive constant $c$ such that
$$ \Delta_rf(z', z_{n+1}') \ge \Delta_rf(z,
z_{n+1})-c z_{n+1}^s\delta^s \ge (\varepsilon-
c\delta^s)z_{n+1}^s \ge \frac{\varepsilon-c\delta^s}{(2\delta+1)^{s}} z_{n+1}'^{s}.$$
Since $\varepsilon_1<\varepsilon$, we deduce
that there exists  $\delta\in (0, 1)$ such
that $\frac{\varepsilon-c\delta^s}{(2\delta+1)^{s}}
>\varepsilon_1$, which is the desired estimate
in this case.

If $s=1$, then, using Lemmas \ref{lem-4-5-0}
and \ref{lem-4-545}, \eqref{5-7},
and \eqref{5-8}, we obtain
\begin{align*}
&|\Delta_rf(z,z_{n+1})-\Delta_r(z', z_{n+1}')| \\
&\quad\lesssim z_{n+1}
\left[ \frac {|z-z'|}{3z_{n+1}}\log \left( e+
\frac {3z_{n+1}} {|z-z'|}\right) +\frac {|z_{n+1}-
z_{n+1}'|}{3z_{n+1}} \log
\left( e+\frac {3z_{n+1}} {|z_{n+1}-z_{n+1}'|}\right)
\right]\\
&\quad\lesssim z_{n+1}
\sup_{t \in(0,4\delta)} t \log \left(e+\frac 1t\right)
\lesssim \sqrt{\delta} z_{n+1},
\end{align*}
which, combined with the assumption that
$\mathbf{z}\in S_r(s, f,\varepsilon)$ and
\eqref{dafg1},
further implies that there exists a positive
constant $c$ such that
\begin{align*}
\Delta_rf(z', z_{n+1}')\ge\Delta_rf(z,z_{n+1})-
c z_{n+1}\sqrt{\delta}
\ge\left(\varepsilon-c\sqrt{\delta} \right) z_{n+1}\ge
\frac{\varepsilon-c\sqrt{\delta} }{2\delta+1} z_{n+1}'.
\end{align*}
Choosing $\delta\in(0,\infty)$ so that
$\frac{\varepsilon-c\sqrt{\delta} }{2\delta+1}>\varepsilon_1$,
we obtain the desired estimate in this case.

Finally, if $s\in(1,\infty)$,
using   Lemma \ref{lem-4-5-0}, \eqref{5-7},
and \eqref{5-8}, we conclude that
\begin{align*}
\left|\Delta_rf(z,z_{n+1})-\Delta_r(z', z_{n+1}')\right|
&\lesssim z_{n+1}^{s-1}  \left(|x-x'| +|y-y'|\right)
\lesssim z_{n+1}^s \delta,
\end{align*}
which, together with the assumption that
$\mathbf{z}\in S_r(s, f,\varepsilon)$ and
\eqref{dafg1}, further
implies that there exists a positive constant
$c$ such that
\begin{align*}
\Delta_rf(z', z_{n+1}')\ge \Delta_rf(z,z_{n+1})-c
z_{n+1}^s \delta\ge (\varepsilon-c\delta )  z_{n+1}^s
\ge \frac{ \varepsilon-c\sqrt{\delta} }{(2\delta+1)^{s}} z_{n+1}'^{s}.
\end{align*}
Choosing $\delta\in(0,\infty)$ so that
$ \frac{ \varepsilon-c\sqrt{\delta} }{(2\delta+1)^{s}}
>\varepsilon_1$,
we obtain  the desired estimate in this case.

Thus, for any $\mathbf{z}' \in B_\rho(\mathbf{z}, \delta)$,
$B_\rho(\mathbf{z}, \delta) \subset S_r(s, f,\varepsilon_1)$,
which further implies that the above claim holds.
This finishes the proof of Theorem \ref{thm-3adf}.
\end{proof}

\subsection{Connections between wavelets
 and differences}\label{sec:3-2}

In this subsection, we discuss the relation
between wavelets and
differences (see Theorems \ref{cor-6-2-0}
and \ref{cor-6-7-0}).  We should
point out that Theorems \ref{cor-6-2-0} and
\ref{cor-6-7-0} play an important
role in the proofs of Theorems \ref{pppz23}
and \ref{thm-7-11}.
For any cube  $I\subset\mathbb{R}^n$,
we always use $\ell(I)$ to
denote its edge length and let $\widehat I
:=I\times [0, \ell(I))$.

For any $\varepsilon\in(0,
\infty)$, $j\in\mathbb Z_+$, and
$f\in \Lambda_s$, let
$$
W^0(s, f, \varepsilon):= \left\{ I\in\mathcal D:\max_{\{\omega\in\Omega:I_{\omega}
=I\}}|\langle f,\psi_{(\ell,I)}\rangle|
>\varepsilon |I|^{\frac sn+\frac 12}\right\},
$$
$$
W(s, f, \varepsilon):=W^0(s, f, \varepsilon)
\cap \left\{ I_\omega\in\mathcal D:\omega\in\Omega_1\right\},
$$
$$
W_j(s, f, \varepsilon):=W(s, f, \varepsilon)
\cap\mathcal D_j,\
W_j^0(s, f, \varepsilon):=W^0(s, f, \varepsilon)
\cap\mathcal D_j,
$$
$$
V_0(s, f, \varepsilon):= W_0^0(s, f, \varepsilon)
\cap \left\{ I_\omega\in\mathcal D:\omega\in\Omega_0\right\},
$$
$$ T_j(s, f, \varepsilon): =
\bigcup_{I \in W_j(s, f, \varepsilon)} T(I),\
T_j^0(s, f, \varepsilon): =
\bigcup_{I \in W_j^0(s, f, \varepsilon)} T(I),$$
$$ T(s, f, \varepsilon): =
\bigcup_{I \in W(s, f, \varepsilon)} T(I),\
\text{and}\
T^0(s, f, \varepsilon): =
\bigcup_{I \in W^0(s, f, \varepsilon)} T(I).$$

\begin{theorem}\label{cor-6-2-0}
Let  $s\in(0,\infty)$, $r\in\mathbb N$ with
$r>s$, and  $f\in\Lambda_s$.
Assume that
the regularity parameter $L\in\mathbb N$ of
the Daubechies wavelet system
$\{\psi_\omega\}_{\omega\in\Omega}$ satisfies that
$L\geq r-1$.
Then there exist positive constants
$c\in (0, 1)$, $m\in\mathbb N$,
and  $R\in(1,\infty)$ such that, for
any $\varepsilon\in(0,\infty)$ and $j\in\mathbb Z_+$,
$$T_j(s,f,\varepsilon)
\subset\bigcup_{i=j+1}^{j+m}\left[S_{r,i}(s,f, c\varepsilon)
\right]_R,$$
where $c$, $m$,
and  $R$ are  independent of $f$ and $j$.
\end{theorem}

Towards the proof of the above theorem,
we first establish the following Whitney
type inequality.

\begin{lemma}\label{thm-691}
Let $A_0\in(2,\infty)$ and $r\in\mathbb N$.
Then  there exist positive constants $C$ and $A$
such that, for any
$f\in \mathfrak{C}$ and $I\in \mathcal D$,
\begin{equation*}
\inf_{P\in\Pi_{r-1}^n}  \|(f-
P)\mathbf{1}_{A_0 I}\|_{L^\infty  }
\leq C \sup_{y\in[\frac1A\ell(I),
\frac12\ell(I)]} \sup _{x\in AI}\Delta_r f(x,y),
\end{equation*}
where $\Pi_{r-1}^n$ denotes the linear space
of all polynomials on ${\mathbb{R}^n}$
of total degree not greater than $r-1$ and
$\Delta_r f$ is as in
\eqref{nnfpaw} with $k$ replaced
by $r$.
\end{lemma}

\begin{proof}
Fix $I\in\mathcal D$.
For any $i\in\{1,\ldots,n\}$, let $e_i$ be
the unit vector in the $i$-th direction.
From \cite[Lemma 2.1]{d}, it follows that
\begin{align}\label{awof}
\inf_{P\in\Pi_{r-1}^n}  \|(f-P)\mathbf{1}_{A_0
I}\|_{L^\infty  }
& \lesssim \max_{\{i\in\mathbb{N}:1\le i\le n\}}\sup_{\{h\in\mathbb{R}:|h|\le
\frac{A_0\ell(I)}{r}\}}
\sup _{x,x\pm \frac{rhe_i}{2}\in A_0 I}
\Delta^r_{he_i} f(x).
\end{align}
Let
$
A_0I:=[a_1,b_1]\times\cdots\times[a_{n},
b_{n}]
$
and
$$
Q_{e_{i}}:=[a_1,b_1]\times\cdots\times
[a_{i-1},b_{i-1}]\times\{0\}\times
[a_{i+1},b_{i+1}]
\times\cdots\times[a_{n},b_{n}].
$$
Observe that, for any $a\in\mathbb N$,
\begin{align}\label{fsadf}
\left|D^r_{ah}[f(\xi+(\cdot)e_i)](\tau)
\right|\lesssim\sum_{j=0}^{(a-1)r}
\left|D^r_{h}[f(\xi+(\cdot+jh)e_i)](\tau)\right|.
\end{align}
By this and \cite[p.\,185, Lemma 5.1]{dl93},
we conclude that
\begin{align*}
&\sup_{0<h\le\frac{A_0\ell(I)}{r}}
\sup _{x,x\pm \frac{rhe_i}{2}\in A_0 I}
\left|\Delta^r_{he_i} f(x)\right|\\
& \quad=\sup_{0<h\le\frac{A_0\ell(I)}{r}}
\sup _{x\in A_0 I,x+ rhe_i\in A_0 I}
\left|D^r_{he_i} f(x)\right|
=\sup_{\xi\in Q_{e_{i}}}\sup_{0<h\le\frac{
A_0\ell(I)}{r}}
\sup _{\tau,\tau+rh\in[a_i,b_i]} \left|D^r_{h}[f(\xi+(\cdot)e_i)](\tau)\right|\\
&\quad =\sup_{\xi\in Q_{e_{i}}}\sup_{0<
h\le\frac{A_0\ell(I)}{4r}}
\sup _{\tau,\tau+4rh\in[a_i,b_i]}
\left|D^r_{4h}[f(\xi+(\cdot)e_i)](\tau)
\right|\\
&\quad\lesssim\sup_{\xi\in Q_{e_{i}}}\sup_{
0<h\le\frac{A_0\ell(I)}{4r}}
\max_{0\le j\le3r}
\sup _{\tau,\tau+4rh\in[a_i,b_i]}
\left|D^r_{h}[f(\xi+(\cdot+jh)e_i)](\tau)
\right|\\
&\quad\lesssim\sup_{\xi\in Q_{e_{i}}}
\frac{1}{\ell(I)}\int_{\frac{
A_0\ell(I)}{8r^2}}^{\frac{A_0\ell(I)}{4r}}
\sup _{\tau,\tau+rh\in[a_i,b_i]}
\left|D^r_{h}[f(\xi+(\cdot)e_i)](\tau)
\right|\,dh.
\end{align*}
which further implies that there exists
a positive integer $m$
such that
\begin{align}\label{,opw}
\sup_{0<h\le\frac{A_0\ell(I)}{r}}
\sup _{x,x\pm \frac{rhe_i}{2}\in A_0 I}
\left|\Delta^r_{he_i} f(x)\right|
\lesssim\sup_{\xi\in Q_{e_{i}}}\frac{1}{
\ell(I)}
\int_{\frac{\ell(I)}{2^{m+1}}}^{\frac{2^{m} \ell(I)}{2}}
\sup _{\tau,\tau+rh\in[a_i,b_i]} \left|
D^r_{h}[f(\xi+(\cdot)e_i)](\tau)
\right|\,dh.
\end{align}
Using \eqref{fsadf}, we find that, for
any $a\in\mathbb N$,
\begin{align*}
&\int_{\frac{a\ell(I)}{4}}^{\frac{a\ell(I)}{2}}
\sup _{\tau,\tau+rh\in[a_i,b_i]}
\left|D^r_{h}[f(\xi+(\cdot)e_i)](\tau)
\right|\,dh\\
&\quad\lesssim\int_{\frac{\ell(I)}{4}
}^{\frac{\ell(I)}{2}}
\sup _{\tau,\tau+rah\in[a_i,b_i]}
\left|D^r_{ah}[f(\xi+(\cdot)e_i)](\tau)
\right|\,dh\\
&\quad\lesssim\max_{\{j\in\mathbb{Z}_+:0\le j\le(a-1)r\}}
\int_{\frac{\ell(I)}{4}
}^{\frac{\ell(I)}{2}}
\sup _{\tau,\tau+rah\in[a_i,b_i]}
\left|D^r_{h}[f(\xi+(\cdot+jh)e_i)]
(\tau)\right|\,dh\\&\quad\lesssim\sup_{
\frac{\ell(I)}{4}\le h\le\frac{\ell(I)}{2}}
\sup _{\tau,\tau+rh\in[a_i,b_i]} \left|
D^r_{h}[f(\xi+(\cdot)e_i)](\tau)
\right|.
\end{align*}
From this and \eqref{,opw}, we infer that
\begin{align*}
\sup_{0<h\le\frac{A_0\ell(I)}{r}}
\sup _{x,x\pm \frac{rhe_i}{2}\in A_0 I}
\left|\Delta^r_{he_i} f(x)\right|
&\lesssim
\sup_{\frac{\ell(I)}{2^{m+1}}\le h\le\frac{\ell(I)}{2}}
 \sup _{x,x\pm \frac{rhe_i}2\in A_0 I}
 \left|\Delta^r_{he_i} f(x)\right|,
\end{align*}
which, combined with \eqref{awof},
further implies
that
\begin{align*}
\inf_{P\in\Pi_{r-1}^n}
\|(f-P)\mathbf{1}_{A_0 I}\|_{L^\infty
}
& \lesssim \max_{\{i\in\mathbb{N}:1\le i\le n\}}\sup_{
\frac{\ell(I)}{2^{m+1}}\le h\le\frac{\ell(I)}{2}}
\sup _{x,x\pm \frac{rhe_i}{2}\in A_0 I}
\left|\Delta^r_{he_i} f(x)\right|\\
& \le\sup_{\frac{\ell(I)}{2^{m+1}}\leq y\leq
\frac{\ell(I)}{2}}
\sup _{x\in A_0 I}\Delta_r f(x,y).
\end{align*}
This finishes the proof of Lemma
\ref{thm-691}.
\end{proof}

To prove Theorem \ref{cor-6-2-0},
we first need the following conclusion
about
wavelets and differences.
In what follows, for any $p\in(0,\infty]$,
we use $L^p$ to denote the Lebesgue
space equipped with the well-known (quasi-)norm.

\begin{lemma}\label{thm-6-1}
Let  $s\in(0,\infty)$ and $r\in\mathbb
N$ with $r>s$.
Assume that
the regularity parameter $L\in\mathbb
N$ of the Daubechies wavelet system
$\{\psi_\omega\}_{\omega\in\Omega}$
satisfies that
$L\geq r-1$.
Then  there exist positive constants
$C$ and $A_0$ such that, for any
$f\in\Lambda_s$ and $I\in \mathcal D$,
\begin{equation*}
|I|^{-\frac12 -\frac sn}	\max_{\{l\in\mathbb{N}:1\leq
l<2^n\}} |\langle f,\psi_{(\ell,I)}\rangle|
\leq C\sup _{(x,y)\in T_{A_0}(I)}\frac{\Delta_r f(x,y)}
 { y^s},
\end{equation*}
where $T_{A_0}(I):=A_0I\times[ \frac{\ell(I)}{A_0}
,\frac{\ell(I)}2]$
and $\Delta_r f$ is as in
\eqref{nnfpaw} with $k$ replaced
by $r$.
\end{lemma}

\begin{proof}
Fix $I\in\mathcal D$.
Let $\omega \in\Omega_1$ be such that
$I_\omega=I$.
Assume that $A_0\in(1,\infty)$ is such
that
$\mathop\mathrm{\,supp\,} \psi_\omega \subset
A_0 I $  and $\int_{\mathbb{R}^n} x^\alpha \psi_\omega (x)
\, dx=0$ for any
$\alpha\in\mathbb Z_+^n$ with $|\alpha|\leq
L$.
Let $\Pi_L^n$ denote the linear space of all
polynomials on ${\mathbb{R}^n}$
of total degree not greater than $L$.
Then, for any $P\in\Pi_L^n$,
\begin{align*}
|\langle f,\psi_{\omega}\rangle|
=\left|\int_{ \mathbb{R}^n} [f(x) -P(x)]
\psi_\omega(x)\, dx \right|\leq \|(f-P)
\mathbf{1}_{A_0 I}\|_{L^\infty  }
\|\psi_\omega\|_{L^1}
\lesssim |I|^{\frac12}  \|(f-P)\mathbf{1}_{
A_0 I}\|_{L^\infty  }.
\end{align*}
By this and Lemma \ref{thm-691}, we then obtain
\begin{align*}
|\langle f,\psi_{\omega}\rangle|&\lesssim
|I|^{\frac12} \inf_{P\in\Pi_L^n}
\|(f-P)\mathbf{1}_{A_0 I}\|_{L^\infty  }
\lesssim |I|^{\frac12} \sup _{x\in A_0 I}
\sup_{\frac{\ell(I)}{A_0}\leq y\leq \frac{\ell(I)}{2}}
\Delta_r f(x,y).
\end{align*}
This further implies the desired estimate and hence
finishes the proof of Lemma \ref{thm-6-1}.
\end{proof}

\begin{proof} [Proof of Theorem \ref{cor-6-2-0}]
By Lemma \ref{thm-6-1}, there exists a constant
$A_0\in(1,\infty)$ such that,
for any $I\in W_j(s,f, \varepsilon)$,
$$ \varepsilon\leq |I|^{-\frac12 -\frac sn}	
\max_{\{l\in\mathbb{N}:1\leq l<2^n\}} |\langle f,\psi_{(\ell,I)}\rangle|
\lesssim \sup_{(x, y) \in T_{A_0}(I)}\frac{\Delta_r f(x,y)}
 { y^s}.$$
Let $m\in\mathbb N$ be such that  $2^{m-1}\le
A_0<2^m.$
Then there
exists a positive constant $c\in(0,1)$ such
that, for any $I\in W_j(s,f, \varepsilon)$,
$T_{A_0}(I)\cap[\bigcup_{i=j+1}^{j+m}S_{r,
i}(s,f, c\varepsilon)]\neq\emptyset.$
From Lemma \ref{lem-3-2-0}, it follows that
there
exists a positive constant $R$ such that, for
any $I\in W_j(s,f, \varepsilon)$,
$
T_{A_0}(I) \subset  \bigcup_{i=j+1}^{j+m}
[S_{r,i}(s,f, c\varepsilon)]_R.
$
This shows that
$T_j(s,f,\varepsilon) \subset\bigcup_{i=j+1}^{j+m}[S_{r,i}(s,f,
c\varepsilon)]_R,$ which completes
the proof of Theorem \ref{cor-6-2-0}.
\end{proof}

\begin{theorem}\label{cor-6-7-0}
Let  $s,\varepsilon\in(0,\infty)$, $r\in\mathbb
N$ with $r>s$, and  $f\in\Lambda_s$.
Assume that
the regularity parameter $L\in\mathbb N$ of the
Daubechies wavelet system $\{\psi_\omega\}_{
\omega\in\Omega}$ satisfies that
$L\geq r-1$ and $L>s$.
Then there exist positive constants $c\in (0, 1)$,
$m\in\mathbb N$,
and  $R\in(1,\infty)$ such that, for any
$j\in\mathbb Z_+$,
$$S_{r,j}(s,f, \varepsilon)\subset
\bigcup_{i=j-m}^{j+m}\left[T_i^0(s,f,
c\varepsilon)\right]_R,$$
where, for any $j\in\mathbb Z$ with $j<0$, $T_j^0(s,f, c\varepsilon):=\emptyset$
and the constants $m$ and $R$ may depend on $f$.
\end{theorem}

To show Theorem \ref{cor-6-7-0}, we first
need the following conclusion about
wavelets and differences.
For any given  $(x,y)\in\mathbb{R}_+^{n+1}$
and $A,R\in(0,\infty)$,
let  $\mathcal{A}_{R,A}(x,y)$ denote the
collection of all $\omega \in\Omega$ such that
$\frac y R \leq \ell(I_{\omega}) \le Ry$
and $|x-c_{I_{\omega}}|\leq A Ry.$
\begin{lemma}\label{thm-6-5-0}
Let  $s\in(0,\infty)$ and $r\in\mathbb N$
with $r>s$.
Assume that the regularity parameter $L\in\mathbb N$ of
the Daubechies wavelet system $\{
\psi_\omega\}_{\omega\in\Omega}$ satisfies
that $L\geq r-1$ and $L>s$.
Then there exist positive constants  $A$ and
$C$ such that,
for any $R\in(1,\infty)$, $f\in\Lambda_s$, and
$(x,y)\in\mathbb{R}_+^{n+1}$,
\begin{equation*}
\frac{\Delta_r f(x,y)}
 { y^s}\leq C\left\{\sup_{
\omega\in\mathcal{A}_{R,A}(x,y)}
|I_{\omega}|^{-\frac12-\frac sn}
|\langle f,\psi_{\omega}\rangle|+
\|f\|_{\Lambda_s}
\left[\left(\frac 1R\right)^s+\left(
\frac 1R\right)^{r-s}\right]\right\},
\end{equation*}
where $\Delta_r f$ is as in
\eqref{nnfpaw} with $k$ replaced
by $r$.
\end{lemma}
\begin{proof}
Assume that $A_0\in(1,\infty)$ is such that,
for any
$\omega \in\Omega$,
$\mathop\mathrm{\,supp\,} \psi_\omega \subset
A_0 I_\omega $.
For any $y\in(0,\infty)$, let $h\in\mathbb{R}^n$
be such that $|h|=y$.
For any given $\omega \in\Omega$ and for any
$x\in\mathbb{R}^n$,
let $$W_{I_\omega,h}(x):=\max_{\{i\in\mathbb Z:
0\le i\le r\}}\mathbf{1}_{A_0I_{\omega}}\left(x
+\left[\frac r2-i\right]h\right).$$
Then,  for any
$x\in\mathbb{R}^n$,
\begin{align}\label{lkj}
|\Delta^r_h(\psi_\omega)(x)|\lesssim\min
\left\{|I_\omega|^{-\frac 12},|I_\omega|^{-\frac 12-r}y^r\right\}W_{I_\omega,h}(x).
\end{align}
From Lemma \ref{asqw}, we deduce that
\begin{align*}
f&=\sum_{\omega \in\Omega}\langle f,
\psi_{\omega}\rangle \psi_\omega
=\sum_{i=1}^4\sum_{\omega \in\Gamma_i}
\langle f,\psi_{\omega}\rangle \psi_\omega=:
\sum_{i=1}^4f_i,
\end{align*}
where
$
\Gamma_1:=\{\omega \in\Omega:\ell(I_{
\omega})\le R^{-1}y\}$,
$$
\Gamma_2:=\{\omega \in\Omega:\ell(I_{
\omega})\geq Ry\},\
\Gamma_3:=\mathcal{A}_{R,A}(x,y),
$$
and
$
\Gamma_4:=\{\omega \notin\mathcal{A}_{R,A}(x,y): R^{-1}y\le\ell(I_{\omega})\le Ry\}.
$
For $f_1$, using Lemma \ref{asqw} and \eqref{lkj},
we obtain, for any $x\in\mathbb{R}^n$,
\begin{align*}
\left|\Delta^r_h(f_1)(x)\right|&=\left|\sum_{
\omega \in\Gamma_1}\langle f,\psi_{\omega}\rangle \Delta^r_h(\psi_\omega)(x)\right|
\lesssim\sum_{\omega \in\Gamma_1}|\langle f,\psi_{\omega}\rangle||I_\omega|^{-\frac 12}
W_{I_\omega,h}(x)\\
&\lesssim\sum_{\{j\in\mathbb Z_+:2^{-j}\le R^{-1}y\}}
\sum_{\{\omega \in\Gamma_1:I_\omega\in\mathcal{D}_j\}}|
\langle f,\psi_{\omega}\rangle||I_\omega|^{-
\frac 12-\frac sn} |I_\omega|^{\frac sn}
W_{I_\omega,h}(x)\\
&\lesssim\sum_{\{j\in\mathbb Z_+:2^{-j}\le
R^{-1}y\}}2^{-js}\|f\|_{\Lambda_s}
\lesssim y^s \frac 1{R^s}\|f\|_{\Lambda_s}.
\end{align*}

As for $f_2$, by Lemma \ref{asqw} and
\eqref{lkj}, we conclude that, for any
$x\in\mathbb{R}^n$,
\begin{align*}
\left|\Delta^r_h(f_2)(x)\right|
&=\left|\sum_{\omega \in\Gamma_2}\langle
f,\psi_{\omega}\rangle \Delta^r_h(
\psi_\omega)(x)\right|\lesssim\sum_{\omega
\in\Gamma_2}|\langle f,\psi_{\omega}\rangle|
|I_\omega|^{-\frac 12-\frac rn}y^r W_{I_\omega,
h}(x)\\
&\lesssim\sum_{\{j\in\mathbb Z_+:2^{-j}\geq Ry\}}
\sum_{\{\omega \in\Gamma_1:I_\omega\in
\mathcal{D}_j\}}|\langle f,\psi_{\omega}
\rangle||I_\omega|^{-\frac 12-\frac sn}
|I_\omega|^{\frac sn-\frac rn}y^r W_{I_\omega,h}(x)\\
&\lesssim\sum_{\{j\in\mathbb Z_+:2^{-j}
\geq Ry\}}2^{-j(s-r)}y^r\|f\|_{\Lambda_s}
\lesssim y^s \frac 1{R^{r-s}}\|f\|_{\Lambda_s}.
\end{align*}

Next, we deal with $f_3$.
From Lemma \ref{asqw} and \eqref{lkj},
we infer that
\begin{align*}
\left|\Delta^r_h(f_3)(x)\right|&=\left|
\sum_{\omega \in\Gamma_3}\langle f,\psi_{
\omega}\rangle \Delta^r_h(\psi_\omega)(x)
\right|
\le\sup_{\omega\in\mathcal{A}_{R,A}(x,y)}
|I_{\omega}|^{-\frac12-\frac sn}
|\langle f,\psi_{\omega}\rangle|
\sum_{\omega \in\Gamma_3}|I_\omega|^{
\frac 12+\frac sn}\left|\Delta^r_h(
\psi_\omega)(x)\right|\\
&\lesssim\sup_{\omega\in\mathcal{A}_{R,A}
(x,y)}
|I_{\omega}|^{-\frac12-\frac sn}
|\langle f,\psi_{\omega}\rangle|
\left[\sum_{\{j\in\mathbb Z_+:2^{-j}\le y\}}
\sum_{\{\omega \in\Gamma_3:I_\omega\in
\mathcal{D}_j\}} |I_\omega|^{\frac sn}
W_{I_\omega,h}(x)\right.\\
&\quad\left.+
\sum_{\{j\in\mathbb Z_+:2^{-j}> y\}}
\sum_{\{\omega \in\Gamma_3:I_\omega
\in\mathcal{D}_j\}}
|I_\omega|^{\frac sn-\frac rn}y^r
W_{I_\omega,h}(x)\right]\\
&\lesssim\sup_{\omega\in\mathcal{A}_{R,
A}(x,y)}
|I_{\omega}|^{-\frac12-
\frac sn}
|\langle f,\psi_{\omega}\rangle| y^s.
\end{align*}

Finally, we estimate $f_4$.
We have
\begin{align*}
\Delta^r_h(f_4)(x)&=\sum_{\omega
\in\Gamma_4}\langle f,\psi_{\omega}
\rangle \Delta^r_h(\psi_\omega)(x).
\end{align*}
Observe that, for any $\omega \in\Omega$,
$\Delta^r_h(\psi_\omega)$ is supported in
$$\left\{x\in\mathbb{R}^n:\text{for some}\ i\in\{0,\ldots,r\},\
x+\left(\frac r2-i\right)h\in A_0I\right\}.$$
In what follows, for any $I:=2^{-j} ( k+[0,1)^n) \in\mathcal
D$ with $j\in\mathbb{Z}_+$ and $k\in \mathbb{Z}_+^n$,
let $c(I):=2^{-j}k$ and $\ell(I):=2^{-j}$.
Moreover, for any $\omega\in\Gamma_4$,
$\frac {\ell(I_\omega)}{R}\le y\le R\ell(I_\omega)$
and
$|x-c(I_\omega)|\geq ARy$.
From this, it follows that, if $A$ is
sufficiently large,
then, for any $i\in\{0,\ldots,r\}$,
\begin{align*}
\left|x+\left(\frac r2-i\right)h-c(
I_\omega)\right|\geq|x-
c(I_\omega)|-\left(\frac r2-i\right)y
\geq\left(AR-\frac r2\right)y\geq
\left(A-\frac r2\right)\ell(I_\omega)
\geq2A_0\ell(I_\omega),
\end{align*}
which further implies that $x+(
\frac r2-i)h\notin A_0 I$.
By this, we obtain, for any $\omega
\in\Gamma_4$,
$\Delta^r_h(\psi_\omega)(x)=0$,
which further implies that
$\Delta^r_h(f_4)(x)=0$.

Combining the estimates for
$\{f_i\}_{i=1}^4$,
we conclude that, for any $h\in
\mathbb{R}^n$ with $|h|=y$,
\begin{equation*}
\left|\frac{\Delta^r_h (f)(x)}
 { y^s}
\right| \lesssim\sup_{\omega
\in\mathcal{A}_{R,A}(x,y)}
|I_{\omega}|^{-\frac12-\frac sn}
|\langle f,\psi_{\omega}\rangle|+
\|f\|_{\Lambda_s}
\left[\left(\frac 1R\right)^s
+\left(\frac 1R\right)^{r-s}\right],
\end{equation*}
which completes
the proof of Lemma \ref{thm-6-5-0}.
\end{proof}

\begin{proof} [Proof of Theorem
\ref{cor-6-7-0}]
By Lemma \ref{thm-6-5-0}, we conclude
that  there exists a constant
$R_1\in(1,\infty)$, depending on
$\varepsilon$
and $f$, such that, for any
$(x,y)\in\mathbb{R}_+^{n+1}$,
\begin{equation*}
\frac{\Delta_r f(x,y)}
 { y^s} \leq C\sup_{
\omega\in\mathcal{A}_{R_1,A}(x,y)}
|I_{\omega}|^{-\frac12-\frac sn}
|\langle f,\psi_{\omega}\rangle|
+\frac {\varepsilon}2,
\end{equation*}
which further implies that, for any given
$(x,y)\in S_{r,j}(s,f, \varepsilon)$,
there exists
$\omega\in \Omega$ such that
\begin{equation}\label{kibo2}
\frac y R_1 \leq \ell(I_{\omega})
\le R_1y,\  |x-c_{I_{\omega}}|\leq A R_1y,
\end{equation}
and
$c\varepsilon\leq
|I_{\omega}|^{-\frac12 -\frac sn}
|\langle f,\psi_{\omega}\rangle|.$
Let $m\in\mathbb N$ be such that
$2^{-m-j-1}\le \ell(I_{\omega})<2^{m-j}.$
Then
$$I_{\omega}\in \bigcup_{i=j-m}^{j+m}
W_i^0(s,f, c\varepsilon),$$
where, for any $i<0$, $W_i^0(s,f, c
\varepsilon):=\emptyset$.
This further implies that
$$T(I_{\omega})\cap
\left[\bigcup_{i=j-m}^{j+m}T_{i}^0(s,f,
\varepsilon)\right]\neq\emptyset,$$
where,
for any $i<0$, $T_i^0(s,f, c\varepsilon)
:=\emptyset$.
From \eqref{kibo2}, it follows that
there exists a positive constant $R\in(1,
\infty)$ such that $(x,y)\in T(I_{\omega})_R$.
This proves that
$S_{r,j}(s,f, \varepsilon)
\subset\bigcup_{i=j-m}^{j+m}
[T_{i}^0(s,f, \varepsilon)]_R,$
which completes
the proof of Theorem \ref{cor-6-7-0}.
\end{proof}

\subsection{Proof of Theorem
\ref{pppz23}}\label{sec:3-3}

We also need the following convexification
concept of quasi-normed lattices
of function sequences.

\begin{definition}\label{tuhua}
Let $X$ be a quasi-normed lattice
of function sequences and $u\in(0,\infty)$.
The \emph{$u$-convexification} $X^u$ of $X$
is defined by setting
$X^u:=\{\mathbf{F}\in\mathscr{M}_{
\mathbb{Z}_+}(\mathbb{R}^n):\mathbf{F}=
\{| f_j|^u\}_{j\in\mathbb{Z}_+}\in X\}$
equipped with the quasi-norm $\|\mathbf{F}
\|_{X^u}:=\|\,\{| f_j|^u\}_{j\in\mathbb{Z}_+}
\|_{X}^\frac{1}{u}$
for any $\mathbf{F}\in X^u$.
\end{definition}
\begin{lemma}\label{a2.15s}
Let $R\in(0,\infty)$ and $\delta\in(0,1/4)$ with
$R>\delta$.
Let $X$  be a quasi-normed lattice of function
sequences satisfying Assumption \ref{a1}
for some $u\in(0,\infty)$.
Then there exists a positive constant $C$ such that,
for any sequence $\{A_j\}_{j\in\mathbb Z_+}$ with $A_j$
for any $j\in\mathbb Z_+$
being a measurable subset of $\mathbb{R}^{n+1}_+$,
\begin{equation*}
\left\|\left\{\int_{0}^\infty \mathbf{1}_{
(A_j)_R}(\cdot,y)\,\frac{dy}{y}
\right\}_{j\in\mathbb Z_+}\right\|_{X^u}
\le C\left\|\left\{\int_{0}^\infty
\mathbf{1}_{(A_j)_{\delta}}(\cdot,y)\,
\frac{dy}{y}
\right\}_{j\in\mathbb Z_+}\right\|_{X^u}.
\end{equation*}
\end{lemma}
\begin{proof}
Fix $j\in\mathbb Z_+$. By the Zorn lemma,
we conclude that there exists
a maximal subset $P_j$ of $A_j$ such that
$\rho(\mathbf{z},\mathbf{z}' )\ge 2\delta$
for any $\mathbf{z},\mathbf{z}' \in P_j$
with $\mathbf{z}\neq\mathbf{z}'$.
From this, we deduce that
$A_j\subset \bigcup_{\mathbf{z}\in P_j}
B_\rho(\mathbf{z}, 2\delta)$,
which further implies that
$(A_j)_R\subset \bigcup_{\mathbf{z}\in P_j}
B_\rho(\mathbf{z}, R+2\delta)$.
Using this and both (i) and (iv) of Lemma
\ref{ddaf},  we conclude that
$P_j$ for any $j\in\mathbb Z_+$ is a
countable set and
there exists a positive constant $C$
such that,
for any $j\in\mathbb Z_+$ and
$(x,y)\in{\mathbb R}_+^{n+1}$,
\begin{align*}
\mathbf{1}_{(A_j)_R}(x,y) &\leq
\sum_{\mathbf{z}\in P_j}\mathbf{1}_{
B_\rho (\mathbf{z}, R+2\delta)}(x,y)
\le\sum_{(z,z_{n+1})\in P_j}
\mathbf{1}_{B(z, Cz_{n+1})}(x)
\mathbf{1}_{[C^{-1}z_{n+1}, Cz_{n+1}]}(y),
\end{align*}
which further implies that
\begin{align}\label{boooi}
\int_{0}^\infty \mathbf{1}_{(A_j)_R}(x,y)
\,\frac{dy}{y}&\leq
\sum_{(z,z_{n+1})\in P_j}
\mathbf{1}_{B(z, Cz_{n+1})}(x)
\int_{0}^\infty\mathbf{1}_{[C^{-1}z_{n+1},
Cz_{n+1}]}(y)\,\frac{dy}{y}\nonumber\\
&\lesssim\sum_{(z,z_{n+1})\in P_j}
\mathbf{1}_{B(z, Cz_{n+1})}(x).
\end{align}
For any $\mathbf{z}:=(z,z_{n+1})\in P_j$,
let $I_{\mathbf{z}}:=B(z,\delta z_{n+1}/4)$.
From Lemma \ref{ddaf}(i),
we infer that
\begin{align}\label{boooi2}
B_\rho (\mathbf{z}, \delta)\supset
I_{\mathbf{z}}\times
\left[\left(1-\frac \delta4\right)z_{n+1},
\left(1+\frac \delta4\right)z_{n+1}\right]
=:J(\mathbf{z}, \delta).
\end{align}
Let $\alpha\in(1,\infty)$ be such that
$\alpha p/q>1$.
Observe that, for any $x\in\mathbb{R}^n$,
$
\mathbf{1}_{B(z, Cz_{n+1})}(x)=\mathbf{1}_{
(4C/\delta)I_{\mathbf{z}}}(x).
$
From this, Assumption \ref{a1}, \eqref{boooi},
and
\eqref{boooi2},
we deduce that
\begin{align*}
\left\|\left\{\int_{0}^\infty \mathbf{1}_{
(A_j)_R}(\cdot,y)\,\frac{dy}{y}
\right\}_{j\in\mathbb Z_+}\right\|_{X^u}&
\lesssim\left\|\left\{
\sum_{(z,z_{n+1})\in P_j}
\mathbf{1}_{B(z, Cz_{n+1})}
\right\}_{j\in\mathbb Z_+}\right\|_{X^u}
\lesssim\left\|\left\{
\sum_{(z,z_{n+1})\in P_j}
\mathbf{1}_{I_{\mathbf{z}}}\right\}_{j\in
\mathbb Z_+}\right\|_{X^u}\nonumber\\
&\lesssim\left\|\left\{
\sum_{(z,z_{n+1})\in P_j}
\int_{0}^\infty\mathbf{1}_{J(\mathbf{z}, \delta)}(\cdot,y)\,\frac{dy}{y}\right\}_{j\in\mathbb Z_+}\right\|_{X^u}\nonumber\\
&\lesssim\left\|\left\{
\sum_{(z,z_{n+1})\in P_j}
\int_{0}^\infty\mathbf{1}_{B_\rho (\mathbf{z}, \delta)}(\cdot,y)\,\frac{dy}{y}\right\}_{j\in\mathbb Z_+}\right\|_{X^u}\nonumber\\
&\lesssim\left\|\left\{\int_{0}^\infty \mathbf{1}_{(A_j)_{\delta}}(\cdot,y)\,\frac{dy}{y}
\right\}_{j\in\mathbb Z_+}\right\|_{X^u}\nonumber,
\end{align*}
which completes the proof of Lemma \ref{a2.15s}.
\end{proof}

\begin{lemma}\label{a2.15ss}
Let $R\in(0,\infty)$ and let
$X$  be a quasi-normed lattice of function
sequences satisfying Assumption \ref{a1}
for some $u\in(0,\infty)$.
Then there exists a positive constant $C$ such that,
for any $k\in\mathbb N$ and any sequence
$\{I_{k,j}\}_{j\in\mathbb Z_+}$ with
$I_{k,j}\in \mathcal D$ for any $j\in\mathbb Z_+$,
\begin{equation*}
\left\|\left\{\int_{0}^\infty \mathbf{1}_{
(\bigcup_{k\in\mathbb N}T(I_{k,j}))_R}(\cdot,y)\,\frac{dy}{y}
\right\}_{j\in\mathbb Z_+}\right\|_{X^u}
\le C\left\|\left\{\int_{0}^\infty
\mathbf{1}_{\bigcup_{k\in\mathbb N}T(
I_{k,j})}(\cdot,y)\,\frac{dy}{y}
\right\}_{j\in\mathbb Z_+}\right\|_{X^u}.
\end{equation*}
\end{lemma}
\begin{proof}
By Lemmas \ref{lem-3-2-0} and \ref{lem-3-24},
we conclude that
there exist positive constants $C_1,C_2$
such that, for any
$(x,y)\in{\mathbb R}_+^{n+1}$,
\begin{align*}
\mathbf{1}_{(\bigcup_{k\in\mathbb N}
T(I_{k,j}))_R}(x,y)
&\leq\sum_{k\in\mathbb N}\mathbf{1}_{
(T(I_{k,j}))_R}(x,y)\le\sum_{k\in\mathbb N}
\mathbf{1}_{B_\rho((c(I_{k,j}),\ell(I_{k,j
})), C_1+R)}(x,y)\\
&\le\sum_{k\in\mathbb N}
\mathbf{1}_{B(c(I_{k,j}),C_2\ell(I_{k,j}))
}(x)
\mathbf{1}_{[C_2^{-1}\ell(I_{k,j}), C_2
\ell(I_{k,j})]}(y),
\end{align*}
which further implies that
\begin{align*}
\int_{0}^\infty \mathbf{1}_{(\bigcup_{
k\in\mathbb N}T(I_{k,j}))_R}(x,y) \,\frac{dy}{y}&\leq
\sum_{k\in \mathbb N}
\mathbf{1}_{B(c(I_{k,j}),C_2\ell(I_{k,j}))}(x)
\int_{0}^\infty\mathbf{1}_{[C_2^{-1}\ell(I_{k,j}), C_2\ell(I_{k,j})]}(y)
\,\frac{dy}{y}\\\nonumber&
\lesssim\sum_{k\in \mathbb N}
\mathbf{1}_{B(c(I_{k,j}),C_2\ell(I_{k,j}))}(x).
\end{align*}
From this and Assumption \ref{a1},
we infer that
\begin{align*}
\left\|\left\{\int_{0}^\infty \mathbf{1}_{
(\bigcup_{k\in\mathbb N}T(I_{k,j}))_R}
(\cdot,y)\,\frac{dy}{y}
\right\}_{j\in\mathbb Z_+}\right\|_{X^u}&
\lesssim\left\|\left\{
\sum_{k\in \mathbb N}
\mathbf{1}_{B(c(I_{k,j}),C_2\ell(I_{k,j}))}
\right\}_{j\in\mathbb Z_+}\right\|_{X^u}
\lesssim\left\|\left\{
\sum_{k\in \mathbb N}
\mathbf{1}_{I_{k,j}}\right\}_{j\in
\mathbb Z_+}\right\|_{X^u}\nonumber\\
&\lesssim\left\|\left\{\int_{0}^\infty
\mathbf{1}_{\bigcup_{k\in\mathbb N}
T(I_{k,j})}(\cdot,y)\,\frac{dy}{y}
\right\}_{j\in\mathbb Z_+}\right\|_{X^u}
\nonumber,
\end{align*}
which completes the proof of Lemma
\ref{a2.15s}.
\end{proof}

\begin{proof}[Proof of Theorem \ref{pppz23}]
We first show that
$\mathop\mathrm{\,dist\,}(f,
\Lambda_X^{s})_{\Lambda_s}\lesssim
\text{the right-hand side of \eqref{pofwj}}.$
Let $\varepsilon_1>\varepsilon$ and
$\delta$ be the same  as  in Theorem
\ref{thm-3adf}.
From Theorem \ref{cor-6-2-0}, Lemma
\ref{dda2f}, Assumption \ref{a1}, and
Lemma \ref{a2.15s},
it follows that there exist positive
constants $c\in(0,1)$ and $R\in(1,\infty)$
such that
\begin{align*}
\left\|\left\{
\sum_{I\in W_j(s, f, \varepsilon)}
\mathbf{1}_{I}
\right\}_{j\in\mathbb{Z}_+}
\right\|_{X}&\sim\left\|\left\{\int_{0}^\infty
\mathbf{1}_{T_{j}(s,f, \varepsilon)}(\cdot,y)
\,\frac{dy}{y}
\right\}_{j\in\mathbb Z_+}\right\|_{X^u}^u
\lesssim
\left\|\left\{\int_{0}^\infty \mathbf{1}_{
(S_{r,j+1}(s,f, c\varepsilon))_R}(
\cdot,y)\,\frac{dy}{y}
\right\}_{j\in\mathbb Z_+}\right\|_{X^u}^u\\
&\lesssim
\left\|\left\{\int_{0}^\infty \mathbf{1}_{
(S_{r,j+1}(s,f, c\varepsilon))_\delta}
(\cdot,y)\,\frac{dy}{y}
\right\}_{j\in\mathbb Z_+}\right\|_{X^u}^u,
\end{align*}
which, together with Theorem \ref{thm-3adf}
and Assumption \ref{a1},
further implies that
\begin{align*}
\left\|\left\{
\sum_{I\in W_j(s, f, \varepsilon)}
\mathbf{1}_{I}
\right\}_{j\in\mathbb{Z}_+}
\right\|_{X}&\lesssim
\left\|\left\{\int_{0}^\infty \mathbf{1}_{
S_{r,j}(s,f, c\varepsilon_1)}(\cdot,y)\,\frac{dy}{y}
\right\}_{j\in\mathbb Z_+}\right\|_{X^u}^u.
\end{align*}
This, combined with Corollary \ref{pp}(i),
implies the desired conclusion.

Now, we prove that
$\text{the right-hand side of \eqref{pofwj}}
\lesssim\mathop\mathrm{\,dist\,}(f,
\Lambda_X^{s})_{\Lambda_s}.$
Indeed, using Lemma \ref{dda2f}, Theorem
\ref{cor-6-7-0}, and
Lemma \ref{a2.15ss},
we conclude that  there exist positive
constants $c\in(0,1)$ and $R\in(1,\infty)$
such that
\begin{align*}
\left\|\left\{\int_{0}^\infty \mathbf{1}_{S_{r,j}(s,f, \varepsilon)}(\cdot,y)\,\frac{dy}{y}
\right\}_{j\in\mathbb Z_+}\right\|_{X^u}^u&
\lesssim
\left\|\left\{\int_{0}^\infty \mathbf{1}_{
(T_{j}^0(s,f, c\varepsilon))_R}(\cdot,y)\,
\frac{dy}{y}
\right\}_{j\in\mathbb Z_+}\right\|_{X^u}^u
\\
&\lesssim
\left\|\left\{\int_{0}^\infty \mathbf{1}_{T_{j}^0(s,f, c\varepsilon)}(\cdot,y)\,\frac{dy}{y}
\right\}_{j\in\mathbb Z_+}\right\|_{X^u}^u
\sim\left\|\left\{
\sum_{I\in W_j^0(s, f, c\varepsilon)}
\mathbf{1}_{I}
\right\}_{j\in\mathbb{Z}_+}
\right\|_{X},
\end{align*}
which completes the proof of Theorem
\ref{pppz23}.
\end{proof}
\begin{proof}[Proof of Theorem
\ref{pppz2a}]
Observe that $f\in\overline{\Lambda_X^{s}}^{
\Lambda_s}$ if and only if
$f\in \Lambda_s$
and $\mathop\mathrm{\,dist\,}(f, \Lambda_X^{s})_{
\Lambda_s}=0$.
From this and  Theorem \ref{pppz23},
it follows that, when $f\in \Lambda_s$,
$\mathop\mathrm{\,dist\,}(f, \Lambda_X^{s})_{
\Lambda_s}=0$
if and only if, for any $\varepsilon\in(0,\infty)$,
$$
\left\|\left\{ \left|\int_{0}^\infty\mathbf{1}_{
S_{r,j}(s,f, \varepsilon)}
(\cdot,y)\,\frac{dy}{y} \right|^u \right\}_{
j\in\mathbb Z_+}
\right\|_{X}^{\frac 1u} +\left\|\left\{
\sum_{I\in V_0(s, f, \varepsilon)}
\mathbf{1}_{I}\delta_{0,j}
\right\}_{j\in\mathbb{Z}_+}
\right\|_{X}
<\infty.
$$
Thus, (i) $\Longleftrightarrow$ (ii),
which completes the proof of Theorem \ref{pppz2a}.
\end{proof}

\subsection{Proof of Theorem \ref{thm-7-11}
}\label{sec:3-4}

\begin{proof}[Proof of Theorem \ref{thm-7-11}]
We first show that
$\mathop\mathrm{\,dist\,}(f, \Lambda_X^{s})_{\Lambda_s}\lesssim
\text{the right-hand side of \eqref{mop}}.$
Let $\varepsilon_1>\varepsilon$ and $\delta$ be
the same  as  in Theorem \ref{thm-3adf}.
From Theorem \ref{cor-6-2-0} and Lemma \ref{dda2f},
it follows that  there exist positive constants
$c\in(0,1)$ and $R\in(1,\infty)$ such that
\begin{align*}
\nu\left(\left\{T_{j}(s,f, \varepsilon)
\right\}_{j\in\mathbb Z_+}\right)
\lesssim
\nu\left(\left\{\left(S_{r,j+1}(s,f,
c\varepsilon)\right)_R
\right\}_{j\in\mathbb Z_+}\right)
\lesssim
\nu\left(\left\{\left(S_{r,j+1}(s,f,
c\varepsilon)\right)_\delta
\right\}_{j\in\mathbb Z_+}\right),
\end{align*}
which, together with Theorem \ref{thm-3adf},
Definition \ref{Debqf2s}, and Assumption \ref{pplp},
further implies that,
\begin{align*}
\mathrm{if}\ \nu\left(\left\{S_{r,j}(s,f,
c\varepsilon_1)
\right\}_{j\in\mathbb Z_+}\right)<\infty,\
\mathrm{then}\
\left\|\left\{
\sum_{I\in W_j(s, f, \varepsilon)}
\mathbf{1}_{I}
\right\}_{j\in\mathbb{Z}_+}
\right\|_{X}<\infty.
\end{align*}
This, combined with Corollary \ref{pp}(i),
further implies the desired conclusion.

Next, we show that
\begin{align}\label{modq}
\text{the right-hand side of \eqref{mop}
}&\lesssim\mathop\mathrm{
\,dist\,}\left(f, \Lambda_X^{s}\right)_{
\Lambda_s}.
\end{align}
Indeed, by Lemma \ref{dda2f} and Theorem
\ref{cor-6-7-0},
we find that  there exist positive constants
$c\in(0,1)$ and $R\in(1,\infty)$ such that
\begin{align*}
\nu\left(\left\{S_{r,j}(s,f, \varepsilon)
\right\}_{j\in\mathbb Z_+}\right)&\lesssim
\nu\left(\left\{\left(T_{j}^0(s,f, c\varepsilon)
\right)_R\right\}_{j\in\mathbb Z_+}\right)
\lesssim\nu
\left(\left\{T_{j}^0(s,f, c\varepsilon)
\right\}_{j\in\mathbb Z_+}\right),
\end{align*}
which, together with Assumption \ref{pplp}
and Definition \ref{Debqf2s},
further implies that,
\begin{align*}
\mathrm{if}\ \left\|\left\{
\sum_{I\in W_j(s, f, c\varepsilon)}
\mathbf{1}_{I}
\right\}_{j\in\mathbb{Z}_+}
\right\|_{X}<\infty,\ \mathrm{then}\
\nu\left(\left\{S_{r,j}(s,f, \varepsilon)
\right\}_{j\in\mathbb Z_+}\right)<\infty.
\end{align*}
From this, we deduce that \eqref{modq} holds.
This finishes the proof of Theorem \ref{thm-7-11}.
\end{proof}

\section{Applications to specific
spaces\label{sea8}}

In this section, we apply the main framework-based results
obtained in the previous sections
to some specific well-known function spaces,
precisely, to Sobolev spaces and $\mathrm{J}_s(\mathop\mathrm{\,bmo\,})$
in Subsection \ref{2.1}, to Besov spaces in Subsection \ref{5.2}, to Triebel--Lizorkin spaces in
Subsection \ref{5.3}, to Besov-type spaces in Subsection \ref{fwefm}, to Triebel--Lizorkin-type spaces in
Subsection \ref{fwefm2}.
These examples demonstrate that the results presented in
this article are of wide general applicability and operational
feasibility,
and they can definitely be applied to more function spaces;
see, for example, \cite{F1,F2,s011a}.

\subsection{Sobolev spaces and  $\mathrm{J}_s(\mathop\mathrm{\,bmo\,})$}\label{2.1}

We begin with Sobolev spaces $W^{1,p}$ with
$p\in(1,\infty)$.
To this end, we need to introduce more concept.
Let $q,p\in(0,\infty]$. The \emph{space}
$L^p(l^q)_{\mathbb{Z}_+}$ is
defined to be the set of all
$\mathbf{F}:=\{f_j\}_{j\in\mathbb{Z}_+}
\in \mathscr{M}_{\mathbb{Z}_+}$
such that
$\|\mathbf{F}\|_{L^p(l^q)_{\mathbb{Z}_+}}
:=\|[\sum_{j\in\mathbb{Z}_+}
|f_j|^q]^{\frac{1}{q}}\|_{L^p}<\infty,$
where the usual modification is made
when $q=\infty$.
The space $F_{\infty,q}(\mathbb{R}^n,\mathbb{Z}_+)$
is defined to
 be the set of all
$\mathbf{F}:=\{f_j\}_{j\in\mathbb{Z}_+}
\in \mathscr{M}_{\mathbb{Z}_+}$
such that
\begin{align*}
\|f\|_{F_{\infty ,q}(\mathbb{R}^n,\mathbb{Z}_+)}:=
\sup_{l\in\mathbb Z_+,m\in\mathbb Z^n}
\left\{\fint_{I_{l,m}}\sum_{j=l}^{\infty}
\left|f_j(x)\right|^q\,dx\right\}^{\frac 1q}<\infty,
\end{align*}
where the usual modification is made
when $q=\infty$.

From these definitions, it is easy to infer that
$L^p(l^q)_{\mathbb{Z}_+}$ and
$F_{\infty,q}(\mathbb{R}^n,\mathbb{Z}_+)$
are quasi-normed lattices of function sequences.

Here, and thereafter, the \emph{Hardy--Littlewood
maximal operator} $\mathcal M$
is defined by setting, for any $f\in L_{
{\mathop\mathrm{\,loc\,}}}^1$
and $x\in\mathbb{R}^n$,
\begin{equation*}
\mathcal M(f)(x):=\sup_{B\ni x}\frac1{|B|}
\int_B|f(y)|\,dy,
\end{equation*}
where the supremum is taken over all balls
$B\subset\mathbb{R}^n$
containing $x$.

\begin{lemma}\label{as123jl}
Let  $s,p\in(0,\infty)$ and $q\in(0,\infty]$.
Then $ L^p(l^q)_{\mathbb{Z}_+}$ satisfies
Assumption \ref{a1}.
\end{lemma}
\begin{proof}
We consider first the case $q<\infty$. Let
$X:= L^p(l^q)_{\mathbb{Z}_+}$ and $u:=1/q$.
It follows that
\begin{align*}
\left\|\left\{\sum_{k\in\mathbb N}\mathbf{1}_{
\beta B_{k,j}}
\right\}_{j\in\mathbb Z_+}\right\|_{
X^{\frac 1q}}&=
\left\|\sum_{j\in\mathbb{Z}_+}\sum_{k\in\mathbb N
}\mathbf{1}_{\beta B_{k,j}}
\right\|_{L^{\frac{p}{q}}}
\end{align*}
Let $\alpha\in(1,\infty)$ be such that
$\alpha p/q>1$.
Observe that
$
\mathbf{1}_{\beta B_{k,j}}
\lesssim [
\mathcal M(\mathbf{1}_{B_{k,j}})]^{\alpha}.
$
From this and the Fefferman--Stein vector-valued
inequality
 (see, for instance, \cite[Theorem 1(a)]{FS}),
we deduce that
\begin{align*}
\left\|\left\{\sum_{k\in\mathbb N}\mathbf{1}_{
\beta B_{k,j}}
\right\}_{j\in\mathbb Z_+}\right\|_{X^{\frac 1q}}&\lesssim\left\|\left\{\sum_{k\in\mathbb N}
\left[\mathcal M(\mathbf{1}_{B_{k,j}})\right]^{
\alpha}
\right\}_{j\in\mathbb Z_+}\right\|_{L^{\frac{p}{q}}}
\lesssim\left\|\left\{\sum_{j\in\mathbb{Z}_+}
\sum_{k\in \mathbb N}
\left[\mathcal M(\mathbf{1}_{B_{k,j}})
\right]^{\alpha}\right\}^{\frac 1{\alpha}}
\right\|_{L^{\frac{\alpha p}{q}}}^{\alpha}\nonumber\\
&\lesssim\left\|\left\{\sum_{j\in\mathbb{Z}_+}
\sum_{k\in \mathbb N}
\mathbf{1}_{B_{k,j}}\right\}^{\frac 1{\alpha}}
\right\|_{L^{\frac{\alpha p}{q}}}^{\alpha}
=\left\|\left\{\sum_{k\in\mathbb N}\mathbf{
1}_{B_{k,j}}
\right\}_{j\in\mathbb Z_+}\right\|_{
X^{\frac 1q}},\nonumber
\end{align*}
which completes the proof of Lemma \ref{as123}
in case $q<\infty$. In turn, the case
$q=\infty$ follows in a straightforward
manner from the fact that
 Fefferman--Stein vector-valued maximal
 inequality holds also on  $L^p(\ell^\infty)_{\mathbb{Z}_+}$, with $p\in(1,\infty)$, 
 and we may reduce to the
 case $p>1$ as before. This finishes the proof of Lemma \ref{as123jl}.
\end{proof}

Using Theorems \ref{thm-2-622a} and
\ref{pppz2a},  \cite[Corollary 2]{T20},
Lemmas \ref{as123jl} and \ref{as12312},
we obtain the following conclusions. We note that a 1-dimensional 
analogue of Theorem \ref{bvcaso45} was given by Nicolau and Soler i Gibert in  \cite{ns}.

\begin{theorem}\label{bvcaso45}
Let  $p\in(1,\infty)$. Assume that
the regularity parameter $L\in\mathbb N$
of the Daubechies wavelet system
$\{\psi_\omega\}_{\omega\in\Omega}$ satisfies that
$L\geq 2$. Let $\varphi$
be the same as in \eqref{fanofn}.
Then the following statements hold.
\begin{itemize}
\item[\rm(i)]
For any $f\in\Lambda_1$,
\begin{align*}
\mathop\mathrm{\,dist\,}\left(f, W^{1,p}\cap\Lambda_1
\right)_{\Lambda_1}
\sim\inf\left\{\varepsilon\in(0,\infty):\left\|\left[
\sum_{I \in W^0(1, f, \varepsilon)}\mathbf{1}_{I}
\right]^{\frac 12}\right\|_{L^p}<\infty\right\}
\end{align*}
with positive equivalence constants independent of $f$,
where
\begin{align}\label{fdam;}
W^0(1, f, \varepsilon):= \left\{ I\in\mathcal D:\max_{\{\omega\in\Omega:I_{\omega}
=I\}}|\langle f,\psi_{\omega}\rangle|
>\varepsilon |I|^{\frac 1n+\frac 12}\right\}.
\end{align}
\item[\rm(ii)]
For any $f\in \mathfrak{C}_0\cap \Lambda_1$,
 \begin{align*}
&\mathop\mathrm{\,dist\,}\left(f, W^{1,p}\cap\Lambda_1
\right)_{\Lambda_1}\\
&\quad\sim \inf\left\{\varepsilon\in(0,\infty):\left\|\left[\int_{0}^1 \mathbf{1}_{\{
(x,y)\in\mathbb{R}^n\times(0,1]:y^{-1}
\Delta_2 f(x, y)>\varepsilon\}}(\cdot,y)\,\frac {dy}{y}
\right]^{\frac 12}\right\|_{L^p}<\infty\right\}.
\end{align*}
with positive equivalence constants independent
of $f$.
\item[\rm(iii)]
$f\in\overline{W^{1,p}\cap\Lambda_1
}^{\Lambda_1}$
if and only if
$f\in \Lambda_1$ and,
for any $\varepsilon\in(0,\infty)$,
\begin{align*}
\left\|\left[
\sum_{I \in W^0(1, f, \varepsilon)}\mathbf{1}_{I}
\right]^{\frac 1q}\right\|_{L^p}<\infty,
\end{align*}
where $W^0(1, f, \varepsilon)$ is the same as
in \eqref{fdam;}.
\item[\rm(iv)]
$f\in\overline{W^{1,p}\cap\Lambda_1
}^{\Lambda_1}$
if and only if
$f\in \Lambda_1$ and,
for any $\varepsilon\in(0,\infty)$,
 \begin{align*}
\left\|\left[\int_{0}^1 \mathbf{1}_{\{(x,y)
\in\mathbb{R}^n\times(0,1]:y^{-1} \Delta_2 f(x, y)>\varepsilon\}}(\cdot,y)\,\frac {dy}{y}
\right]^{\frac 1q}\right\|_{L^p}<\infty
\end{align*}
and
\begin{align*}
\limsup_{k\in\mathbb{Z}^n,\;|k|\rightarrow\infty} \left|\int_{\mathbb{R}^n}\varphi(x-k)f(x)\,dx\right|=0.
\end{align*}
\end{itemize}
\end{theorem}

When $X:=L^p(l^2)_{\mathbb{Z}_+}$, for any
$f\in X$,
the
\emph{Lipschitz deviation degree} $\varepsilon_pf$ is defined by
setting
 $$\varepsilon_pf:=\inf\left\{\varepsilon
 \in(0,\infty):\left\|\left[\int_{0}^1 \mathbf{1}_{\{(x,y)
\in\mathbb{R}^n\times(0,1]:y^{-1} \Delta_2 f(x, y)>\varepsilon\}}(\cdot,y)\,\frac {dy}{y}
\right]^{\frac 12}\right\|_{L^p}<\infty\right\}.$$
Applying Theorem \ref{bvcaso45}, we can
characterize the Lipschitz
deviation degree $\varepsilon_p$ by using the distance in the
Lipschitz space
and we omit
the details of its proof.

\begin{theorem}\label{dsa124}
Let  $p\in(0,\infty)$
and $f\in \mathfrak{C}_0\cap \Lambda_1$.
\begin{itemize}
  \item[\rm (i)] If $\varepsilon\in(0,
  \varepsilon_pf)$, then
  $\|[\int_{0}^1 \mathbf{1}_{S_{2}(1,f,
  \varepsilon)}(\cdot,y)\,\frac {dy}{y}]^{\frac 12}\|_{L^p}=\infty$
 and,
  for any $(x,y)\in\mathbb{R}^n\times(0,1]$
  with
  $(x,y)\notin S_{2}(1,f, \varepsilon)$,
  $\frac{\Delta_2 f(x, y)}{y}\le\varepsilon_pf$;
  \item[\rm (ii)] If $\varepsilon\in(
  \varepsilon_pf,\infty)$,
  then $\|[\int_{0}^1 \mathbf{1}_{
  S_{2}(1,f, \varepsilon)}(\cdot,y)\,\frac {dy}{y}]^{\frac 12}
  \|_{L^p}<\infty$ and,
  for any
  $(x,y)\in S_{2}(1,f, \varepsilon)$,
  $\frac{\Delta_2 f(x, y)}{y}>\varepsilon_pf$.
\item[\rm (iii)] $
\mathop\mathrm{\,dist\,}(f, W^{1,p}\cap
\Lambda_1)_{\Lambda_1}
\sim \varepsilon_pf
$
with positive equivalence constants independent of $f$.
\item[\rm(iv)]
$f\in\overline{W^{1,p}\cap\Lambda_1
}^{\Lambda_1}$
if and only if
$f\in \Lambda_1$, $\varepsilon_pf=0$,
and
\begin{align*}
\limsup_{k\in\mathbb{Z}^n,\;|k|\rightarrow\infty} \left|\int_{\mathbb{R}^n}\varphi(x-k)f(x)\,dx\right|=0.
\end{align*}
\end{itemize}
\end{theorem}

As the other application of Theorem \ref{bvcaso45},
we have the following proper inclusions.

\begin{theorem}\label{fwqf2}
Let  $p\in(1,\infty)$.
Then
$W^{1,p}\cap\Lambda_1\subsetneqq
\overline{W^{1,p}\cap\Lambda_1}^{\Lambda_1}
\subsetneqq\Lambda_1$.
\end{theorem}
To prove Theorem \ref{fwqf2}, we need the
following conclusion, which
is about the wavelet characterization of
the Lipschitz  space (see, for instance,
\cite[Proposition 1.11]{T20}).

\begin{lemma}\label{asqw}
Let  $s\in(0,\infty)$. Assume that
the regularity parameter $L\in\mathbb N$ of
the Daubechies wavelet system $\{\psi_\omega
\}_{\omega\in\Omega}$ satisfies that
$L>s$. Then $f\in \Lambda_s$
if and only if its wavelet coefficients
$\{\langle f,\psi_{\omega}\rangle\}_{\omega\in
\Omega}$
satisfy that
\begin{align*}
\left\|\{\langle f,\psi_{\omega}\rangle
\}_{\omega\in \Omega}\right\|_{b^s_{\infty,
\infty}}:=\sup_{\omega\in \Omega}
|I_\omega|^{-\frac sn-\frac12}\left|\langle f,
\psi_{\omega}\rangle\right|
<\infty
\end{align*}
and
$
f=\sum_{\omega\in\Omega} \langle f,
\psi_{\omega}\rangle \psi_\omega
$
converges unconditionally in $\mathcal{S}'$.
Moreover, the representation is unique and
$
{\mathcal I}:f\mapsto\{\langle f,
\psi_{\omega}\rangle\}_{\omega\in \Omega}
$
is an isomorphic map of $\Lambda_s$ onto
$b^s_{\infty,\infty}$, where
$b^s_{\infty,\infty}$ is defined by setting
$$
b^s_{\infty,\infty}:=\left\{\{\lambda_\omega
\}_{\omega\in \Omega}\in\mathbb{C}^{\Omega}:\|\{\lambda_\omega\}_{\omega\in \Omega}\|_{b^s_{\infty,\infty}}<\infty\right\}.
$$
\end{lemma}
\begin{proof}[Proof of Theorem \ref{fwqf2}]
  We first show that  $\overline{W^{1,p}
  \cap\Lambda_1}^{\Lambda_1}
  \subsetneqq\Lambda_1$. Let
$$f_0:=\sum_{\{\omega\in \Omega: I_\omega\in
\mathcal{D},\ I_\omega\subset[0,1]^n\}}
|I_\omega|^{\frac sn+\frac 12}\psi_{\omega}.$$
From  Lemma \ref{asqw},
it follows that $f_0\in\Lambda_1$. Observe that,
for any
$I\in\mathcal{D}$ and $I\subset[0,1]^n$,
$$\max_{\{\omega\in\Omega:I_{\omega} =I\}}
\left|\int_{\mathbb{R}^n} f_0(x) \psi_{\omega}(x)\, dx\right|
= |I|^{\frac sn+\frac 12}.$$
By this, we conclude that
\begin{align*}
\left\|\left[
\sum_{I \in W^0(1, f_0, \frac12)}\mathbf{1}_{I}
\right]^{\frac 12}\right\|_{L^p}&=\left\|\left[
\sum_{\{I\subset[0,1]^n: I \in \mathcal{D}\}}\mathbf{1}_{I}
\right]^{\frac 12}\right\|_{L^p}=\infty,
\end{align*}
which, combined with Theorem \ref{bvcaso45},
further implies that
$
\mathop\mathrm{\,dist\,}\left(f_0,
\mathrm{J}_s(\mathop\mathrm{\,bmo\,})
\right)_{\Lambda_s}\gtrsim\frac12.
$
From this, it follows that $f_0\notin
\overline{W^{1,p}\cap\Lambda_s
}^{\Lambda_s}$.
Thus, $\overline{W^{1,p}\cap
\Lambda_s}^{\Lambda_s}
  \subsetneqq\Lambda_s$.

  Now, we show $W^{1,p}\cap\Lambda_s
  \subsetneqq
\overline{W^{1,p}\cap\Lambda_s}^{\Lambda_s}$.
Let
$f_0:=\sum_{j=0}^{\infty}\frac{1}{\sqrt{j}}
\sum_{\{\omega\in \Omega:I_\omega
\in\mathcal{D}_j,I_\omega\subset[0,1]^n\}}
|I_\omega|^{\frac sn+\frac 12}\psi_{\omega}$.
From  Lemma \ref{asqw},
it follows that $f_0\in\Lambda_s$. Observe that,
for any
$I\in\mathcal{D}_j$ and $I\subset[0,1]^n$,
$$\max_{\{\omega\in\Omega: I_\omega=I\}}\left|
\int_{\mathbb{R}^n} f_0(x) \psi_{\omega}(x)\, dx\right|
= \frac1{\sqrt{j}}|I|^{\frac sn+\frac 12}.$$
By this, we conclude that, for any $\varepsilon\in(0,\infty)$,
\begin{align*}
\left\|\left[
\sum_{I \in W^0(s, f_0, \varepsilon)}
\mathbf{1}_{I}
\right]^{\frac 12}\right\|_{L^p}&
\le\left\|\left[\sum_{\{j\in\mathbb{N}:j<
1/\varepsilon^2\}}\mathbf{1}_{[0,1]^n}
\right]^{\frac 12}\right\|_{L^p}<\infty,
\end{align*}
which, combined with Theorem \ref{bvcaso45},
further implies that $f_0\in\overline{
W^{1,p}\cap\Lambda_s}^{\Lambda_s}$.
Moreover,
\begin{align*}
\|f_0\|_{W^{1,p}}=\left\|\left[\sum_{j=0}^{\infty}
\frac1j\left|
\sum_{\{\omega\in \Omega: I_\omega\in
\mathcal{D}_j,\, I_\omega\subset[0,1]^n\}}
\mathbf{1}_{I_\omega}(x)
\right|
^2\,dx\right]^{\frac 12}\right\|_{L^p}
\geq \left\|\left[\sum_{j=0}^{\infty}
\frac1j
\mathbf{1}_{[0,1]^n}
\right]^{\frac 12}\right\|_{L^p}=\infty,
\end{align*}
which
further implies that $f_0\notin W^{1,p}\cap\Lambda_s$.
Thus, $ W^{1,p}\cap\Lambda_s\subsetneqq\overline{
W^{1,p}\cap\Lambda_s}^{\Lambda_s}
$,
which completes the proof of Theorem \ref{fwqf2}.
\end{proof}

We now turn to the space $\mathrm{J}_s(\mathop\mathrm{\,bmo\,})$ with $s\in(0,\infty)$.
We also need more concepts.
For any cube  $I\subset\mathbb{R}^n$, we
always use $\ell(I)$ to
denote its edge length and let $
\widehat I:=I\times [0, \ell(I))$.
For any given  measurable set $A\subset \mathbb{R}_+^{n+1}:=\mathbb{R}^n\times(0,
\infty)$, let
\begin{align}\label{asdf}
M(A):=\sup_{I\in \mathcal D} \frac 1 {|I|}
\int_I \left[\int_0^{\ell(I)}
\mathbf{1}_A(x,y) \,\frac {dy} {y}\right]\,
dx=\sup_{I\in\mathcal D} \frac 1 {|I|}
\iint_{\widehat I} \frac {\mathbf{1}_A(x,y)}y
\,dy\,dx.
\end{align}

Let $d\mu (x,x_{n+1}):=
\frac {dx\,dx_{n+1}} {x_{n+1}}.$
Since we assume that the edge length of
any cube in $\mathcal D$ is at most $1$,
we infer that, for any $A\subset \mathbb{R}_+^{n+1}$,
$M(A)=M\left(A\cap (\mathbb{R}^n\times (0, 1])\right).$
In other words, $M(A)<\infty$ is equivalent
to that $\mathbf{1}_{A\cap (\mathbb{R}^n
\times (0, 1])}d\mu$
is a Carleson measure.

\begin{lemma}\label{de1}
Let $A\subset \mathbb{R}^{n+1}_+$ be a
measurable set such that  $M(A)<\infty$, where $M$
is the same as in \eqref{asdf}. Assume that
there  exist constants $\delta, \delta'\in
(0, 1/10)$ such that, for any $\mathbf{z}\in A$,
\begin{equation*}
|B_\rho(\mathbf{z},\delta) \cap A |\ge \delta'
|B_\rho(\mathbf{z},\delta)|.
\end{equation*}
Then, for any $R\in(0,\infty)$, $M(A_R)<\infty$.
\end{lemma}

\begin{proof}
By the definition of $M$, we conclude that,
to show the present lemma,
it suffices to prove that, for any $I\in\mathcal{D}$,
\begin{align}\label{qer}
\mu(A_R\cap \widehat I)\lesssim |I|.
\end{align}
We first claim that
\begin{align}\label{qer2}
A_R\cap \widehat I\subset \left(
\widetilde{A}\right)_R,
\end{align}
where $\widetilde{A}:=\{ \mathbf{z}\in A:\rho(\mathbf{z}, \widehat I) < R\}$.
Indeed, for any $z\in A_R\cap \widehat I$,
$$
R\ge\rho(\mathbf{z}, A)
=\inf_{\mathbf{y}\in A\cap B_\rho(\mathbf{z},
R)}
\rho(\mathbf{z}, \mathbf{y}) = \inf\left\{
\rho(\mathbf{z}, \mathbf{y}):\mathbf{y}\in A, \  \rho\left(\mathbf{y},
\widehat I\right)<R\right\}
=\rho\left(\mathbf{z},\widetilde A\right).
$$
From this, it follows that $A_R\cap
\widehat I\subset (\widetilde{A})_R$.
Thus, the claim \eqref{qer2} holds.
Observe that,  for any $\mathbf{z}:=(z, z_{n+1})
\in A$,
\begin{align}\label{5-2}
\mu(A\cap B_\rho(\mathbf{z},\delta))&=\int_{
B_\rho(\mathbf{z},\delta)}
\frac {\mathbf{1}_A(x,x_{n+1})} {x_{n+1}}
\,dx\,dx_{n+1}
\sim \frac {|B_\rho(\mathbf{z},\delta)\cap A|}
{z_{n+1}}
\gtrsim \frac {|B_\rho(\mathbf{z},\delta)|}
{z_{n+1}}
\sim\mu(B_\rho(\mathbf{z},\delta)).
\end{align}
By the Zorn lemma, we conclude that there exists
a maximal subset $P$ of $\widetilde{A}$
such that
$\rho(\mathbf{z},\mathbf{z}' )\ge 2\delta$
for any $\mathbf{z},\mathbf{z}' \in P$ with
$\mathbf{z}\neq\mathbf{z}'$.
From this, we deduce that
$(\widetilde{A})_R\subset \bigcup_{\mathbf{z}
\in P} B_\rho(\mathbf{z}, R+2\delta)$.
Using this, the claim \eqref{qer2}, Lemma \ref{ddaf}(iv),
and \eqref{5-2},
we conclude that
\begin{align}\label{3-4}
\mu\left(A_R\cap\widehat I\right) &\leq
\sum_{\mathbf{z}\in P} \mu \left( B_\rho
(\mathbf{z}, R+2\delta)\right)
\lesssim \sum_{\mathbf{z}\in P}
\mu ( B_\rho (\mathbf{z}, \delta))
\lesssim\sum_{\mathbf{z}\in  P}
\mu\left(A \cap B_\rho (\mathbf{z}, \delta)\right)\nonumber\\
&=\mu\left(A \cap \left[
\bigcup_{\mathbf{z}\in P} B_\rho(\mathbf{z},
\delta) \right]\right).
\end{align}
We next claim that there exists a constant
$C_{(R)}\in(1,\infty)$ such that,
for any $\mathbf{z}\in\widetilde A$,
\begin{equation}\label{5-3}
B_\rho(\mathbf{z},\delta) \subset
\widehat {C_{(R)} I}.
\end{equation}
Assuming that this claim holds for the moment,
then, by \eqref{3-4},  we find that
\begin{align*}
\mu\left(A_R\cap\widehat I\right)&\leq \mu\left(A
\cap
\widehat {C_{(R)} I}\cap (\mathbb{R}^n\times
(0,1])\right)
+\mu\left(A \cap \widehat {C_{(R)} I}\cap
[\mathbb{R}^n\times (1,\infty)]\right)\\
&\lesssim|C_{(R)} I|+\sum_{\{j\in\mathbb Z_+: 2^j\leq C_{(R)}\}}\int_{2^j \ell(I)}^{2^{j+1} \ell(I)}
\int_{C_{(R)} I}\,dx\,\frac {dx_{n+1}}{x_{n+1}}
\lesssim|C_{(R)} I|\lesssim|I|,
\end{align*}
which is the desired conclusion \eqref{qer}.

Therefore, to
show the present lemma, it remains to prove
the above claim \eqref{5-3}.
Fix $\mathbf{z}=(z, z_{n+1})\in\widetilde A$.
By the definition of $\widetilde A$,
there exists
$\mathbf q=(q, q_{n+1})\in \widehat I$ such
that $\rho(\mathbf{z}, \mathbf  q) <R$,
which further implies that
$|z_{n+1}-q_{n+1}|^2
 \lesssim q_{n+1} z_{n+1}\lesssim \ell(I)
 z_{n+1}.$
From this, we infer that
\begin{align*}
z_{n+1}^2& \lesssim \max\left\{
(z_{n+1}-q_{n+1})^2,q_{n+1}^2\right\}
\lesssim\max\left\{\ell(I) z_{n+1},
\left[\ell(I)\right]^2\right\},
\end{align*}
which further implies that
$0<z_{n+1}\lesssim \ell(I).$
Using this, we obtain
$|q-z|^2 \lesssim z_{n+1}q_{n+1}\lesssim
\left[\ell(I)\right]^2,$
which further implies that
\begin{align}\label{daff}
|q-z|\lesssim \ell(I).
\end{align}
From \eqref{dafg2}, it follows that,
for any $\mathbf{p}=(p, p_{n+1})
\in B_\rho(\mathbf{z}, \delta)$,
$|p-z|\lesssim z_{n+1} \lesssim \ell(I),$
which, combined with \eqref{daff} and Lemma
\ref{ddaf}, further implies
that there exists a constant $C_{(R)}
\in(1,\infty)$, depending on $R$, such that
$|p-q|\leq  |p-z|+ |z-q|\le C_{(R)} \ell(I)$
and  $ 0<p_{n+1} \leq C_{(R)} \ell(I)$.
By this, we conclude that, for any $\mathbf{p}
\in B_\rho(\mathbf{z}, \delta)$,
$\mathbf{p}\in \widehat{C_{(R)} I}$ holds.
This finishes the proof of the above claim
\eqref{5-3} and hence Lemma \ref{de1}.
\end{proof}

\begin{lemma}\label{as12312}
Let $q\in(0,\infty]$.
Then $F_{\infty,q}(\mathbb{R}^n,\mathbb{Z}_+)$
satisfies Assumption \ref{pplp}.
\end{lemma}

\begin{proof}
For any $\{A_j\}_{j\in\mathbb{Z}_+}
\in\mathscr{P}_{\mathbb{Z}_+}(\mathbb{R}^{n+1}_+)$, let
$\nu(\{A_j\}_{j\in\mathbb Z_+}):=M(\bigcup_{j\in\mathbb Z_+}A_j)$.
Then we have,
for any $A\subset \mathcal D$,
\begin{align*}
\left\|\left\{
\sum_{I\in A\cap \mathcal D_j}
\mathbf{1}_{I}
\right\}_{j\in\mathbb{Z}_+}
\right\|_{F_{\infty,q}(\mathbb{R}^n,\mathbb{Z}_+)}^q&
=\sup_{l\in\mathbb Z_+,m\in\mathbb Z^n}
\fint_{I_{l,m}}\sum_{j=l}^{\infty}
\sum_{I\in A\cap \mathcal D_j}
\mathbf{1}_{I}(x)\,dx\sim
\sup_{J\in\mathcal D}
\frac{1}{|J|}
\sum_{\{I\in A:
I\subset J\}}|I|\\
&\sim\sup_{J\in\mathcal D}\frac{1}{|J|}\sum_{\{I\in A:
I\subset J\}}\int_{I}\int_{\frac12\ell(I)}^{
\ell(I)}\frac{1}{y}\,dy\,dx
\sim M\left(\bigcup_{I\in A}T(I)\right)\\
&= M\left(\bigcup_{j\in\mathbb Z_+}\bigcup_{
I\in A\cap \mathcal D_j}T(I)\right)
=\nu\left(\left\{\bigcup_{I\in A\cap \mathcal D_j}T(I)\right\}_{j\in\mathbb{Z}_+}\right).
\end{align*}
From this and Lemma \ref{de1}, it follows that
$F_{\infty,q}(\mathbb{R}^n,\mathbb{Z}_+)$ satisfies
Assumption \ref{pplp},
which completes the proof of Lemma \ref{as12312}.
\end{proof}

Using  Theorems \ref{thm-2-622a} and \ref{thm-7-117},
\cite[Corollary 2]{T20}, and Lemmas  \ref{as123jl} and
\ref{as12312}, we obtain the following conclusions.

\begin{theorem}\label{bvcaso3}
Let  $s\in(0,\infty)$ and $p\in(0,\infty)$.
Assume that $r\in\mathbb N$ with $r>s$
and that the regularity parameter $L\in\mathbb N$ of the
Daubechies wavelet system $\{\psi_\omega\}_{
\omega\in\Omega}$ satisfies that
$L>s$ and $L\geq r-1$. Let $\varphi$
be the same as in \eqref{fanofn}.
Then the following statements hold.
\begin{itemize}
\item[\rm(i)]
For any $f\in\Lambda_s$,
$$
\mathop\mathrm{\,dist\,}\left(f,  \mathrm{J}_s(\mathop\mathrm{\,bmo\,})\right)_{
\Lambda_s}\sim
\inf\left\{\varepsilon\in(0,\infty):M\left(
\bigcup_{I \in W^0(s, f, \varepsilon)} T(I)\right)
<\infty\right\}
$$
with positive equivalence constants independent
of $f$, where $W^0(s, f, \varepsilon)$ is the
same as in \eqref{fdam;}.
\item[\rm(ii)]
For any $f\in\Lambda_s$,
$\mathop\mathrm{\,dist\,}(f,  \mathrm{J}_s(\mathop\mathrm{\,bmo\,}))_{
\Lambda_s}\sim \inf\{\varepsilon\in(0,\infty):
M( S_r(s,f, \varepsilon))<\infty\}$
with positive equivalence constants independent of $f$.
\item[\rm(iii)]
$f\in\overline{\mathrm{J}_s(\mathop\mathrm{
\,bmo\,})
}^{\Lambda_s}$
if and only if
$f\in \Lambda_s$ and,
for any $\varepsilon\in(0,\infty)$,
$M(\bigcup_{I \in W^0(s, f, \varepsilon)}
T(I))<\infty.$
\item[\rm(iv)]
$f\in\overline{\mathrm{J}_s(\mathop\mathrm{\,bmo\,})
}^{\Lambda_s}$
if and only if
$f\in \Lambda_s$ and,
for any $\varepsilon\in(0,\infty)$,
$M( S_r(s,f, \varepsilon))<\infty$
and
\begin{align*}
\limsup_{k\in\mathbb{Z}^n,\;|k|\rightarrow\infty} \left|\int_{\mathbb{R}^n}\varphi(x-k)f(x)\,dx\right|=0.
\end{align*}
\end{itemize}
\end{theorem}

\begin{remark}
In both (i) and (ii) of Theorem \ref{bvcaso3}, if  $s\in(0,1]$,
then the conclusions are precisely
\cite[Theorems 1 and 3]{ss}, while
the other cases of Theorem \ref{bvcaso3} are completely new.
\end{remark}

Applying Theorem \ref{bvcaso3}, we characterize
the Lipschitz deviation
degree $\varepsilon_{r,s}$ in \eqref{1-8b} by using
the distance in the Lipschitz space
and we omit
the details of its proof.

\begin{theorem}\label{dsa12422}
Let  $s\in(0,\infty)$ and $p\in(0,\infty)$.
Assume that $r\in\mathbb N$ with $r>s$
and $f\in\Lambda_s$.
\begin{itemize}
  \item[\rm (i)] If $\varepsilon\in(0,\varepsilon_{r,s}f
  )$, then  $M( S_r(s,f, \varepsilon))=\infty$ and,
  for any $(x,y)\in\mathbb{R}^n\times(0,1]$ with
  $(x,y)\notin S_{r}(s,f, \varepsilon)$,
  $\frac{\Delta_r f(x, y)}{y^s}\le\varepsilon_{r,s}f $;
  \item[\rm (ii)] If $\varepsilon\in(\varepsilon_{r,s}f,
  \infty)$,
  then  $M( S_r(s,f, \varepsilon))<\infty$ and,
  for any
  $(x,y)\in S_{r}(s,f, \varepsilon)$,
  $\frac{\Delta_r f(x, y)}{y^s}>\varepsilon_{r,s}f$.
  \item[\rm (iii)] $
\mathop\mathrm{\,dist\,}\left(f, \mathrm{J}_s(\mathop\mathrm{\,bmo\,})\right)_{\Lambda_s}
\sim \varepsilon_{r,s}f
$
with positive equivalence constants independent of $f$.
\item[\rm(iv)]
$f\in\overline{\mathrm{J}_s(\mathop\mathrm{\,bmo\,})
}^{\Lambda_s}$
if and only if
$f\in \Lambda_s$ and $\varepsilon_{r,s}f=0$
and
\begin{align*}
\limsup_{k\in\mathbb{Z}^n,\;|k|\rightarrow\infty} \left|\int_{\mathbb{R}^n}\varphi(x-k)f(x)\,dx\right|=0.
\end{align*}
\end{itemize}
\end{theorem}

Using Theorem \ref{bvcaso3}, we can obtain the following
proper inclusions.

\begin{theorem}\label{fwqf}
If $s\in(0,\infty)$, then
$\mathrm{J}_s(\mathop\mathrm{\,bmo\,})\subsetneqq
\overline{\mathrm{J}_s(\mathop\mathrm{\,bmo\,})}^{
\Lambda_s}\subsetneqq\Lambda_s$.
\end{theorem}

\begin{proof}
  We first show that  $\overline{\mathrm{J}_s(
  \mathop\mathrm{\,bmo\,})
  }^{\Lambda_s}
  \subsetneqq\Lambda_s$. Let
$$f_0:=\sum_{\{\omega\in \Omega:I_\omega\in\mathcal{D},\,
I_\omega\subset[0,1]^n\}}
|I_\omega|^{\frac sn+\frac 12}\psi_{\omega}.$$ From
Lemma \ref{asqw},
it follows that $f_0\in\Lambda_s$. Observe that,
for any
$I\in\mathcal{D}$ and $I\subset[0,1]^n$,
$$\max_{\{\omega\in\Omega: I_\omega=I\}}\left|\int_{
\mathbb{R}^n} f_0(x) \psi_{\omega}(x)\, dx\right|
= |I|^{\frac sn+\frac 12}.$$
By this, we conclude that
\begin{align*}
M\left(\bigcup_{I \in W^0(s, f_0, \frac12)}
T(I)\right)&=
\sup_{I\in \mathcal D} \frac 1 {|I|} \int_I
\left[\int_0^{\ell(I)} \mathbf{1}_{\bigcup_{I
 \in W^0(s, f_0, \frac12)} T(I)}(x,y) \,
 \frac {dy} {y}\right]\, dx\\\notag
&\geq \int_{[0,1]^n} \left[\int_0^{1} \mathbf{1}_{
\bigcup_{I \in W^0(s, f_0, \frac12)} T(I)}(x,y)
\,\frac {dy} {y}\right]\, dx\\\notag
&=\int_{[0,1]^n} \left[\int_0^{1}  \,\frac {dy} {y}
\right]\, dx=\infty,
\end{align*}
which, combined with Theorem \ref{bvcaso3},
further implies that
$
\mathop\mathrm{\,dist\,}\left(f_0,  \mathrm{J}_s(\mathop\mathrm{\,bmo\,})\right)_{
\Lambda_s}\gtrsim\frac12.
$
From this, it follows that $f_0\notin\overline{
\mathrm{J}_s(\mathop\mathrm{
\,bmo\,})}^{\Lambda_s}$.
Thus, $\overline{\mathrm{J}_s(\mathop\mathrm{
\,bmo\,})}^{\Lambda_s}
  \subsetneqq\Lambda_s$.

Now, we show
$\mathrm{J}_s(\mathop\mathrm{
\,bmo\,})\subsetneqq
\overline{\mathrm{J}_s(\mathop\mathrm{
\,bmo\,})}^{\Lambda_s}$.
Let
$f_0:=\sum_{j=0}^{\infty}\frac{1}{\sqrt{j}}
\sum_{\{\omega\in \Omega: I_\omega\in\mathcal{D}_j,\,
I_\omega\subset[0,1]^n\}}
|I_\omega|^{\frac sn+\frac 12}\psi_{\omega}$.
From  Lemma \ref{asqw},
it follows that $f_0\in\Lambda_s$. Observe that,
for any
$I\in\mathcal{D}_j$ and $I\subset[0,1]^n$,
$$\max_{\{\omega\in\Omega: I_\omega=I\}}\left|
\int_{\mathbb{R}^n} f_0(x) \psi_{\omega}(x)\,
dx\right|
= \frac1{\sqrt{j}}|I|^{\frac sn+\frac 12}.$$
By this, we conclude that, for any $\varepsilon
\in(0,\infty)$,
\begin{align*}
M\left(\bigcup_{I \in W^0(s, f_0, \varepsilon)}
T(I)\right)&
=
\sup_{I\in \mathcal D} \frac 1 {|I|} \int_I \left[
\int_0^{\ell(I)} \mathbf{1}_{\bigcup_{I \in W^0(
s, f_0, \varepsilon)} T(I)}(x,y) \,\frac {dy} {y}
\right]\, dx\\\notag&
\le\sup_{x\in[0,1]^n} \left[\int_0^{1} \mathbf{1}_{
\bigcup_{I \in W^0(s, f_0, \varepsilon)} T(I)}(x,y)
\,\frac {dy} {y}\right]\\\notag&
\le\sup_{x\in[0,1]^n} \left[\int_{2^{-1/\varepsilon^2}
}^{1}  \,\frac {dy} {y}\right]<\infty,
\end{align*}
which, combined with Theorem \ref{bvcaso3},
further implies that $f_0\in\overline{\mathrm{J}_s(
\mathop\mathrm{\,bmo\,})
}^{\Lambda_s}$.
Moreover,
\begin{align*}
\|f_0\|_{\mathrm{J}_s(\mathop\mathrm{\,bmo\,})}
&\sim\sup_{l\in\mathbb Z_+,m\in\mathbb Z^n}
\left\{\fint_{I_{l,m}}\sum_{j=\ell}^{\infty}
\frac1j\left|
\sum_{\{\omega\in \Omega: I_\omega\in\mathcal{D}_j,\,
I_\omega\subset[0,1]^n\}}
\mathbf{1}_{I_\omega}(x)
\right|
^2\,dx\right\}^{\frac 12}\\\notag
&\geq \left\{\int_{[0,1]^n}\sum_{j=0}^{\infty}\frac1j
\left|
\sum_{\{\omega\in \Omega: I_\omega\in\mathcal{D}_j,\,
I_\omega\subset[0,1]^n\}}
\mathbf{1}_{I_\omega}(x)\right|
^2\,dx\right\}^{\frac 12}
\\\notag
&\sim \left\{\int_{[0,1]^n}\sum_{j=0}^{\infty}
\frac1j\left|
\sum_{\{I\in\mathcal{D}_j: I\subset[0,1]^n\}}
\mathbf{1}_{I}(x)\right|
^2\,dx\right\}^{\frac 12}= \left\{\int_{[0,1]^n}\sum_{j=0}^{\infty}
\frac1j
\,dx\right\}^{\frac 12}=\infty,
\end{align*}
which
further implies that $f_0\notin\mathrm{J}_s(
\mathop\mathrm{\,bmo\,})$.
Thus, $ \mathrm{J}_s(\mathop\mathrm{\,bmo\,})\subsetneqq
\overline{\mathrm{J}_s(\mathop\mathrm{\,bmo\,})}^{\Lambda_s}
$,
which completes the proof of Theorem \ref{fwqf}.
\end{proof}

\subsection{Besov spaces}\label{5.2}

Let $\{\phi_j\}_{j\in\mathbb{Z}_+}$
be the same as in \eqref{eq-phi1} and \eqref{eq-phik}.
As usual, we  denote by $\mathcal{S}({\mathbb{R}^n})$
the space of all Schwartz functions on ${\mathbb{R}^n}$,
equipped with the well-known topology
determined by a countable family of norms,
and by
$\mathcal{S}'({\mathbb{R}^n})$ its topological
dual space (that is,  the space of all tempered
distributions on $\mathbb{R}^n$),
equipped with the weak-$\ast$ topology.
Now, we present the concept of Besov spaces
as follows; see, for instance,
\cite[Definition 1.1]{T20}.

\begin{definition}\label{df-Triebel}
Let $p,q\in (0,\infty]$ and $s\in {\mathbb R}$.
Then the
\emph{Besov space} $B^{s}_{p,q}$
is defined to be the set of
all $f\in\mathcal S'$ such that
$$
\|f\|_{B^{s}_{p,q}}:=\left\{\sum_{j=0}^{\infty}
\left[2^{js}
\left\|\phi_j\ast f\right\|_{L^p}\right]^{q}
\right\}^{\frac 1q}<\infty,
$$
where the usual modification is made
when $q=\infty$.
\end{definition}

Let $q,p\in(0,\infty]$. The \emph{space
$l^q(L^p)_{\mathbb{Z}_+}$} is
defined to be the set of all
$\mathbf{F}:=\{f_j\}_{j\in\mathbb{Z}_+}
\in \mathscr{M}_{\mathbb{Z}_+}$
such that
$\|\mathbf{F}\|_{l^q(L^p)_{\mathbb{Z}_+}}
:=[\sum_{j\in\mathbb{Z}_+}
\|f_j\|^q_{L^p}]^{\frac{1}{q}}<\infty,$
where the usual modification is made
when $q=\infty$.
From the definition of the space $l^q(L^p
)_{\mathbb{Z}_+}$, it is easy to deduce that
$l^q(L^p)_{\mathbb{Z}_+}$ is a quasi-normed
lattice of function sequences.

The following lemma is about the wavelet
characterization of
Besov spaces (see, for instance,
\cite[Proposition 1.11]{T20}).

\begin{lemma}\label{as}
Let  $s\in(0,\infty)$ and $p,q\in(0,\infty]$.
Assume that
the regularity parameter $L\in\mathbb N$ of
the Daubechies wavelet system $\{\psi_\omega\}_{\omega\in\Omega}$
satisfies that
$L>\max\{s,n(\max\{\frac 1p,1\}-1)-s\}$.
Then $B^s_{p,q}\cap\Lambda_s
=\Lambda_X^{s}$
with equivalent quasi-norm, where $\Lambda_X^{s}$
is the
Daubechies $s$-Lipschitz $X$-based space with
$X:=l^q(L^p)_{\mathbb{Z}_+}$.
\end{lemma}

\begin{lemma}\label{as123}
Let  $p\in(0,\infty)$ and $q\in(0,\infty]$.
Then $l^q(L^p)_{\mathbb{Z}_+}$ satisfies Assumption \ref{a1}.
\end{lemma}
\begin{proof}
Let $\beta\in(1,\infty)$, $X:=l^q(L^p)_{\mathbb{Z}_+}$,
and $u:=1/p$.
It follows that
\begin{align*}
\left\|\left\{\sum_{k\in\mathbb N}\mathbf{1}_{\beta B_{k,j}}
\right\}_{j\in\mathbb Z_+}\right\|_{X^{\frac 1p}}&=\left[\sum_{j\in\mathbb{Z}_+}
\left\|\sum_{k\in\mathbb N}\mathbf{1}_{\beta B_{k,j}}
\right\|^{\frac{q}{p}}_{L^1}
\right]^{\frac{p}{q}}=\left[\sum_{j\in\mathbb{Z}_+}
\left\{\sum_{k\in\mathbb N}|\beta B_{k,j}|\right\}^{\frac{q}{p}}
\right]^{\frac{p}{q}}\lesssim\left[\sum_{j\in\mathbb{Z}_+}
\left\{\sum_{k\in\mathbb N}|B_{k,j}|\right\}^{\frac{q}{p}}
\right]^{\frac{p}{q}}\\
&=\left[\sum_{j\in\mathbb{Z}_+}
\left\|\sum_{k\in\mathbb N}\mathbf{1}_{B_{k,j}}
\right\|^{\frac{q}{p}}_{L^1}
\right]^{\frac{p}{q}}\lesssim\left\|\left\{\sum_{k\in\mathbb N}\mathbf{1}_{B_{k,j}}
\right\}_{j\in\mathbb Z_+}\right\|_{X^{\frac 1p}},
\end{align*}
which completes the proof of Lemma \ref{as123}.
\end{proof}

\begin{lemma}\label{as12k312}
Let $q\in(0,\infty)$.
Then $l^q(L^\infty)_{\mathbb{Z}_+}$ satisfies
Assumption \ref{pplp}.
\end{lemma}
\begin{proof}
For any $\{A_j\}_{j\in\mathbb{Z}_+}
\in\mathscr{P}_{\mathbb{Z}_+}(\mathbb{R}^{n+1}_+)$, let
$\nu(\{A_j\}_{j\in\mathbb Z_+}):=\sharp\{j\in\mathbb Z_+: A_j\neq\emptyset\}$.
Then we have,
for any $A\subset  \mathcal D$,
\begin{align*}
\left\|\left\{
\sum_{I\in A\cap \mathcal D_j}
\mathbf{1}_{I}
\right\}_{j\in\mathbb{Z}_+}
\right\|_{l^q(L^\infty)_{\mathbb{Z}_+}}^q
:=&\,\sum_{j\in\mathbb{Z}_+}
\left\|\sum_{I\in A\cap \mathcal D_j}
\mathbf{1}_{I}\right\|^q_{L^\infty}
=\sharp\left\{j\in\mathbb Z_+:\bigcup_{I\in A\cap
\mathcal D_j}T(I)\neq\emptyset\right\}\\
=&\,\nu\left(\left\{\bigcup_{I\in A\cap \mathcal D_j}T(I)\right\}_{j\in\mathbb{Z}_+}\right).
\end{align*}
From this, it follows that
$l^q(L^\infty)_{\mathbb{Z}_+}$ satisfies Assumption
 \ref{pplp},
which completes the proof of Lemma \ref{as12k312}.
\end{proof}

Using Theorem \ref{thm-7-11}, Lemmas \ref{as},
\ref{as123}, and \ref{as12k312}, Proposition \ref{pp},
and Corollary \ref{pp2}, we obtain the following conclusions.

\begin{theorem}\label{falsnlf}
Let  $s\in(0,\infty)$, $p\in(0,\infty)$, and
$q\in(0,\infty]$.
Assume that $r\in\mathbb N$ with $r>s$ and that
the regularity parameter $L\in\mathbb N$ of the
Daubechies wavelet system $\{\psi_\omega\}_{\omega\in\Omega}$ satisfies that
$L>\max\{s,n(\max\{\frac 1p,1\}-1)-s\}$ and $L
\geq r-1$. Let $\varphi$
be the same as in \eqref{fanofn}. Then the
following statements hold.
\begin{itemize}
\item[\rm(i)]
For any $f\in\Lambda_s$,
\begin{align*}
\mathop\mathrm{\,dist\,}\left(f, B^s_{p,q}\cap\Lambda_s
\right)_{\Lambda_s}
\sim\inf\left\{\varepsilon\in(0,\infty):\left\{\sum_{j=0}^\infty\left[
\sum_{I \in W^0(s, f, \varepsilon),I\in\mathcal{D}_j}|I|
\right]^{\frac qp}\right\}^{\frac 1q}<\infty\right\}
\end{align*}
with positive equivalence constants independent
of $f$.
\item[\rm(ii)]
For any $f\in\Lambda_s$,
\begin{align*}
\mathop\mathrm{\,dist\,}\left(f, B^s_{p,q}
\cap\Lambda_s
\right)_{\Lambda_s}
&\sim \inf\left\{\varepsilon\in(0,\infty):\left\{\sum_{j=0}^{\infty}
\left[\mu( S_{r,j}(s,f, \varepsilon))\right]^{
\frac qp}\right\}^{\frac 1q}<\infty\right\}\\
&\quad +\limsup_{k\in\mathbb{Z}^n,\;|k|\rightarrow\infty}\left|
\int_{\mathbb{R}^n}\varphi(x-k)f(x)\,dx\right|
\end{align*}
with positive equivalence constants independent of $f$.
\item[\rm(iii)]
For any $f\in\Lambda_s$ and $q\in(0,\infty)$,
$$
\mathop\mathrm{\,dist\,}\left(f, B^s_{\infty,q}
\right)_{\Lambda_s}\sim
\inf\left\{\varepsilon\in(0,\infty):\sharp\{j\in\mathbb Z_+:W_j(s, f, \varepsilon
)\neq\emptyset\}<\infty
\right\}
$$
with positive equivalence constants independent of $f$.
\item[\rm(iv)]
For any $f\in\Lambda_s$ and $q\in(0,\infty)$,
$$
\mathop\mathrm{\,dist\,}\left(f, B^s_{\infty,q}
\right)_{\Lambda_s}\sim
\inf\left\{\varepsilon\in(0,\infty):\sharp\{j\in\mathbb Z_+:S_{r,j}(s, f,
\varepsilon)\neq\emptyset\}<\infty
\right\}
$$
with positive equivalence constants independent of $f$.
\end{itemize}
\end{theorem}

\begin{remark}
To the best of our knowledge,
Theorem \ref{falsnlf} is completely new.
\end{remark}

Using Theorem \ref{pppz2a}, Corollary \ref{pp2}, and
Lemmas \ref{as}, \ref{as123}, and \ref{as12k312},
we obtain the following conclusions.

\begin{theorem}\label{fasnof6}
Let  $s\in(0,\infty)$, $p\in(0,\infty)$, and $q\in(0,\infty]$.
Assume that $r\in\mathbb N$ with $r>s$ and that
the regularity parameter $L\in\mathbb N$ of the
Daubechies wavelet system $\{\psi_\omega\}_{\omega\in\Omega}$ satisfies that
$L>\max\{s,n(\max\{\frac 1p,1\}-1)-s\}$ and $L\geq r-1$.
Let $\varphi$
be the same as in \eqref{fanofn}.
Then the following statements hold.
\begin{itemize}
\item[\rm(i)]
$f\in\overline{B^s_{p,q}\cap\Lambda_s
}^{\Lambda_s}$
if and only if
$f\in \Lambda_s$ and,
for any $\varepsilon\in(0,\infty)$,
\begin{align*}
\left\{\sum_{j=0}^\infty\left[
\sum_{I \in W^0(s, f, \varepsilon),I\in
\mathcal{D}_j}|I|
\right]^{\frac qp}\right\}^{\frac 1q}<\infty.
\end{align*}
\item[\rm(ii)]
$f\in\overline{B^s_{p,q}\cap\Lambda_s
}^{\Lambda_s}$
if and only if
$f\in \Lambda_s$ and,
for any $\varepsilon\in(0,\infty)$,
 \begin{align*}
\left\{\sum_{j=0}^{\infty}
\left[\mu( S_{r,j}(s,f, \varepsilon))
\right]^{\frac qp}\right\}^{\frac 1q}<\infty\
\mathrm{and}\
\limsup_{k\in\mathbb{Z}^n,\;|k|\rightarrow\infty}\left|
\int_{\mathbb{R}^n}\varphi(x-k)f(x)\,dx\right|=0.
\end{align*}
\item[\rm(iii)]
$f\in\overline{B^s_{\infty,q}
}^{\Lambda_s}$
if and only if
$f\in \Lambda_s$ and,
for any $\varepsilon\in(0,\infty)$,
$
\sharp\{j\in\mathbb Z_+:W_j(s, f, \varepsilon)\neq\emptyset\}<\infty.
$
\item[\rm(iv)]
$f\in\overline{B^s_{\infty,q}
}^{\Lambda_s}$
if and only if
$f\in \Lambda_s$ and,
for any $\varepsilon\in(0,\infty)$,
$
\sharp\{j\in\mathbb Z_+:S_{r,j}(s, f, \varepsilon)\neq\emptyset\}<\infty.
$
\end{itemize}
\end{theorem}

\begin{remark}
To the best of our knowledge,
Theorem \ref{fasnof6} is completely new.
\end{remark}

When $X:=l^q(L^p)_{\mathbb{Z}_+}$, for any
$f\in X$,
the
\emph{Lipschitz deviation degree} $\varepsilon_{p,q}f$ is defined by
setting
 $$\varepsilon_{p,q}f:=\inf\left\{\varepsilon
 \in(0,\infty):\left\{\sum_{j=0}^{\infty}
\left[\mu( S_{r,j}(s,f, \varepsilon))\right]^{
\frac qp}\right\}^{\frac 1q}<\infty\right\}.$$
When $X:=l^q(L^\infty)_{\mathbb{Z}_+}$, for any
$f\in X$,
the
\emph{Lipschitz deviation degree} $\varepsilon_{\infty,q}f$ is defined by
setting
 $$\varepsilon_{\infty,q}f:=\inf\left\{\varepsilon
 \in(0,\infty):\sharp\{j\in\mathbb Z_+:S_{r,j}(s, f,
\varepsilon)\neq\emptyset\}<\infty\right\}.$$
Applying Theorems \ref{falsnlf} and \ref{fasnof6}, we can
characterize the Lipschitz
deviation degree $\varepsilon_{p,q}$ for any
$p\in (0,\infty)$ in Theorem \ref{t5.20}
and for $p=\infty$ in Theorem \ref{t5.21} by using the distance in the
Lipschitz space
and we omit
the details of its proof.

\begin{theorem}\label{t5.20}
Let  $s,p\in(0,\infty)$,
$q\in(0,\infty]$,
and $f\in\Lambda_s\cap \mathfrak{C}_0$.
\begin{itemize}
  \item[\rm (i)] If $\varepsilon\in(0,
  \varepsilon_{p,q}f)$, then
  $\{\sum_{j=0}^{\infty}
[\mu( S_{r,j}(s, f,
\varepsilon))]^{
\frac qp}\}^{\frac 1q}=\infty$
and,
  for any $(x,y)\in\mathbb{R}^n\times(0,1]$
  with
  $(x,y)\notin S_{r}(s, f,
\varepsilon)$,
  $\frac {\Delta_r f(x, y)}{y^s}\le\varepsilon_{p,q}f$;
  \item[\rm (ii)] If $\varepsilon\in(
  \varepsilon_{p,q}f,\infty)$,
  then $\{\sum_{j=0}^{\infty}
[\mu( S_{r,j}(s,f, \varepsilon))]^{
\frac qp}\}^{\frac 1q}<\infty$ and,
  for any
  $(x,y)\in S_{r}(s, f,
\varepsilon)$,
  $\frac {\Delta_rf(x, y)}{y^s}>\varepsilon_{p,q}f$.
 \item[\rm (iii)] $
\mathop\mathrm{\,dist\,}(f, B^s_{p,q}\cap
\Lambda_s)_{\Lambda_s}
\sim \varepsilon_{p,q}f
$
with positive equivalence constants independent of $f$.
\item[\rm(iv)]
$f\in\overline{B^s_{p,q}\cap\Lambda_s
}^{\Lambda_s}$
if and only if
$f\in \Lambda_s$, $\varepsilon_{p,q}f=0$,
and
\begin{align*}
\limsup_{k\in\mathbb{Z}^n,\;|k|\rightarrow\infty} \left|\int_{\mathbb{R}^n}\varphi(x-k)f(x)\,dx\right|=0.
\end{align*}
\end{itemize}
\end{theorem}

\begin{theorem}\label{t5.21}
Let  $s,q\in(0,\infty)$
and $f\in\Lambda_s$.
\begin{itemize}
    \item[\rm (i)] If $\varepsilon\in(0,
  \varepsilon_{\infty,q}f)$, then
  $\sharp\{j\in\mathbb Z_+:S_{r,j}(s, f,
\varepsilon)\neq\emptyset\}=\infty$
and,
  for any $(x,y)\in\mathbb{R}^n\times(0,1]$
  with
  $(x,y)\notin S_{r}(s, f,
\varepsilon)$,
  $\frac{\Delta_r f(x, y)}{y^s}\le\varepsilon_{\infty,q}f$;
  \item[\rm (ii)] If $\varepsilon\in(
  \varepsilon_{\infty,q}f,\infty)$,
  then $\sharp\{j\in\mathbb Z_+:S_{r,j}(s, f,
\varepsilon)\neq\emptyset\}<\infty$ and,
  for any
  $(x,y)\in S_{r}(s, f,
\varepsilon)$,
  $\frac{\Delta_r f(x, y)}{y^s}>\varepsilon_{\infty,q}f$.
\item[\rm (iii)] $
\mathop\mathrm{\,dist\,}(f, B^s_{\infty,q})_{\Lambda_s}
\sim \varepsilon_{\infty,q}f
$
with positive equivalence constants independent of $f$.
\item[\rm(iv)]
$f\in\overline{B^s_{\infty,q}
}^{\Lambda_s}$
if and only if
$f\in \Lambda_s$, $\varepsilon_{\infty,q}f=0$.
\end{itemize}
\end{theorem}

Using Theorem \ref{falsnlf}, we can obtain the following
proper inclusions.

\begin{theorem}\label{fwqf33}
Let  $s\in(0,\infty)$, $p\in(0,\infty)$, and $q\in(0,\infty]$. Then
$B^s_{p,q}\cap\Lambda_s\subsetneqq
\overline{B^s_{p,q}\cap\Lambda_s
}^{\Lambda_s}\subsetneqq\Lambda_s$.
\end{theorem}

\begin{proof}
We first show that  $\overline{B^s_{p,q}\cap\Lambda_s
}^{\Lambda_s}\subsetneqq\Lambda_s$. Let
$f_0:=\sum_{\{\omega\in \Omega: I_\omega\in
\mathcal{D},\ I_\omega\subset[0,1]^n\}}
|I_\omega|^{\frac sn+\frac 12}\psi_{\omega}.$
From  Lemma \ref{asqw},
it follows that $f_0\in\Lambda_1$. Observe that,
for any
$I\in\mathcal{D}$ and $I\subset[0,1]^n$,
$$\max_{\{\omega\in\Omega: I_\omega=I\}}
\left|\int_{\mathbb{R}^n} f_0(x) \psi_{\omega}(x)\, dx\right|
= |I|^{\frac sn+\frac 12}.$$
By this, we conclude that
\begin{align*}
\left[\sum_{j\in\mathbb{Z}_+}
\left\|\sum_{\{I \in W^0(s, f_0, \frac12): I\in\mathcal{D}_j\}}
\mathbf{1}_{I}\right\|^q_{L^p}
\right]^{\frac{1}{q}}&=\left[\sum_{j\in\mathbb{Z}_+}
\left\|\sum_{\{I\subset[0,1]^n: I \in \mathcal{D}_j\}}
\mathbf{1}_{I}\right\|^q_{L^p}
\right]^{\frac{1}{q}}
=\infty,
\end{align*}
which, combined with Theorem \ref{falsnlf},
further implies that
$
\mathop\mathrm{\,dist\,}(f_0,
B^s_{p,q}\cap\Lambda_s
)_{\Lambda_s}\gtrsim\frac12.
$
From this, it follows that $f_0\notin
\overline{B^s_{p,q}\cap\Lambda_s
}^{\Lambda_s}$.
Thus, $\overline{B^s_{p,q}\cap\Lambda_s}^{\Lambda_s}
  \subsetneqq\Lambda_s$.

  Now, we show $B^s_{p,q}\cap\Lambda_s
  \subsetneqq
\overline{B^s_{p,q}\cap\Lambda_s}^{\Lambda_s}$.
Let
$f_0:=\sum_{j=0}^{\infty}j^{-\frac1q}
\sum_{\{\omega\in \Omega:I_\omega
\in\mathcal{D}_j,I_\omega\subset[0,1]^n\}}
|I_\omega|^{\frac sn+\frac 1q}\psi_{\omega}$.
From  Lemma \ref{asqw},
it follows that $f_0\in\Lambda_s$. Observe that,
for any
$I\in\mathcal{D}_j$ and $I\subset[0,1]^n$,
$$\max_{\{\omega\in\Omega:I_{\omega} =I\}}\left|
\int_{\mathbb{R}^n} f_0(x) \psi_{\omega}(x)\, dx\right|
= j^{-\frac1q}|I|^{\frac sn+\frac 1q}.$$
By this, we conclude that, for any $\varepsilon\in(0,\infty)$,
\begin{align*}
\left[\sum_{j\in\mathbb{Z}_+}
\left\|\sum_{\{I \in W^0(s, f_0, \varepsilon): I\in\mathcal{D}_j\}}
\mathbf{1}_{I}\right\|^q_{L^p}
\right]^{\frac{1}{q}}&
\le \left[\sum_{\{j\in\mathbb{Z}_+:j<
1/\varepsilon^q\}}
\left\|\mathbf{1}_{[0,1]^n}\right\|^q_{L^p}
\right]^{\frac{1}{q}}<\infty,
\end{align*}
which, combined with Theorem \ref{fasnof6},
further implies that $f_0\in\overline{
B^s_{p,q}\cap\Lambda_s}^{\Lambda_s}$.
Moreover,
\begin{align*}
\|f_0\|_{B^s_{p,q}}=\left[\sum_{j\in\mathbb{Z}_+}
j^{-1}\left\|
\sum_{\{\omega\in \Omega:I_\omega\in
\mathcal{D}_j,I_\omega\subset[0,1]^n\}}
\mathbf{1}_{I_\omega}
\right\|^q_{L^p}
\right]^{\frac{1}{q}}
\geq \left[\sum_{j\in\mathbb{Z}_+}
j^{-1}\left\|\mathbf{1}_{[0,1]^n}
\right\|^q_{L^p}\right]^{\frac{1}{q}}
=\infty,
\end{align*}
which
further implies that $f_0\notin B^s_{p,q}\cap\Lambda_s$.
Thus, $ B^s_{p,q}\cap\Lambda_s\subsetneqq\overline{
B^s_{p,q}\cap\Lambda_s}^{\Lambda_s}
$,
which completes the proof of Theorem \ref{fwqf33}.
\end{proof}

\begin{theorem}\label{fwqf331}
Let  $s\in(0,\infty)$ and $q\in(0,\infty)$. Then
$B^s_{\infty,q}\subsetneqq
\overline{B^s_{\infty,q}
}^{\Lambda_s}\subsetneqq\Lambda_s$.
\end{theorem}

\begin{proof}
  We first show that  $\overline{B^s_{\infty,q}
}^{\Lambda_s}\subsetneqq\Lambda_s$. Let
$f_0:=\sum_{\{\omega\in \Omega:I_\omega\in
\mathcal{D},I_\omega\subset[0,1]^n\}}
|I_\omega|^{\frac sn+\frac 12}\psi_{\omega}.$
From  Lemma \ref{asqw},
it follows that $f_0\in\Lambda_s$. Observe that,
for any
$I\in\mathcal{D}$ and $I\subset[0,1]^n$,
$$\max_{\{\omega\in\Omega:I_{\omega} =I\}}
\left|\int_{\mathbb{R}^n} f_0(x) \psi_{\omega}(x)\, dx\right|
= |I|^{\frac sn+\frac 12}.$$
By this, we conclude that
\begin{align*}
\left[\sum_{j\in\mathbb{Z}_+}
\left\|\sum_{\{I \in W^0(s, f_0, \frac12): I\in\mathcal{D}_j\}}
\mathbf{1}_{I}\right\|^q_{L^\infty}
\right]^{\frac{1}{q}}&=\left[\sum_{j\in\mathbb{Z}_+}
1
\right]^{\frac{1}{q}}
=\infty,
\end{align*}
which, combined with Theorem \ref{falsnlf},
further implies that
$
\mathop\mathrm{\,dist\,}(f_0,
B^s_{\infty,q}
)_{\Lambda_s}\gtrsim\frac12.
$
From this, it follows that $f_0\notin
\overline{B^s_{\infty,q}
}^{\Lambda_s}$.
Thus, $\overline{B^s_{\infty,q}}^{\Lambda_s}
  \subsetneqq\Lambda_s$.

  Now, we show $B^s_{\infty,q}
  \subsetneqq
\overline{B^s_{\infty,q}}^{\Lambda_s}$.
Let
$f_0:=\sum_{j=0}^{\infty}j^{-\frac1q}
\sum_{\{\omega\in \Omega:I_\omega
\in\mathcal{D}_j,I_\omega\subset[0,1]^n\}}
|I_\omega|^{\frac sn+\frac 1q}\psi_{\omega}$.
From  Lemma \ref{asqw},
it follows that $f_0\in\Lambda_s$. Observe that,
for any
$I\in\mathcal{D}_j$ and $I\subset[0,1]^n$,
$$\max_{\{\omega\in\Omega:I_{\omega} =I\}}\left|
\int_{\mathbb{R}^n} f_0(x) \psi_{\omega}(x)\, dx\right|
= j^{-\frac1q}|I|^{\frac sn+\frac 1q}.$$
By this, we conclude that, for any $\varepsilon\in(0,\infty)$,
\begin{align*}
\sharp\left\{j\in\mathbb Z_+:W_j(s, f, \varepsilon)\neq\emptyset\right\}&
\le
\sum_{\{j\in\mathbb{Z}_+: j<
1/\varepsilon^q\}}1<\infty,
\end{align*}
which, combined with Theorem \ref{fasnof6},
further implies that $f_0\in\overline{
B^s_{\infty,q}}^{\Lambda_s}$.
Moreover,
\begin{align*}
\|f_0\|_{B^s_{\infty,q}}=\left[\sum_{j\in\mathbb{Z}_+}
j^{-1}
\left\|
\sum_{\{\omega\in \Omega:I_\omega\in
\mathcal{D}_j,I_\omega\subset[0,1]^n\}}
\mathbf{1}_{I_\omega}
\right\|^q_{L^\infty}
\right]^{\frac{1}{q}}
\geq \left[\sum_{j\in\mathbb{Z}_+}
j^{-1}
\left\|
\mathbf{1}_{[0,1]^n}
\right\|^q_{L^\infty}
\right]^{\frac{1}{q}}
=\infty,
\end{align*}
which
further implies that $f_0\notin B^s_{\infty,q}$.
Thus, $ B^s_{\infty,q}\subsetneqq\overline{
B^s_{\infty,q}}^{\Lambda_s}
$,
which completes the proof of Theorem \ref{fwqf331}.
\end{proof}

\subsection{Triebel--Lizorkin spaces}\label{5.3}

Next, we present the concept of Triebel--Lizorkin
spaces
as follows; see, for instance,
\cite[Definition 1.1]{T20}.
 Let $\{\phi_j\}_{j\in\mathbb{Z}_+}$
be the same as in \eqref{eq-phi1} and \eqref{eq-phik}.

\begin{definition}\label{sanfnlk}
Let $q\in (0,\infty]$ and $s\in {\mathbb R}$.
\begin{itemize}
\item[(i)]
If $p\in (0,\infty)$, then the
\emph{Triebel--Lizorkin space} $F^{s}_{p,q}$
is defined to be the set of
all $f\in\mathcal S'$ such that
$$
\|f\|_{F^{s}_{p ,q}}:=
\left\|\left\{\sum_{j=0}^{\infty}
\left|2^{js}\phi_j\ast f\right|^q\right\}^{
\frac 1q}\right\|_{L^p}<\infty.
$$
\item[(ii)]
The
\emph{Triebel--Lizorkin space} $F^{s}_{\infty,q}$
is defined to be the set of
all $f\in\mathcal S'$ such that
$$
\|f\|_{F^{s}_{\infty ,q}}:=
\sup_{\genfrac{}{}{0pt}{}{J\in\mathbb Z_+,
M\in\mathbb Z^n}{I_{J,M}\in\mathcal{D}}}
\left\{\fint_{I_{J,M}}\sum_{j=J}^{\infty}
\left|2^{js}\phi_j\ast f(x)\right|^q\,dx
\right\}^{\frac 1q}<\infty.
$$
\end{itemize}
\end{definition}

The following lemma is about the wavelet
characterization of
Triebel--Lizorkin spaces (see, for instance,
\cite[Corollary 2]{T20}).

\begin{lemma}\label{asljk}
Let  $s\in(0,\infty)$ and $p,q\in(0,\infty]$.
Assume that
the regularity parameter $L\in\mathbb N$ of
the Daubechies wavelet system $\{\psi_\omega
\}_{\omega\in\Omega}$ satisfies that
$L>\max\{s,n(\max\{\frac 1p,\frac 1q,1\}-1)-s\}$.
Then $F^s_{p,q}\cap\Lambda_s
=\Lambda_X^{s}$,
where $\Lambda_X^{s}$ is the
Daubechies $s$-Lipschitz $X$-based space with
$X:=L^p(l^q)_{\mathbb{Z}_+}$ when $p\neq\infty$
or with $X:=F_{\infty,q}(\mathbb{R}^n,\mathbb{Z}_+)$
when $p=\infty$.
\end{lemma}

Using Theorem \ref{thm-7-11}, Lemmas
\ref{asljk}, \ref{as123jl}, and \ref{as12312},
Proposition \ref{pp},
and Corollary \ref{pp2}, we obtain the following conclusions.

\begin{theorem}\label{bvcaso}
Let  $s\in(0,\infty)$, $p\in(0,\infty)$, and $q\in(0,\infty]$.
Assume that $r\in\mathbb N$ with $r>s$
and that
the regularity parameter $L\in\mathbb N$ of
the Daubechies wavelet system $\{\psi_\omega
\}_{\omega\in\Omega}$ satisfies that
$L>\max\{s,n(\max\{\frac1p,\frac1q,1\}-1)-s\}$
and $L\geq r-1$. Let $\varphi$
be the same as in \eqref{fanofn}.
Then the following statements hold.
\begin{itemize}
\item[\rm(i)]
For any $f\in\Lambda_s$,
\begin{align*}
\mathop\mathrm{\,dist\,}\left(f, F^s_{p,q}\cap\Lambda_s
\right)_{\Lambda_s}
\sim\inf\left\{\varepsilon\in(0,\infty):\left\|\left[
\sum_{I \in W^0(s, f, \varepsilon)}\mathbf{1}_{I}
\right]^{\frac 1q}\right\|_{L^p}<\infty\right\}
\end{align*}
with positive equivalence constants independent of $f$.
\item[\rm(ii)]
For any $f\in\Lambda_s$,
\begin{align*}
\mathop\mathrm{\,dist\,}\left(f, F^s_{p,q}\cap\Lambda_s
\right)_{\Lambda_s}
&\sim \inf\left\{\varepsilon\in(0,\infty):
\left\|\left[\int_{0}^1 \mathbf{1}_{ S(s, f,
\varepsilon)}(\cdot,y)\,\frac {dy}{y}
\right]^{\frac 1q}\right\|_{L^p}<\infty\right\}\\
&\quad +
\limsup_{k\in\mathbb{Z}^n,\;|k|\rightarrow\infty}
\left|\int_{\mathbb{R}^n}\varphi(x-k)f(x)\,dx\right|.
\end{align*}
with positive equivalence constants independent of $f$.
\item[\rm(iii)]
For any $f\in\Lambda_s$,
$$
\mathop\mathrm{\,dist\,}\left(f, F^s_{
\infty,q}\right)_{\Lambda_s}\sim
\inf\left\{\varepsilon\in(0,\infty):M\left(\bigcup_{I \in W(s, f, \varepsilon)}
T(I)\right)<\infty\right\}
$$
with positive equivalence constants independent of $f$.
\item[\rm(iv)]
For any $f\in\Lambda_s$,
$\mathop\mathrm{\,dist\,}(f, F^s_{\infty,q})_{\Lambda_s}
\sim \inf\{\varepsilon\in(0,\infty):
M( S_r(s,f, \varepsilon))<\infty\}$
with positive equivalence constants independent of $f$.
\end{itemize}
\end{theorem}

\begin{remark}
In (iii) and (iv) of Theorem \ref{bvcaso}, if
$s\in(0,1]$ and $q=2$,
then the conclusions are precisely
\cite[Theorems 1 and 3]{ss}, while
Theorem \ref{bvcaso} in other
cases is completely new.
\end{remark}

Using Theorem \ref{pppz2a}, Corollary \ref{pp2}, and
Lemmas \ref{asljk}, \ref{as123jl}, and \ref{as12312},
we obtain the following conclusions.

\begin{theorem}\label{fasnof4}
Let  $s\in(0,\infty)$, $p\in(0,\infty)$, and
$q\in(0,\infty]$.
Assume that $r\in\mathbb N$ with $r>s$
and that
the regularity parameter $L\in\mathbb N$ of the
Daubechies wavelet system $\{\psi_\omega\}_{
\omega\in\Omega}$ satisfies that
$L>\max\{s,n(\max\{\frac1p,\frac1q,1\}-1)-s\}$
and $L\geq r-1$. Let $\varphi$
be the same as in \eqref{fanofn}.
Then the following statements hold.
\begin{itemize}
\item[\rm(i)]
$f\in\overline{F^s_{p,q}\cap\Lambda_s
}^{\Lambda_s}$
if and only if
$f\in \Lambda_s$ and,
for any $\varepsilon\in(0,\infty)$,
$\|[\sum_{I \in W^0(s, f, \varepsilon)}
\mathbf{1}_{I}]^{\frac 1q}\|_{L^p}<\infty.$
\item[\rm(ii)]
$f\in\overline{F^s_{p,q}\cap\Lambda_s
}^{\Lambda_s}$
if and only if
$f\in \Lambda_s$ and,
for any $\varepsilon\in(0,\infty)$,
 \begin{align*}
\left\|\left[\int_{0}^1 \mathbf{1}_{
S_r(s, f, \varepsilon)}(\cdot,y)\,\frac {dy}{y}
\right]^{\frac 1q}\right\|_{L^p}<\infty\ \mathrm{and}\
\limsup_{k\in\mathbb{Z}^n,\;|k|\rightarrow\infty} \left|\int_{\mathbb{R}^n}\varphi(x-k)f(x)\,dx\right|=0.
\end{align*}
\item[\rm(iii)]
$f\in\overline{F^s_{\infty,q}
}^{\Lambda_s}$
if and only if
$f\in \Lambda_s$ and,
for any $\varepsilon\in(0,\infty)$,
$M(\bigcup_{I \in W(s, f, \varepsilon)}
T(I))<\infty.$
\item[\rm(iv)]
$f\in\overline{F^s_{\infty,q}
}^{\Lambda_s}$
if and only if
$f\in \Lambda_s$ and,
for any $\varepsilon\in(0,\infty)$,
$M( S_r(s,f, \varepsilon))<\infty.$
\end{itemize}
\end{theorem}

\begin{remark}
To the best of our knowledge,
Theorem \ref{fasnof4} is completely new.
\end{remark}

Let  $p\in(0,\infty)$ and $q\in(0,\infty]$.
When $X:=L^p(l^q)_{\mathbb{Z}_+}$, for any
$f\in X$,
the
\emph{Lipschitz deviation degree} $\varepsilon_{p,q}f$ is defined by
setting
 $$\varepsilon_{p,q}f:=\inf\left\{\varepsilon
 \in(0,\infty):\left\|\left[\int_{0}^1 \mathbf{1}_{
S_r(s, f, \varepsilon)}(\cdot,y)\,\frac {dy}{y}
\right]^{\frac 1q}\right\|_{L^p}<\infty\right\}.$$
When $X:=F_{\infty,q}(\mathbb{R}^n,\mathbb{Z}_+)$, for any
$f\in X$,
the
\emph{Lipschitz deviation degree} $\varepsilon_{\infty,q}f$ is defined by
setting
 $$\varepsilon_{\infty,q}f:=\inf\left\{\varepsilon
 \in(0,\infty):M( S_r(s,f, \varepsilon))<\infty\right\}.$$
Applying Theorems \ref{bvcaso} and \ref{fasnof4}, we can
characterize the Lipschitz
deviation degree $\varepsilon_{p,q}$ for any $p\in (0,\infty)$
in Theorem \ref{t5.30} and for $p=\infty$ in Theorem \ref{t5.31}
by using the distance in the
Lipschitz space and we omit
the details of its proof.

\begin{theorem}\label{t5.30}
Let  $s,p\in(0,\infty)$, $q\in(0,\infty]$,
and $f\in\Lambda_s\cap \mathfrak{C}_0$.
\begin{itemize}
  \item[\rm (i)] If $\varepsilon\in(0,
  \varepsilon_{p,q}f)$, then
  $\|[\int_{0}^1 \mathbf{1}_{
S_r(s, f, \varepsilon)}(\cdot,y)\,\frac {dy}{y}
]^{\frac 1q}\|_{L^p}=\infty$
and,
  for any $(x,y)\in\mathbb{R}^n\times(0,1]$
  with
  $(x,y)\notin S_{r}(s, f,
\varepsilon)$,
  $\frac{\Delta_r f(x, y)}{y^s}\le\varepsilon_{p,q}f$;
  \item[\rm (ii)] If $\varepsilon\in(
  \varepsilon_{p,q}f,\infty)$,
  then $\|[\int_{0}^1 \mathbf{1}_{
S_r(s, f, \varepsilon)}(\cdot,y)\,\frac {dy}{y}
]^{\frac 1q}\|_{L^p}<\infty$ and,
  for any
  $(x,y)\in S_{r}(s, f,
\varepsilon)$,
  $\frac{\Delta_r f(x, y)}{y^s}>\varepsilon_{p,q}f$.
 \item[\rm (iii)] $
\mathop\mathrm{\,dist\,}(f, F^s_{p,q}\cap
\Lambda_s)_{\Lambda_s}
\sim \varepsilon_{p,q}f
$
with positive equivalence constants independent of $f$.
\item[\rm(iv)]
$f\in\overline{F^s_{p,q}\cap\Lambda_s
}^{\Lambda_s}$
if and only if
$f\in \Lambda_s$, $\varepsilon_{p,q}f=0$,
and
\begin{align*}
\limsup_{k\in\mathbb{Z}^n,\;|k|\rightarrow\infty} \left|\int_{\mathbb{R}^n}\varphi(x-k)f(x)\,dx\right|=0.
\end{align*}
\end{itemize}
\end{theorem}

\begin{theorem}\label{t5.31}
Let  $s,q\in(0,\infty)$
and $f\in\Lambda_s$.
\begin{itemize}
    \item[\rm (i)] If $\varepsilon\in(0,
  \varepsilon_{\infty,q}f)$, then
  $M( S_r(s,f, \varepsilon))=\infty$
and,
  for any $(x,y)\in\mathbb{R}^n\times(0,1]$
  with
  $(x,y)\notin S_{r}(s, f,
\varepsilon)$,
  $\frac{\Delta_r f(x, y)}{y^s}\le\varepsilon_{\infty,q}f$;
  \item[\rm (ii)] If $\varepsilon\in(
  \varepsilon_{\infty,q}f,\infty)$,
  then $M( S_r(s,f, \varepsilon))<\infty$ and,
  for any
  $(x,y)\in S_{r}(s, f,
\varepsilon)$,
  $\frac{\Delta_r f(x, y)}{y^s}>\varepsilon_{\infty,q}f$.
\item[\rm (iii)] $
\mathop\mathrm{\,dist\,}(f, F^s_{\infty,q})_{\Lambda_s}
\sim \varepsilon_{\infty,q}f
$
with positive equivalence constants independent of $f$.
\item[\rm(iv)]
$f\in\overline{F^s_{\infty,q}
}^{\Lambda_s}$
if and only if
$f\in \Lambda_s$, $\varepsilon_{\infty,q}f=0$.
\end{itemize}
\end{theorem}

\begin{theorem}\label{fwqf26}
Let  $s\in(0,\infty)$, $p\in(0,\infty)$, and
$q\in(0,\infty]$.
Then
$F^s_{p,q}\cap\Lambda_s\subsetneqq
\overline{F^s_{p,q}\cap\Lambda_s}^{\Lambda_s}
\subsetneqq\Lambda_s$.
\end{theorem}

\begin{proof}
We first show that
$\overline{F^s_{p,q}\cap\Lambda_s}^{\Lambda_s}
\subsetneqq\Lambda_s$. Let
$f_0:=\sum_{\{\omega\in \Omega:I_\omega\in
\mathcal{D},I_\omega\subset[0,1]^n\}}
|I_\omega|^{\frac sn+\frac 12}\psi_{\omega}.$
From  Lemma \ref{asqw},
it follows that $f_0\in\Lambda_s$. Observe that,
for any
$I\in\mathcal{D}$ and $I\subset[0,1]^n$,
$$\max_{\{\omega\in\Omega:I_{\omega} =I\}}
\left|\int_{\mathbb{R}^n} f_0(x) \psi_{\omega}(x)\, dx\right|
= |I|^{\frac sn+\frac 12}.$$
By this, we conclude that
\begin{align*}
\left\|\left[
\sum_{I \in W^0(s, f_0, \frac12)}\mathbf{1}_{I}
\right]^{\frac 1q}\right\|_{L^p}&=\left\|\left[
\sum_{\{I\subset[0,1]^n: I \in \mathcal{D}\}}\mathbf{1}_{I}
\right]^{\frac 1q}\right\|_{L^p}=\infty,
\end{align*}
which, combined with Theorem \ref{bvcaso},
further implies that
$
\mathop\mathrm{\,dist\,}\left(f_0,
\mathrm{J}_s(\mathop\mathrm{\,bmo\,})
\right)_{\Lambda_s}\gtrsim\frac12.
$
From this, it follows that $f_0\notin
\overline{F^s_{p,q}\cap\Lambda_s
}^{\Lambda_s}$.
Thus, $\overline{F^s_{p,q}\cap\Lambda_s}^{\Lambda_s}
  \subsetneqq\Lambda_s$.

  Now, we show $F^s_{p,q}\cap\Lambda_s
  \subsetneqq
\overline{F^s_{p,q}\cap\Lambda_s}^{\Lambda_s}$.
Let
$f_0:=\sum_{j=0}^{\infty}j^{-\frac1q}
\sum_{\{\omega\in \Omega:I_\omega
\in\mathcal{D}_j,I_\omega\subset[0,1]^n\}}
|I_\omega|^{\frac sn+\frac 12}\psi_{\omega}$.
From  Lemma \ref{asqw},
it follows that $f_0\in\Lambda_s$. Observe that,
for any
$I\in\mathcal{D}_j$ and $I\subset[0,1]^n$,
$$\max_{\{\omega\in\Omega:I_{\omega} =I\}}\left|
\int_{\mathbb{R}^n} f_0(x) \psi_{\omega}(x)\, dx\right|
= j^{-\frac1q}|I|^{\frac sn+\frac 12}.$$
By this, we conclude that, for any $\varepsilon\in(0,\infty)$,
\begin{align*}
\left\|\left[
\sum_{I \in W^0(s, f_0, \varepsilon)}
\mathbf{1}_{I}
\right]^{\frac 1q}\right\|_{L^p}&
\le\left\|\left[\sum_{\{j\in\mathbb{N}: j<
1/\varepsilon^q\}}\mathbf{1}_{[0,1]^n}
\right]^{\frac 1q}\right\|_{L^p}<\infty,
\end{align*}
which, combined with Theorem \ref{fasnof4},
further implies that $f_0\in\overline{
F^s_{p,q}\cap\Lambda_s}^{\Lambda_s}$.
Moreover,
\begin{align*}
\|f_0\|_{F^s_{p,q}}&=\left\|\left[\sum_{j=0}^{\infty}
\frac1j\left|
\sum_{\{\omega\in \Omega:I_\omega\in
\mathcal{D}_j,I_\omega\subset[0,1]^n\}}
\mathbf{1}_{I_\omega}(x)
\right|
^q\,dx\right]^{\frac 1q}\right\|_{L^p}
\geq \left\|\left[\sum_{j=0}^{\infty}
\frac1j
\mathbf{1}_{[0,1]^n}
\right]^{\frac 1q}\right\|_{L^p}=\infty,
\end{align*}
which
further implies that $f_0\notin F^s_{p,q}\cap\Lambda_s$.
Thus, $ F^s_{p,q}\cap\Lambda_s\subsetneqq\overline{
F^s_{p,q}\cap\Lambda_s}^{\Lambda_s}
$,
which completes the proof of Theorem \ref{fwqf26}.
\end{proof}

\begin{theorem}\label{fwqfqq}
If $s\in(0,\infty)$ and
$q\in(0,\infty)$, then
$F^s_{\infty,q}\subsetneqq
\overline{F^s_{\infty,q}}^{
\Lambda_s}\subsetneqq\Lambda_s$.
\end{theorem}

\begin{proof}
  We first show that  $\overline{F^s_{\infty,q}
  }^{\Lambda_s}
  \subsetneqq\Lambda_s$.  From Theorem \ref{fasnof4},
  we infer that $\overline{F^s_{\infty,q}
  }^{\Lambda_s}=\overline{\mathrm{J}_s(\mathop\mathrm{
\,bmo\,})}^{\Lambda_s}$,
which, combined with Theorem \ref{fwqf}, further
implies that $\overline{F^s_{\infty,q}
  }^{\Lambda_s}
  \subsetneqq\Lambda_s$.

  Now, we show $F^s_{\infty,q}\subsetneqq
\overline{F^s_{\infty,q}}^{\Lambda_s}$.
Let
$f_0:=\sum_{j=0}^{\infty}j^{-\frac1q}
\sum_{\{\omega\in \Omega:I_\omega\in\mathcal{D}_j,
I_\omega\subset[0,1]^n\}}
|I_\omega|^{\frac sn+\frac 12}\psi_{\omega}$.
From  Lemma \ref{asqw},
it follows that $f_0\in\Lambda_s$. Observe that,
for any
$I\in\mathcal{D}_j$ and $I\subset[0,1]^n$,
$$\max_{\{\omega\in\Omega:I_{\omega} =I\}}\left|
\int_{\mathbb{R}^n} f_0(x) \psi_{\omega}(x)\,
dx\right|
= j^{-\frac1q}|I|^{\frac sn+\frac 12}.$$
By this, we conclude that, for any $\varepsilon
\in(0,\infty)$,
\begin{align*}
M\left(\bigcup_{I \in W^0(s, f_0, \varepsilon)}
T(I)\right)&
=
\sup_{I\in \mathcal D} \frac 1 {|I|} \int_I \left[
\int_0^{\ell(I)} \mathbf{1}_{\bigcup_{I \in W^0(
s, f_0, \varepsilon)} T(I)}(x,y) \,\frac {dy} {y}
\right]\, dx\\\notag&
= \int_{[0,1]^n} \left[\int_0^{1} \mathbf{1}_{
\bigcup_{I \in W^0(s, f_0, \varepsilon)} T(I)}(x,y)
\,\frac {dy} {y}\right]\, dx\\\notag&
\le\int_{[0,1]^n} \left[\int_{2^{-1/\varepsilon^q}
}^{1}  \,\frac {dy} {y}\right]\, dx<\infty,
\end{align*}
which, combined with Theorem \ref{fasnof4},
further implies that $f_0\in\overline{F^s_{\infty,q}
}^{\Lambda_s}$.
Moreover,
\begin{align*}
\|f_0\|_{F^s_{\infty,q}}
&\sim\sup_{l\in\mathbb Z_+,m\in\mathbb Z^n}
\left\{\fint_{I_{l,m}}\sum_{j=\ell}^{\infty}
\frac1j\left|
\sum_{\{\omega\in \Omega:I_\omega\in\mathcal{D}_j,
I_\omega\subset[0,1]^n\}}
\mathbf{1}_{I_\omega}(x)
\right|
^q\,dx\right\}^{\frac 1q}\\\notag
&\geq \left\{\int_{[0,1]^n}\sum_{j=0}^{\infty}\frac1j
\left|
\sum_{\{\omega\in \Omega:I_\omega\in\mathcal{D}_j,
I_\omega\subset[0,1]^n\}}
\mathbf{1}_{I_\omega}(x)\right|
^q\,dx\right\}^{\frac 1q}
\\\notag
&\sim \left\{\int_{[0,1]^n}\sum_{j=0}^{\infty}
\frac1j\left|
\sum_{\{I\in\mathcal{D}_j: I\subset[0,1]^n\}}
\mathbf{1}_{I}(x)\right|
^q\,dx\right\}^{\frac 1q}
= \left\{\int_{[0,1]^n}\sum_{j=0}^{\infty}
\frac1j
\,dx\right\}^{\frac 1q}=\infty,
\end{align*}
which
further implies that $f_0\notin F^s_{\infty,q}$.
Thus, $ F^s_{\infty,q}\subsetneqq
\overline{F^s_{\infty,q}}^{\Lambda_s}
$,
which completes the proof of Theorem \ref{fwqfqq}.
\end{proof}

\subsection{Besov-type spaces}\label{fwefm}

The Besov-type and the Triebel--Lizorkin-type spaces
were introduced in \cite{yy08,yy10,ysy10} to
connect Triebel--Lizorkin spaces and $Q$ spaces,
and these spaces, together with Besov-type spaces,
were intensively investigated in
\cite{ht23,lsuyy12,syy10,yy10,yy13,ysy10,ysy20}.
Furthermore, they also have a close relation
with Besov--Morrey and Triebel--Lizorkin--Morrey
spaces introduced in \cite{ky94,tx},
and these spaces
are systematically studied
in \cite{HMS16,HS13,s08,s09,s10,st07,s011,s011a}.
We refer to \cite{syy10,yyz14,yzy15,yzy15-a}
for more studies on Besov-type and
Triebel--Lizorkin-type spaces and
to \cite{hl,hlms,hms23,hst23,hst23-2,lz,lz12,mps}
for more variants and their applications.

For any $\gamma:=(\gamma_1,\ldots,\gamma_n)\in \mathbb{Z}_+^n$,
let $|\gamma|:=\sum^n_{i=1}\gamma_i$.
For any $p$, $q\in\mathbb{R}$, let
$p\vee q:=\max\{p,q\}$ and
$p\wedge q:=\min\{p,q\}$;
furthermore, let
$p_+:=\max\{p,0\}$. Let $\mathscr{Q}$ be the set of
all dyadic cubes.
For any $Q\in\mathscr{Q}$,
let $j_Q:=-\log_2\ell(Q).$
Now, we recall the concept of Besov-type spaces.
Let $s\in\mathbb{R}$,
$p,q\in(0,\infty]$, and $\tau\in[0,\infty)$.
The \emph{space}
$(l^qL^p_\tau)_{\mathbb{Z}_+}$ is
defined to be the set of all
$\mathbf{f}:=\{f_j\}_{j\in\mathbb{Z}_+}\in
\mathscr{M}_{\mathbb{Z}_+}$ such that
\begin{align*}
\|\mathbf{f}\|_{(l^qL^p_\tau)_{\mathbb{Z}_+}}
:=\sup_{Q\in\mathscr{Q}}
\left\{\frac{1}{|Q|^\tau}\left[
\sum^\infty_{j=j_Q\vee 0}
\left\|f_j\right\|_{L^p(Q)}^q\right]^{\frac{1}{q}}
\right\}<\infty,
\end{align*}
where the usual modification is made
when $q=\infty$. Let $\{\phi_j\}_{j\in\mathbb{Z}_+}$
be the same as in \eqref{eq-phi1} and \eqref{eq-phik}.
The \emph{Besov-type
space $B^{s,\tau}_{p,q}$}
is defined to be the set of all
$f\in \mathcal{S}'$
such that
\begin{align*}
\|f\|_{B^{s,\tau}_{p,q}}:=
\left\|\left\{2^{js}\phi_j*f\right\}_{j\in\mathbb{Z}_+}
\right\|_{(l^qL^p_\tau)_{\mathbb{Z}_+}}<\infty.
\end{align*}
From the definition of the space $(l^qL^p_\tau
)_{\mathbb{Z}_+}$,
it is easy to infer that
$(l^qL^p_\tau)_{\mathbb{Z}_+}$ is a quasi-normed
lattice of function sequences.

The following lemma is about the wavelet characterization of
Besov-type spaces (see, for instance, \cite[Section 4.2]{ysy10}
or \cite[Theorem 6.3(i)]{lsuyy12}).

\begin{lemma}\label{as2}
Let  $s\in(0,\infty)$,
$p,q\in(0,\infty]$, and $\tau\in[0,\infty)$.
Assume that
the regularity parameter $L\in\mathbb N$ of
the Daubechies wavelet system
$\{\psi_\omega\}_{\omega\in\Omega}$ satisfies that
\begin{align}\label{4.8}
L> \left[s+n\left(\tau+\frac{1}{1\wedge p\wedge q}
-\frac{1}{p\vee 1}\right)+\frac{n}{p}\right]\vee
\left[-s+n\left(\tau+
\frac{1}{1\wedge p\wedge q}+\frac{1}{p}
-1\right)+2\frac{n}{p}\right].
\end{align}
Then $B^{s,\tau}_{p,q}\cap\Lambda_s
=\Lambda_X^{s}$,
where $\Lambda_X^{s}$ is the
Daubechies $s$-Lipschitz $X$-based space with $X:=(l^qL^p_\tau)_{\mathbb{Z}_+}$.
\end{lemma}

\begin{lemma}\label{pro822}
Let  $p\in(1,\infty)$, $q\in[1,\infty]$, $\tau\in[0,\frac1p)$,
and $r\in(1,\infty)$. Then there exists a positive constant $C$ such that,
for any $\{f_{j,\ell}\}_{j\in\mathbb{Z}_+}\in
\mathscr{M}_{\mathbb{Z}_+}$ with $\ell\in\mathbb{N}$,
$$
\left\|\left\{\left(\sum_{\ell=1}^\infty\left
[\mathcal{M}(f_{j,\ell})\right]^r\right)^\frac1r\right\}_{j\in\mathbb{Z}_+}
\right\|_{(l^qL^p_\tau)_{\mathbb{Z}_+}}\le
C\left\|
\left\{\left(\sum_{\ell=1}^\infty|f_{j,\ell}|^r\right)^\frac1r
\right\}_{j\in\mathbb{Z}_+}
\right\|_{(l^qL^p_\tau)_{\mathbb{Z}_+}}.
$$
\end{lemma}

\begin{proof}
Let $Q\in\mathscr{Q}$.
Let $B:=B(x_0,R)\subset\mathbb{B}$ with $x_0\in\mathbb{R}^n$ and $R\in(0,\infty)$ such that $Q\subset B \subset \sqrt{n}Q$.
For any given $j\in\mathbb{N}$, we decompose $f_{j,\ell}$ into
$
f_{j,\ell}=f_{j,\ell}^{(0)}+\sum_{k=1}^\infty f_{j,\ell}^{(k)},
$
where $f_{j,\ell}^{(0)}:=f_{j,\ell}\mathbf1_{2B}$ and, for any $k\in\mathbb{N}$,
$f_{j,\ell}^{(k)}:=f_{j,\ell}\mathbf1_{2^{k+1}B\setminus2^kB}$.
From this and the Minkowski inequality, we deduce that
$$
\left\{\sum_{\ell=1}^\infty
\left[\mathcal{M}(f_{j,\ell})
\right]^r\right\}^\frac1r\leq
\left\{\sum_{\ell=1}^\infty
\left[\mathcal{M}(f_{j,\ell}^{(0)})
\right]^r\right\}^\frac1r
+\sum_{k=1}^\infty\left\{
\sum_{\ell=1}^\infty\left[\mathcal{M}
(f_{j,\ell}^{(k)})\right]^r\right\}^\frac1r.
$$
For any given $\lambda\in(0,\infty)$, we find that
\begin{align*}
&\left\|\left\{\sum_{\ell=1}^\infty
\left[\mathcal{M}(f_{j,\ell})
\right]^r\right\}^\frac1r\right\|_{L^p(B)}\\
&\quad\leq
\left\|\left\{\sum_{\ell=1}^\infty[\mathcal{M}
(f_{j,\ell}^{(0)})(x)]^r\right\}^\frac1r\right\|_{L^p(B)}
+\left\|\sum_{k=1}^\infty\left\{\sum_{\ell=1
}^\infty[\mathcal{M}(f_{j,\ell}^{(k)})(x)]^r
\right\}^\frac1r\right\|_{L^p(B)}\\
&\quad\le
\left\|\left\{\sum_{\ell=1}^\infty[\mathcal{M}
(f_{j,\ell}^{(0)})(x)]^r\right\}^\frac1r\right\|_{L^p(B)}
+\sum_{k=1}^\infty\left\|
\left\{\sum_{\ell=1}^\infty\left[\mathcal{M}(f_{j,\ell}^{(k)})
\right]^r\right\}^\frac1r\right\|_{L^p(B)}\\
&\quad=:\mathrm{I}+\mathrm{II}.
\end{align*}
From the Fefferman--Stein vector-valued inequality (see \cite[Theorem 1(2)]{FS}), it follows that
\begin{align*}
\mathrm{I}\lesssim
\left\|\left[\sum_{\ell=1}^\infty
\left|f_{j,\ell}^{(0)}\right|^r\right]^\frac1r\right\|_{L^p(\mathbb{R}^n)}
=\left\|\left[\sum_{\ell=1}^\infty
\left|f_{j,\ell}\right|^r\right]^\frac1r\right\|_{L^p(2B)}.
\end{align*}
For any given $\ell,\ k\in\mathbb{N}$ and $x\in B$, it is easy to find that
\begin{align*}
\mathcal{M}(f_{j,\ell}^{(k)})(x)&=\sup_{t\in(0,\infty)}\frac1{|B(x,t)|}
\int_{B(x,t)}|f_{j,\ell}^{(k)}(y)|\,dy\\
&\sim\sup_{t>2^kR}\frac1{|B(x,t)|}\int_{B(x,t)}|f_{j,\ell}^{(k)}(y)|\,dy
\lesssim\left(2^{k}R\right)^{-n}\int_{\mathbb{R}^n}|f_{j,\ell}^{(k)}(y)|\,dy.
\end{align*}
From this and the Minkowski inequality, we deduce that, for any $k\in\mathbb{N}$ and $x\in B$,
\begin{align*}
\left\{\sum_{\ell=1}^\infty\left[\mathcal{M}(f_{j,\ell}^{(k)})
(x)\right]^r\right\}^\frac1r&\lesssim
\left\{\sum_{\ell=1}^\infty\left[
\left(2^kR\right)^{-n}\int_{\mathbb{R}^n}
\left|f_{j,\ell}^{(k)}(x)\right|\,dx
\right]^r\right\}^\frac1r\\
&\lesssim\left(2^kR\right)^{-n}
\int_{\mathbb{R}^n}\left[\sum_{\ell=1}^\infty
\left|f_{j,\ell}^{(k)}(x)
\right|^r\right]^\frac1r\,dx\lesssim
\left(2^kR\right)^{-\frac np}\left\|
\left[\sum_{\ell=1}^\infty\left|f_{j,\ell}
\right|^r\right]^\frac1r\right\|_{L^p(2^{k+1}B)},
\end{align*}
which implies that
\begin{align*}
\mathrm{II}\lesssim
\sum_{k=1}^\infty|B|^{\frac 1p}
\left(2^kR\right)^{-\frac np}\left\|\left[\sum_{\ell=1}^\infty
\left|f_{j,\ell}\right|^r\right]^\frac1r\right\|_{L^p(2^{k+1}B)}
\sim\sum_{k=1}^\infty2^{-\frac{kn}p}
\left\|\left[\sum_{\ell=1}^\infty
\left|f_{j,\ell}\right|^r\right]^\frac1r\right\|_{L^p(2^{k+1}B)}.
\end{align*}
By the estimates of $\mathrm{I}$ and $\mathrm{II}$, we conclude that
\begin{align*}
\frac{1}{|Q|^\tau}\left[
\sum^\infty_{j=j_Q\vee 0}
\left\|\left\{\sum_{\ell=1}^\infty
\left[\mathcal{M}(f_{j,\ell})
\right]^r\right\}^\frac1r\right\|_{L^p(Q)}^q\right]^{\frac{1}{q}}
&\lesssim
\sum_{k=0}^\infty
2^{-\frac{kn}{p}}2^{kn\tau}\frac{1}{|2^{k+1}B|^\tau}\left[
\sum^\infty_{j=j_Q\vee 0}
\left\|\left[\sum_{\ell=1}^\infty
\left|f_{j,\ell}\right|^r\right]^\frac1r
\right\|_{L^p(2^{k+1}B)}^q\right]^{\frac{1}{q}}
\\
&\lesssim\sum_{k=0}^\infty2^{-kn(\frac{1}{p}-\tau)}\left\|
\left[\sum_{j=1}^\infty\left|f_{j,\ell}\right|^r
\right]^\frac1r\right\|_{(l^qL^p_\tau)_{\mathbb{Z}_+}}\\
&\sim\left\|\left[\sum_{\ell=1}^\infty
\left|f_{j,\ell}\right|^r\right]^\frac1r\right\|_{(l^qL^p_\tau)_{\mathbb{Z}_+}}.
\end{align*}
This finishes the proof of Lemma \ref{pro822}.
\end{proof}

\begin{lemma}\label{as1123}
Let $\tau\in[0,1/p)$, $p\in(0,\infty)$, and
$q\in(0,\infty]$.
Then $(l^qL^p_\tau)_{\mathbb{Z}_+}$
satisfies Assumption \ref{a1}.
\end{lemma}
\begin{proof}
Let $X:=(l^qL^p_\tau)_{\mathbb{Z}_+}$ and $u:=1$.
Let $\alpha\in(1,\infty)$ be such that
$\alpha p>1$ and $\alpha q>1$.
From Lemma \ref{pro822},
It follows that
\begin{align*}
\left\|\left\{\sum_{k\in\mathbb N}\mathbf{1}_{
\beta B_{k,j}}
\right\}_{j\in\mathbb Z_+}\right\|^{\frac1{\alpha}}_{
X}&\lesssim
\sup_{Q\in\mathscr{Q}}
\left\{\frac{1}{|Q|^{\tau/\alpha}}\left[
\sum^\infty_{j=j_Q\vee 0}
\left\|\sum_{k\in\mathbb N}
\left[\mathcal M(\mathbf{1}_{B_{k,j}})
\right]^{\alpha}\right\|_{L^p(Q)}^{q}\right]^{
\frac{1}{\alpha q}}
\right\}\\
&\lesssim\sup_{Q\in\mathscr{Q}}
\left\{\frac{1}{|Q|^{\tau /\alpha}}\left[\sum_{j=j_Q\vee 0}^\infty
\left\|
\left\{
\sum_{k\in \mathbb N}
\left[\mathcal M(\mathbf{1}_{B_{k,j}})\right]^{\alpha}
\right\}^{\frac 1{\alpha}}
\right\|_{L^{\alpha p}(Q)}^{\alpha q}\right]^{
\frac{1}{\alpha q}}\right\}\\
&\lesssim\left\|
\left\{\sum_{k\in\mathbb N}\mathbf{1}_{B_{k,j}}
\right\}_{j\in\mathbb Z_+}\right\|_{X}^{\frac1{\alpha}},
\end{align*}
which completes the proof of Lemma \ref{as123}.
\end{proof}

Using Theorem \ref{thm-7-11}, Lemmas \ref{as2} and
\ref{as1123},  Proposition \ref{pp},
and Corollary \ref{pp2}, we obtain the following conclusions.

\begin{theorem}\label{fasnof}
Let $\tau\in[0,1/p)$, $s\in(0,\infty)$,
$p\in(0,\infty)$, and $q\in(0,\infty]$.
Assume that $r\in\mathbb N$ with $r>s$ and that
the regularity parameter $L\in\mathbb N$ of
the Daubechies wavelet system $\{\psi_\omega
\}_{\omega\in\Omega}$ satisfies \eqref{4.8}
and $L\geq r-1$. Let $\varphi$
be the same as in \eqref{fanofn}.
Then the following statements hold.
\begin{itemize}
\item[\rm(i)]
For any $f\in\Lambda_s$,
\begin{align*}
&\mathop\mathrm{\,dist\,}\left(f, B^{s,\tau
}_{p,q}\cap\Lambda_s
\right)_{\Lambda_s}\\
&\quad\sim
\inf\left\{\varepsilon\in(0,\infty):\sup_{Q\in\mathscr{Q}}
\frac{1}{|Q|^{\tau}}\left\{\sum_{j=j_Q\vee 0}^\infty\left[
\sum_{I \in W^0(s, f, \varepsilon)\cap\mathcal{D}_j}|I\cap Q|
\right]^{\frac qp}\right\}^{\frac 1q}<\infty\right\}
\end{align*}
with positive equivalence constants independent of $f$.
\item[\rm(ii)]
Let
$$\mathcal{A}:=\left\{\{a_k\}_{k\in
\mathbb N}\subset\mathbb Z^n:\left\|\left\{
\sum_{k\in\mathbb N}
\mathbf{1}_{I_{0,a_k}}
\right\}_{j\in\mathbb{Z}_+}
\right\|_{(l^qL^p_\tau)_{\mathbb{Z}_+}}=\infty
\right\}.$$
Then,
for any $f\in\Lambda_s$,
 \begin{align*}
&\mathop\mathrm{\,dist\,}\left(f, B^{s,
\tau}_{p,q}\cap\Lambda_s(
\mathbb{R}^n)\right)_{\Lambda_s}\\
&\quad\sim \inf\left\{\varepsilon\in(0,\infty):\sup_{Q\in\mathscr{Q}}
\frac{1}{|Q|^{\tau}}\left\{\sum_{j=j_Q\vee 0}^{\infty}
\left[\int_Q\int_{0}^\infty
\mathbf{1}_{S_{r,j}(s,f, \varepsilon)}
(x,y)\,\frac{dy}{y} \,dx\right]^{\frac qp}\right\}^{\frac1q}<\infty\right\}\\&\quad\quad+
\sup_{\{a_k\}_{k\in\mathbb N}\in\mathcal{A}}
\liminf_{k\rightarrow\infty}\left|\int_{\mathbb{R}^n}
\varphi(x-a_k)f(x)\,dx\right|
\end{align*}
with positive equivalence constants independent of $f$.
\end{itemize}
\end{theorem}

\begin{remark}
To the best of our knowledge,
Theorem \ref{fasnof} is completely new.
\end{remark}

Using Theorem \ref{pppz2a}, Corollary \ref{pp2}, and
Lemmas \ref{as2} and \ref{as1123},
we obtain the following conclusions.

\begin{theorem}\label{fasnof3}
Let $\tau\in[0,\infty)$,  $s\in(0,\infty)$,
$p\in(0,\infty)$, and $q\in(0,\infty]$.
Assume that $r\in\mathbb N$ with $r>s$ and that
the regularity parameter $L\in\mathbb N$ of
the Daubechies wavelet system $\{\psi_\omega
\}_{\omega\in\Omega}$ satisfies \eqref{4.8}
and $L\geq r-1$. Let $\varphi$
be the same as in \eqref{fanofn}.
Then the following statements hold.
\begin{itemize}
\item[\rm(i)]
$f\in\overline{B^{s,\tau}_{p,q}\cap
\Lambda_s
}^{\Lambda_s}$
if and only if
$f\in \Lambda_s$ and,
for any $\varepsilon\in(0,\infty)$,
\begin{align*}
\sup_{Q\in\mathscr{Q}}
\frac{1}{|Q|^{\tau}}\left\{\sum_{j=j_Q
\vee 0}^\infty\left[
\sum_{I \in W^0(s, f, \varepsilon)\cap\mathcal{D}_j}|I\cap Q|
\right]^{\frac qp}\right\}^{\frac 1q}<\infty.
\end{align*}
\item[\rm(ii)]
Let
\begin{align}\label{dsfafg}
\mathcal{A}:=\left\{\{a_k\}_{k\in
\mathbb N}\subset\mathbb Z^n:\left\|\left\{
\sum_{k\in\mathbb N}
\mathbf{1}_{I_{0,a_k}}
\right\}_{j\in\mathbb{Z}_+}
\right\|_{(l^qL^p_\tau)_{\mathbb{Z}_+}}=\infty
\right\}.
\end{align}
Then $f\in\overline{B^{s,\tau}_{p,q}
\cap\Lambda_s
}^{\Lambda_s}$
if and only if
$f\in \Lambda_s$ and,
for any $\varepsilon\in(0,\infty)$,
 \begin{align*}
\sup_{Q\in\mathscr{Q}}
\frac{1}{|Q|^{\tau}}\left\{\sum_{j=j_Q
\vee 0}^{\infty}
\left[\int_Q\int_{0}^\infty
\mathbf{1}_{S_{r,j}(s,f, \varepsilon)}
(x,y)\,\frac{dy}{y} \,dx\right]^{\frac qp}\right\}^{\frac1q}<\infty
\end{align*}
and
\begin{align*}
\sup_{\{a_k\}_{k\in\mathbb N}\in\mathcal{A}}
\liminf_{k\rightarrow\infty}\left|\int_{\mathbb{R}^n}
\varphi(x-a_k)f(x)\,dx\right|=0.
\end{align*}
\end{itemize}
\end{theorem}

\begin{remark}
To the best of our knowledge,
Theorem \ref{fasnof3} is completely new.
\end{remark}

Let  $\tau\in[0,1/p)$, $p\in(0,\infty)$, and
$q\in(0,\infty]$.
When $X:=(l^qL^p_\tau)_{\mathbb{Z}_+}$, for any
$f\in X$,
the
\emph{Lipschitz deviation degree} $\varepsilon^{\tau}_{p,q}f$ is defined by
setting
 $$\varepsilon^{\tau}_{p,q}f:=\inf\left\{\varepsilon
 \in(0,\infty):\left\|\left[\int_{0}^1 \mathbf{1}_{
S_r(s, f, \varepsilon)}(\cdot,y)\,\frac {dy}{y}
\right]^{\frac 1q}\right\|_{L^p}<\infty\right\}.$$
Applying Theorems \ref{fasnof} and \ref{fasnof3}, we can
characterize the Lipschitz
deviation degree $\varepsilon^{\tau}_{p,q}$ by using the distance in the
Lipschitz space
and we omit
the details of its proof.

\begin{theorem}
Let  $s,p,q\in(0,\infty)$
and $f\in\Lambda_s\cap \mathfrak{C}_0$.
\begin{itemize}
  \item[\rm (i)] If $\varepsilon\in(0,
  \varepsilon^{\tau}_{p,q}f)$, then
  $\sup_{Q\in\mathscr{Q}}
\frac{1}{|Q|^{\tau}}\{\sum_{j=j_Q
\vee 0}^\infty[
\sum_{I \in W^0(s, f, \varepsilon)\cap\mathcal{D}_j}|I\cap Q|
]^{\frac qp}\}^{\frac 1q}=\infty$
and,
  for any $(x,y)\in\mathbb{R}^n\times(0,1]$
  with
  $(x,y)\notin S_{r}(s, f,
\varepsilon)$,
  $\frac{\Delta_r f(x, y)}{y^s}\le\varepsilon^{\tau}_{p,q}f$;
  \item[\rm (ii)] If $\varepsilon\in(
\varepsilon^{\tau}_{p,q}f,\infty)$,
  then $\sup_{Q\in\mathscr{Q}}
\frac{1}{|Q|^{\tau}}\{\sum_{j=j_Q
\vee 0}^\infty[
\sum_{I \in W^0(s, f, \varepsilon)\cap\mathcal{D}_j}|I\cap Q|
]^{\frac qp}\}^{\frac 1q}<\infty$ and,
  for any
  $(x,y)\in S_{r}(s, f,
\varepsilon)$,
  $\frac{\Delta_r f(x, y)}{y^s}>\varepsilon^{\tau}_{p,q}f$.
 \item[\rm (iii)] $
\mathop\mathrm{\,dist\,}(f, B^{s,\tau}_{p,q}\cap
\Lambda_s)_{\Lambda_s}
\sim \varepsilon^{\tau}_{p,q}f
$
with positive equivalence constants independent of $f$.
\item[\rm(iv)]
$f\in\overline{B^{s,\tau}_{p,q}\cap\Lambda_s
}^{\Lambda_s}$
if and only if
$f\in \Lambda_s$, $\varepsilon^{\tau}_{p,q}f=0$,
and
\begin{align*}
\sup_{\{a_k\}_{k\in\mathbb N}\in\mathcal{A}}
\liminf_{k\rightarrow\infty}\left|\int_{\mathbb{R}^n}
\varphi(x-a_k)f(x)\,dx\right|=0,
\end{align*}
where $\mathcal{A}$ is the same as in \eqref{dsfafg}.
\end{itemize}
\end{theorem}

\begin{theorem}\label{fwqf234}
Let $\tau\in[0,1/p)$,  $s\in(0,\infty)$,
$p\in(0,\infty)$, and $q\in(0,\infty]$. Then
$B^{s,\tau}_{p,q}\cap\Lambda_s\subsetneqq
\overline{B^{s,\tau}_{p,q}\cap\Lambda_s
}^{\Lambda_s}\subsetneqq\Lambda_s$.
\end{theorem}

\begin{proof}
We first show that $\overline{B^{s,\tau}_{p,q}
\cap\Lambda_s}^{\Lambda_s}
\subsetneqq\Lambda_s$. Let
$f_0:=\sum_{\{\omega\in \Omega:I_\omega\in\mathcal{D},
I_\omega\subset[0,1]^n\}}
|I_\omega|^{\frac sn+\frac 12}\psi_{\omega}.$ From
Lemma \ref{asqw},
it follows that $f_0\in\Lambda_s$. Observe that,
for any
$I\in\mathcal{D}$ and $I\subset[0,1]^n$,
$$\max_{\{\omega\in\Omega:I_{\omega} =I\}}\left|\int_{
\mathbb{R}^n} f_0(x) \psi_{\omega}(x)\, dx\right|
= |I|^{\frac sn+\frac 12}.$$
By this, we conclude that
\begin{align*}
\sup_{Q\in\mathscr{Q}}
\frac{1}{|Q|^{\tau}}\left\{\sum_{j=j_Q
\vee 0}^\infty\left[
\sum_{I \in W^0(s, f_0, \frac12)\cap\mathcal{D}_j}|I\cap Q|
\right]^{\frac qp}\right\}^{\frac 1q}
&\geq \left\{\sum_{j=j_Q
\vee 0}^\infty1\right\}^{\frac 1q}=\infty,
\end{align*}
which, combined with Theorem \ref{fasnof},
further implies that
$
\mathop\mathrm{\,dist\,}
(f_0,  B^{s,\tau}_{p,q}\cap\Lambda_s)_{
\Lambda_s}\gtrsim\frac12.
$
From this, it follows that $f_0\notin\overline{
B^{s,\tau}_{p,q}\cap\Lambda_s}^{\Lambda_s}$.
Thus, $\overline{B^{s,\tau}_{p,q}\cap\Lambda_s}^{\Lambda_s}
  \subsetneqq\Lambda_s$.

Now, we show $B^{s,\tau}_{p,q}\cap\Lambda_s\subsetneqq
\overline{B^{s,\tau}_{p,q}\cap\Lambda_s}^{\Lambda_s}$.
Let
$f_0:=\sum_{j=0}^{\infty}j^{-\frac1q}
\sum_{\{\omega\in \Omega: I_\omega\in\mathcal{D}_j,\,
I_\omega\subset[0,1]^n\}}
|I_\omega|^{\frac sn+\frac 12}\psi_{\omega}$.
From  Lemma \ref{asqw},
it follows that $f_0\in\Lambda_s$. Observe that,
for any
$I\in\mathcal{D}_j$ and $I\subset[0,1]^n$,
$$\max_{\{\omega\in\Omega:I_{\omega} =I\}}\left|
\int_{\mathbb{R}^n} f_0(x) \psi_{\omega}(x)\,
dx\right|
=j^{-\frac1q}|I|^{\frac sn+\frac 12}.$$
By this, we conclude that, for any $\varepsilon
\in(0,\infty)$,
\begin{align*}
\sup_{Q\in\mathscr{Q}}
\frac{1}{|Q|^{\tau}}\left\{\sum_{j=j_Q
\vee 0 }^\infty\left[
\sum_{I \in W^0(s, f_0, \varepsilon)\cap\mathcal{D}_j}|I\cap Q|
\right]^{\frac qp}\right\}^{\frac 1q}&
\lesssim\left\{\sum_{0 \leq j<1/\varepsilon^q}1\right\}^{\frac 1q}
<\infty,
\end{align*}
which, combined with Theorem \ref{fasnof3},
further implies that $f_0\in\overline{B^{s,\tau}_{p,q}\cap\Lambda_s
}^{\Lambda_s}$.
Moreover,
\begin{align*}
\|f_0\|_{B^{s,\tau}_{p,q}}
&\sim\sup_{Q\in\mathscr{Q}}
\left\{\frac{1}{|Q|^\tau}\left[
\sum^\infty_{j=j_Q\vee 0}
\left\|\sum_{\{\omega\in\Omega:I_\omega\in\mathcal{D}_j,\,
I_\omega\subset[0,1]^n \}}
\mathbf{1}_{I_\omega}\right\|_{L^p(Q)}^q\right]^{\frac{1}{q}}
\right\}
\gtrsim\left\{\sum_{j=0}^{\infty}\int_{[0,1]^n}\frac1j
\,dx\right\}^{\frac 1q}
=\infty,
\end{align*}
which
further implies that $f_0\notin B^{s,\tau}_{p,q}\cap\Lambda_s$.
Thus, $ B^{s,\tau}_{p,q}\cap\Lambda_s\subsetneqq
\overline{B^{s,\tau}_{p,q}\cap\Lambda_s}^{\Lambda_s}
$,
which completes the proof of Theorem \ref{fwqf234}.
\end{proof}

\subsection{Triebel--Lizorkin-type spaces}\label{fwefm2}

Next, we present the concept of Triebel--Lizorkin spaces
as follows; see, for instance,
\cite{yy08,yy10,ysy10}.
 Let $\{\phi_j\}_{j\in\mathbb{Z}_+}$
be the same as in \eqref{eq-phi1} and \eqref{eq-phik}.
Let $s\in\mathbb{R}$, $p\in(0,\infty)$,
$q\in(0,\infty]$, and $\tau\in[0,\infty)$.
The \emph{space}
$(L^p_\tau l^q)_{\mathbb{Z}_+}$ is
defined to be the set of all
$\mathbf{f}:=\{f_j\}_{j\in\mathbb{Z}_+}\in
\mathscr{M}_{\mathbb{Z}_+}$ such that
\begin{align*}
\|\mathbf{f}\|_{(L^p_\tau l^q)_{\mathbb{Z}_+}}
:=\sup_{Q\in\mathscr{Q}}
\left\{\frac{1}{|Q|^\tau}
\left\|\left(\sum^\infty_{j=j_Q\vee 0}
\left|f_j\right|^q\right)^{\frac{1}{q}}
\right\|_{L^p(Q)}\right\}<\infty,
\end{align*}
where the usual modification is made when $q=\infty$.
The \emph{Triebel--Lizorkin-type
space $F^{s,\tau}_{p,q}$}
is defined to be the set of all
$f\in \mathcal{S}'$
such that
$\|f\|_{F^{s,\tau}_{p,q}}:=
\|\{2^{js}\phi_j*f\}_{j\in\mathbb{Z}_+}\|_{(L^p_\tau l^q)_{\mathbb{Z}_+}}<\infty.$
From these definitions, it is easy to deduce that
$(L^p_\tau l^q)_{\mathbb{Z}_+}$
is a quasi-normed lattice of function sequences.

The following lemma is about the wavelet
characterization of
Triebel--Lizorkin-type spaces  (see, for instance,
\cite[Section 4.2]{ysy10} or \cite[Theorem 6.3(i)]{lsuyy12}).

\begin{lemma}\label{as3}
Let  $s\in(0,\infty)$,
$p\in(0,\infty)$,
$q\in(0,\infty]$, and $\tau\in[0,\infty)$.
Assume that the regularity parameter $L\in\mathbb N$
of the Daubechies
wavelet system $\{\psi_\omega\}_{\omega\in\Omega}$ satisfies that
\begin{align}\label{4.9}
L&> \left[s+n\left(\tau+\frac{1}{1\wedge p\wedge q}
-\frac{1}{p\vee 1}\right)+\frac{n}{p\wedge q}\right]\nonumber\\
&\quad\vee \left[-s+n\left(\tau+
\frac{1}{1\wedge p\wedge q}+\frac{1}{p}
-1\right)+2\frac{n}{p\wedge q}\right].
\end{align}
Then $F^{s,\tau}_{p,q}\cap
\Lambda_s
=\Lambda_X^{s}$,
where $\Lambda_X^{s}$ is the
Daubechies $s$-Lipschitz $X$-based space
with $X:=(L^p_\tau l^q)_{\mathbb{Z}_+}$.
\end{lemma}

\begin{lemma}\label{pro82223}
Let  $p\in(1,\infty)$, $q\in(1,\infty)$, $\tau\in[0,\frac1p)$,
and $r\in(1,\infty)$. Then there exists a positive constant $C$ such that,
for any $\{f_{j,\ell}\}_{j\in\mathbb{Z}_+}\in
\mathscr{M}_{\mathbb{Z}_+}$ with $\ell\in\mathbb{N}$,
$$
\left\|\left\{\left(\sum_{\ell=1}^\infty\left
[\mathcal{M}(f_{j,\ell})\right]^r\right)^\frac1r\right\}_{j\in\mathbb{Z}_+}
\right\|_{(L^p_\tau l^q)_{\mathbb{Z}_+}}\le
C\left\|
\left\{\left(\sum_{\ell=1}^\infty|f_{j,\ell}|^r\right)^\frac1r
\right\}_{j\in\mathbb{Z}_+}
\right\|_{(L^p_\tau l^q)_{\mathbb{Z}_+}}.
$$
\end{lemma}

\begin{proof}
Let $Q\in\mathscr{Q}$.
Let $B:=B(x_0,R)\subset\mathbb{B}$ with $x_0\in\mathbb{R}^n$ and $R\in(0,\infty)$
be such that $Q\subset B \subset \sqrt{n}Q$.
For any given $j\in\mathbb{N}$, we decompose $f_{j,\ell}$ into
$
f_{j,\ell}=f_{j,\ell}^{(0)}+\sum_{k=1}^\infty f_{j,\ell}^{(k)},
$
where $f_{j,\ell}^{(0)}:=f_{j,\ell}\mathbf1_{2B}$ and, for any $k\in\mathbb{N}$,
$f_{j,\ell}^{(k)}:=f_{j,\ell}\mathbf1_{2^{k+1}B\setminus2^kB}$.
Repeating the proof of
the estimations of $\mathrm{I}$ and $\mathrm{II}$ in the proof
of Lemma \ref{pro822} with
the Fefferman--Stein vector-valued inequality
replaced by \cite[Theorem 2]{ba},
we conclude that
\begin{align*}
\frac{1}{|Q|^\tau}
\left\|\left[
\sum^\infty_{j=j_Q\vee 0}\left\{\sum_{\ell=1}^\infty
\left[\mathcal{M}(f_{j,\ell})
\right]^r\right\}^\frac qr\right]^{\frac{1}{q}}\right\|_{L^p(Q)}
&\lesssim
\sum_{k=0}^\infty
2^{-\frac{kn}{p}}2^{kn\tau}\frac{1}{|2^{k+1}B|^\tau}
\left\|\left[
\sum^\infty_{j=j_Q\vee 0}\left[\sum_{\ell=1}^\infty
\left|f_{j,\ell}\right|^r\right]^\frac qr
\right]^{\frac{1}{q}}\right\|_{L^p(2^{k+1}B)}
\\
&\lesssim\sum_{k=0}^\infty2^{-kn(\frac{1}{p}-\tau)}\left\|
\left[\sum_{\ell=1}^\infty\left|f_{j,\ell}\right|^r
\right]^\frac1r\right\|_{(L^p_\tau l^q)_{\mathbb{Z}_+}}\\
&\sim\left\|\left[\sum_{\ell=1}^\infty
\left|f_{j,\ell}\right|^r\right]^\frac1r\right\|_{(L^p_\tau l^q)_{\mathbb{Z}_+}}.
\end{align*}
This finishes the proof of Lemma \ref{pro82223}.
\end{proof}

\begin{lemma}\label{as123j2}
Let  $p\in(0,\infty)$, $\tau\in[0,1/p)$,
and $q\in(0,\infty)$.
Then $(L^p_\tau l^q)_{\mathbb{Z}_+}$ satisfies
Assumption \ref{a1}.
\end{lemma}
\begin{proof}
Let $X:= (L^p_\tau l^q)_{\mathbb{Z}_+}$ and
$u:=1/q$. It follows that
\begin{align*}
\left\|\left\{\sum_{k\in\mathbb N}\mathbf{1}_{\beta B_{k,j}}
\right\}_{j\in\mathbb Z_+}\right\|_{X^{\frac1q}}&=
\sup_{Q\in\mathscr{Q}}
\left\{\frac{1}{|Q|^{\tau q}}
\left\|\sum^\infty_{j=j_Q\vee 0}
\sum_{k\in\mathbb N}\mathbf{1}_{\beta B_{k,j}}
\right\|_{L^{\frac{p}{q}}(Q)}\right\}.
\end{align*}
Let $\alpha\in(1,\infty)$ be such that
$\alpha p/q>1$.
Observe that
$
\mathbf{1}_{\beta B_{k,j}}
\lesssim [\mathcal M(\mathbf{1}_{
B_{k,j}})]^{\alpha}.
$
From this and Lemma \ref{pro82223},
we infer that
\begin{align*}
\left\|\left\{\sum_{k\in\mathbb N}\mathbf{1}_{\beta B_{k,j}}
\right\}_{j\in\mathbb Z_+}\right\|_{X^{\frac1q}}&\lesssim
\left\|\left\{\sum_{k\in\mathbb N}
\left[\mathcal M(\mathbf{1}_{B_{k,j}})
\right]^{\alpha}
\right\}_{j\in\mathbb Z_+}\right\|_{X^{\frac1q}}
\\&\lesssim\sup_{Q\in\mathscr{Q}}
\left\{\frac{1}{|Q|^{\tau q}}\left\|
\left\{\sum_{j=j_Q\vee 0}^\infty
\sum_{k\in \mathbb N}
\left[\mathcal M(\mathbf{1}_{B_{k,j}})\right]^{\alpha}
\right\}^{\frac 1{\alpha}}
\right\|_{L^{\frac{\alpha p}{q}}(Q)}^{\alpha}\right\}\\
&\lesssim\sup_{Q\in\mathscr{Q}}
\left\{\frac{1}{|Q|^{\tau q}}\left\|\left\{
\sum_{j=j_Q\vee 0}^\infty
\sum_{k\in \mathbb N}
\mathbf{1}_{B_{k,j}}\right\}^{\frac 1{\alpha}}
\right\|_{L^{\frac{\alpha p}{q}}(Q)}^{\alpha}\right\}
\lesssim\left\|\left\{\sum_{k\in\mathbb N}
\mathbf{1}_{B_{k,j}}
\right\}_{j\in\mathbb Z_+}\right\|_{X^{\frac1q}},
\end{align*}
which completes the proof of Lemma \ref{as123j2}.
\end{proof}

Using Theorem \ref{thm-7-11}, Lemmas \ref{as3}
and \ref{as123j2},
Proposition \ref{pp},
and  Corollary \ref{pp2}, we obtain the following conclusions.

\begin{theorem}\label{noanfs}
Let  $s\in(0,\infty)$, $\tau\in[0,1/p)$,
$p\in(0,\infty)$, and $q\in(0,\infty)$.
Assume that $r\in\mathbb N$ with $r>s$ and that
the regularity parameter $L\in\mathbb N$ of the
Daubechies wavelet system $\{\psi_\omega\}_{\omega\in\Omega}$
satisfies \eqref{4.9} and $L\geq r-1$. Let $\varphi$
be the same as in \eqref{fanofn}.
Then the following statements hold.
\begin{itemize}
\item[\rm(i)]
For any $f\in\Lambda_s$,
\begin{align*}
&\mathop\mathrm{\,dist\,}\left(f, F^{s,\tau}_{p,q}
\cap\Lambda_s
\right)_{\Lambda_s}\\
&\quad\sim\inf\left\{\varepsilon\in(0,\infty): \sup_{Q\in\mathscr{Q}}\frac{1}{|Q|^{\tau}}
\left\|\left[\sum_{j=j_Q\vee 0}^\infty
\sum_{I \in W^0(s, f, \varepsilon)\cap\mathcal{D}_j}
\mathbf{1}_{I}
\right]^{\frac 1q}\right\|_{L^p(Q)}
<\infty\right\}
\end{align*}
with positive equivalence constants independent of $f$.
\item[\rm(ii)]
Let
$$\mathcal{A}:=\left\{\{a_k\}_{k\in\mathbb N}
\subset\mathbb Z^n:\lim_{k\to\infty}|a_k|=\infty,\
\left\|\left\{
\sum_{k\in\mathbb N}\mathbf{1}_{I_{0,a_k}}
\right\}_{j\in\mathbb{Z}_+}\right
\|_{(L^p_\tau l^q)_{\mathbb{Z}_+}}=\infty\right\}.$$
Then, for any $f\in\Lambda_s$,
 \begin{align*}
&\mathop\mathrm{\,dist\,}\left(f, F^{s,\tau
}_{p,q}\cap\Lambda_s
\right)_{\Lambda_s}\\&\quad\sim \inf
\left\{\varepsilon\in(0,\infty):\sup_{Q\in\mathscr{Q}}
\frac{1}{|Q|^{\tau}}
\left\|\left[\int_{0}^{2^{-j_{Q} \vee 0}}
\mathbf{1}_{ S_{r}(s, f, \varepsilon)}(\cdot,y)\,\frac {dy}{y}
\right]^{\frac 1q}\right\|_{L^p(Q)}
<\infty\right\}\\&\quad\quad+
\sup_{\{a_k\}_{k\in\mathbb N}\in\mathcal{A}}
\liminf_{k\rightarrow\infty}\left|\int_{\mathbb{R}^n}
\varphi(x-a_k)f(x)\,dx\right|
\end{align*}
with positive equivalence constants independent of $f$.
\end{itemize}
\end{theorem}

\begin{remark}
To the best of our knowledge,
Theorem \ref{noanfs} is completely new.
\end{remark}

Using Theorem \ref{pppz2a}, Corollary \ref{pp2}, and
Lemmas \ref{as3} and \ref{as123j2}, we obtain
the following conclusions.

\begin{theorem}\label{noanfs2}
Let  $s\in(0,\infty)$, $\tau\in[0,1/p)$,
$p\in(0,\infty)$, and $q\in(0,\infty)$.
Assume that $r\in\mathbb N$ with $r>s$ and that
the regularity parameter $L\in\mathbb N$ of the
Daubechies wavelet system $\{\psi_\omega\}_{\omega\in\Omega}$
satisfies \eqref{4.9} and $L\geq r-1$. Let $\varphi$
be the same as in \eqref{fanofn}.
Then the following statements hold.
\begin{itemize}
\item[\rm(i)]
$f\in\overline{F^{s,\tau}_{p,q}
\cap\Lambda_s}^{\Lambda_s}$
if and only if
$f\in \Lambda_s$ and,
for any $\varepsilon\in(0,\infty)$,
\begin{align*}
\sup_{Q\in\mathscr{Q}}\frac{1}{|Q|^{\tau}}
\left\|\left[\sum_{j=j_Q\vee 0}^\infty
\sum_{I \in W^0(s, f, \varepsilon)\cap
\mathcal{D}_j}\mathbf{1}_{I}
\right]^{\frac 1q}\right\|_{L^p(Q)}
<\infty.
\end{align*}
\item[\rm(ii)]
Let
\begin{align}\label{dsfafg2}
\mathcal{A}:=\left\{\{a_k\}_{k\in\mathbb N}
\subset\mathbb Z^n:\lim_{k\to\infty}|a_k|=\infty,\
\left\|\left\{
\sum_{k\in\mathbb N}\mathbf{1}_{I_{0,a_k}}
\right\}_{j\in\mathbb{Z}_+}\right
\|_{(L^p_\tau l^q)_{\mathbb{Z}_+}}=\infty\right\}.
\end{align}
Then $f\in\overline{F^{s,\tau}_{p,q}
\cap\Lambda_s}^{\Lambda_s}$
if and only if
$f\in \Lambda_s$ and,
for any $\varepsilon\in(0,\infty)$,
\begin{align*}
\sup_{Q\in\mathscr{Q}}
\frac{1}{|Q|^{\tau}}
\left\|\left[\int_{0}^{2^{-j_{Q} \vee 0}}
\mathbf{1}_{ S_{r}(s, f, \varepsilon)}(
\cdot,y)\,\frac {dy}{y}
\right]^{\frac 1q}\right\|_{L^p(Q)}<\infty
\end{align*}
and
\begin{align*}
\sup_{\{a_k\}_{k\in\mathbb N}\in\mathcal{A}}
\liminf_{k\rightarrow\infty}\left|\int_{
\mathbb{R}^n}
\varphi(x-a_k)f(x)\,dx\right|=0.
\end{align*}
\end{itemize}
\end{theorem}

\begin{remark}
To the best of our knowledge,
Theorem \ref{noanfs2} is completely new.
\end{remark}

Let  $\tau\in[0,1/p)$, $p\in(0,\infty)$, and
$q\in(0,\infty]$.
When $X:=(L^p_\tau l^q)_{\mathbb{Z}_+}$, for any
$f\in X$,
the
\emph{Lipschitz deviation degree} $\varepsilon^{\tau}_{p,q}f$ is defined by
setting
 $$\varepsilon^{\tau}_{p,q}f:=\inf\left\{\varepsilon
 \in(0,\infty):\sup_{Q\in\mathscr{Q}}
\frac{1}{|Q|^{\tau}}
\left\|\left[\int_{0}^{2^{-j_{Q} \vee 0}}
\mathbf{1}_{ S_{r}(s, f, \varepsilon)}(
\cdot,y)\,\frac {dy}{y}
\right]^{\frac 1q}\right\|_{L^p(Q)}<\infty\right\}.$$
Applying Theorems \ref{noanfs} and \ref{noanfs2}, we can
characterize the Lipschitz
deviation degree $\varepsilon^{\tau}_{p,q}$ by using the distance in the
Lipschitz space
and we omit
the details of its proof.

\begin{theorem}
Let  $s\in(0,\infty)$, $\tau\in[0,1/p)$,
$p\in(0,\infty)$, $q\in(0,\infty)$,
and $f\in\Lambda_s\cap \mathfrak{C}_0$.
\begin{itemize}
  \item[\rm (i)] If $\varepsilon\in(0,
  \varepsilon^{\tau}_{p,q}f)$, then
  $\sup_{Q\in\mathscr{Q}}
\frac{1}{|Q|^{\tau}}
\|[\int_{0}^{2^{-j_{Q} \vee 0}}
\mathbf{1}_{ S_{r}(s, f, \varepsilon)}(
\cdot,y)\,\frac {dy}{y}
]^{\frac 1q}\|_{L^p(Q)}=\infty$
and,
  for any $(x,y)\in\mathbb{R}^n\times(0,1]$
  with
  $(x,y)\notin S_{r}(s, f,
\varepsilon)$,
  $\frac{\Delta_r f(x, y)}{y^s}\le\varepsilon^{\tau}_{p,q}f$;
  \item[\rm (ii)] If $\varepsilon\in(
\varepsilon^{\tau}_{p,q}f,\infty)$,
  then $\sup_{Q\in\mathscr{Q}}
\frac{1}{|Q|^{\tau}}
\|[\int_{0}^{2^{-j_{Q} \vee 0}}
\mathbf{1}_{ S_{r}(s, f, \varepsilon)}(\cdot,y)\,\frac {dy}{y}
]^{\frac 1q}\|_{L^p(Q)}$ and,
  for any
  $(x,y)\in S_{r}(s, f,
\varepsilon)$,
  $\frac{\Delta_r f(x, y)}{y^s}>\varepsilon^{\tau}_{p,q}f$.
 \item[\rm (iii)] $
\mathop\mathrm{\,dist\,}(f, F^{s,\tau}_{p,q}\cap
\Lambda_s)_{\Lambda_s}
\sim \varepsilon^{\tau}_{p,q}f
$
with positive equivalence constants independent of $f$.
\item[\rm(iv)]
$f\in\overline{F^{s,\tau}_{p,q}\cap\Lambda_s
}^{\Lambda_s}$
if and only if
$f\in \Lambda_s$, $\varepsilon^{\tau}_{p,q}f=0$,
and
\begin{align*}
\sup_{\{a_k\}_{k\in\mathbb N}\in\mathcal{A}}
\liminf_{k\rightarrow\infty}\left|\int_{
\mathbb{R}^n}
\varphi(x-a_k)f(x)\,dx\right|=0,
\end{align*}
where $\mathcal{A}$ is the same as in \eqref{dsfafg2}.
\end{itemize}
\end{theorem}

\begin{theorem}\label{fwqf234z}
Let $\tau\in[0,1/p)$,  $s\in(0,\infty)$,
$p\in(0,\infty)$, and $q\in(0,\infty)$. Then
$F^{s,\tau}_{p,q}\cap\Lambda_s\subsetneqq
\overline{F^{s,\tau}_{p,q}\cap\Lambda_s
}^{\Lambda_s}\subsetneqq\Lambda_s$.
\end{theorem}

\begin{proof}
We first show that $\overline{F^{s,\tau}_{p,q}\cap\Lambda_s}^{\Lambda_s}
\subsetneqq\Lambda_s$. Let
$f_0:=\sum_{\{\omega\in \Omega: I_\omega\in\mathcal{D},\,
I_\omega\subset[0,1]^n\}}
|I_\omega|^{\frac sn+\frac 12}\psi_{\omega}.$ From
Lemma \ref{asqw},
it follows that $f_0\in\Lambda_s$. Observe that,
for any
$I\in\mathcal{D}$ and $I\subset[0,1]^n$,
$$\max_{\{\omega\in\Omega:I_{\omega} =I\}}\left|\int_{
\mathbb{R}^n} f_0(x) \psi_{\omega}(x)\, dx\right|
= |I|^{\frac sn+\frac 12}.$$
By this, we conclude that
\begin{align*}
\sup_{Q\in\mathscr{Q}}\frac{1}{|Q|^{\tau}}
\left\|\left[\sum_{j=j_Q\vee 0}^\infty
\sum_{I \in W^0(s, f, \frac12)\cap
\mathcal{D}_j}\mathbf{1}_{I}
\right]^{\frac 1q}\right\|_{L^p(Q)}
&\geq \left\|\left\{\sum_{j= 0}^\infty1\right\}^{\frac 1q}\right\|_{L^p([0,1]^n)}=\infty,
\end{align*}
which, combined with Theorem \ref{noanfs},
further implies that
$
\mathop\mathrm{\,dist\,}
(f_0,  F^{s,\tau}_{p,q}\cap\Lambda_s)_{
\Lambda_s}\gtrsim\frac12.
$
From this, it follows that $f_0\notin\overline{
F^{s,\tau}_{p,q}\cap\Lambda_s}^{\Lambda_s}$.
Thus, $\overline{F^{s,\tau}_{p,q}\cap\Lambda_s}^{\Lambda_s}
  \subsetneqq\Lambda_s$.

  Now, we show $F^{s,\tau}_{p,q}\cap\Lambda_s\subsetneqq
\overline{F^{s,\tau}_{p,q}\cap\Lambda_s}^{\Lambda_s}$.
Let
$f_0:=\sum_{j=0}^{\infty}j^{-\frac1q}
\sum_{\{\omega\in\Omega:I_\omega\in\mathcal{D}_j,
I_\omega\subset[0,1]^n \}}
|I_\omega|^{\frac sn+\frac 12}\psi_{\omega}$.
From  Lemma \ref{asqw},
it follows that $f_0\in\Lambda_s$. Observe that,
for any
$I\in\mathcal{D}_j$ and $I\subset[0,1]^n$,
$$\max_{\{\omega\in\Omega:I_{\omega} =I\}}\left|
\int_{\mathbb{R}^n} f_0(x) \psi_{\omega}(x)\,
dx\right|
=j^{-\frac1q}|I|^{\frac sn+\frac 12}.$$
By this, we conclude that, for any $\varepsilon
\in(0,\infty)$,
\begin{align*}
\sup_{Q\in\mathscr{Q}}\frac{1}{|Q|^{\tau}}
\left\|\left[\sum_{j=j_Q\vee 0}^\infty
\sum_{I \in W^0(s, f, \varepsilon)\cap
\mathcal{D}_j}\mathbf{1}_{I}
\right]^{\frac 1q}\right\|_{L^p(Q)}&
\lesssim\left\|\left\{\sum_{0 \leq j<1/\varepsilon^q}1\right\}^{\frac 1q}\right\|_{L^p([0,1]^n)}
<\infty,
\end{align*}
which, combined with Theorem \ref{noanfs2},
further implies that $f_0\in\overline{F^{s,\tau}_{p,q}\cap\Lambda_s
}^{\Lambda_s}$.
Moreover,
\begin{align*}
\|f_0\|_{F^{s,\tau}_{p,q}}
\sim\sup_{Q\in\mathscr{Q}}
\left\{\frac{1}{|Q|^\tau}\left\|\left[
\sum^\infty_{j=j_Q\vee 0}\frac1j
\sum_{\{\omega\in \Omega:I_\omega\in\mathcal{D}_j,\,
I_\omega\subset[0,1]^n\}}
\mathbf{1}_{I_\omega}\right]^{\frac 1q}\right\|_{L^p(Q)}
\right\}\gtrsim\left\|\left[\sum_{j=0}^{\infty}\frac1j
\right]^{\frac 1q}\right\|_{L^p(Q)}
=\infty,
\end{align*}
which
further implies that $f_0\notin F^{s,\tau}_{p,q}\cap\Lambda_s$.
Thus, $ F^{s,\tau}_{p,q}\cap\Lambda_s\subsetneqq
\overline{F^{s,\tau}_{p,q}\cap\Lambda_s}^{\Lambda_s}
$,
which completes the proof of Theorem \ref{fwqf234z}.
\end{proof}

\begin{appendices}
\section{Hyperbolic metrics}

In this Appendix, we review and explore
some basic properties of hyperbolic metrics
in the upper half space, especially in view of
readers not so familiar with the use of
hyperbolic metrics in analysis. Many of the
statements below are certainly well known, but
some of our quantitative estimates could be of
independent interest and  useful in other
applications on function spaces. We  refer
e.g. to the books \cite{bp,r} for general
treatments of hyperbolic spaces and metric.

\begin{definition} Let $a,b\in{\mathbb R}$.
Assume that
$$
\gamma:=\left\{\gamma(t):=\left(
\gamma_1(t),\ldots, \gamma_{n+1}(t)\right)
\in\mathbb{R}_+^{n+1}:t\in [a,b]\right\}
$$
is a piecewise $\mathfrak{C}^1$-curve.
Then the \emph{hyperbolic length} $L(\gamma)$
 of $\gamma$  is defined by setting
$$
L(\gamma) :=\int_a^b \sqrt{ \sum_{j=1}^{n+1} |
\gamma_j'(t)|^2 }\,\frac {dt} {\gamma_{n+1}(t)}
=\int_a^b\frac {|\gamma'(t)|} {\gamma_{n+1}(t)}\,dt.
$$
\end{definition}

In what follows, let $e_{n+1}:=(0,\ldots, 0, 1)
\in\mathbb{R}^{n+1}$
(that is, only the last entry is $1$).
\begin{definition}\label{asldj}
\begin{itemize}

\item[\rm (i)]
For any given $x_0\in\mathbb{R}^n$, let
$V_{x_0}:=\{(x_0, y)\in\mathbb{R}_+^{n+1}:y\in(0,\infty)\}.$

\item[\rm (ii)]
Recall that $\mathbb{S}^{n-1}$ denotes the
unit sphere of $\mathbb{R}^n$.
For any given $x_0\in\mathbb{R}^n$,
$ \xi\in\mathbb{S}^{n-1}$, and  $r\in(0,\infty)$,
let
$ C(x_0, \xi, r):=\left\{ (x_0,0)+r\cos
\theta(\xi,0)+r\sin \theta e_{n+1}:
\   \theta\in (0, \pi)\right\},$
which is the upper half of the circle in the
hyperplane $(x_0, 0)+\operatorname{span}
\{(\xi,0), e_{n+1}\}$  with center $(x_0, 0)$
and radius $r$.

\item[\rm (iii)]
The \emph{geodesics class} $G(\mathbb{R}^{n+1}_+)$
is defined to be the set of all vertical
lines $V_{x_0}$ with $x_0\in\mathbb{R}^n$ and
all  upper half circles $C(x_0, \xi, r)$ with
$x_0\in\mathbb{R}^n$, $\xi\in\mathbb{S}^{n-1}$,
and $r\in(0,\infty)$.
\end{itemize}
\end{definition}

By Definition \ref{asldj}, we have the following
result and we omit
the details of its proof.

\begin{proposition}
For any given distinct points $\mathbf p,
\mathbf q\in \mathbb{R}_+^{n+1}$,
there exists a  unique geodesic
$v\in G(\mathbb{R}^{n+1}_+)$ connecting
$\mathbf p$ and $\mathbf q$.
\end{proposition}

We use $\gamma(\mathbf p,\mathbf q)$ to denote
the
unique geodesic $v\in G(\mathbb{R}^{n+1}_+)$
connecting $\mathbf p$ and $\mathbf q$ and use
$\gamma_{\mathbf  p, \mathbf q}$ to
denote the \emph{standard parametric
representation} of $\gamma(\mathbf  p,
\mathbf  q)$. On the hyperbolic length
$L(\gamma(\mathbf p,\mathbf  q))$, we have
the following conclusion.

\begin{theorem}\label{thm-2-8}
For any $\mathbf p, \mathbf q\in\mathbb{R}_+^{n+1}$,
$L(\gamma(\mathbf p,\mathbf  q)) =\rho(
\mathbf p,\mathbf q)=\inf\left\{ L(\gamma): \gamma\in\Gamma(\mathbf p,\mathbf q)\right\},$
where $\Gamma(\mathbf p,\mathbf q)$ denotes the
set of all piecewise $\mathfrak{C}^1$-curves in
$\mathbb{R}_+^{n+1}$ that connect $\mathbf p,
\mathbf q\in\mathbb{R}_+^{n+1}$.
\end{theorem}

A straightforward
calculation shows the following two lemmas
and we omit the
details of their proofs.

\begin{lemma}\label{dsfa}
For any $\mathbf p \in\mathbb{R}_+^{n+1}$, let
$
\Phi(\mathbf p ) := \frac{\mathbf p} {|\mathbf p|^2}.
$
Then,
for any $\mathbf{p}, \mathbf{q}\in\mathbb{R}_+^{n+1}$,
$
\rho\left(\Phi(\mathbf p), \Phi(\mathbf  q)\right)
=\rho(\mathbf  p, \mathbf  q).
$
Moreover, for  any piecewise $\mathfrak{C}^1$-curve with
parametric representation $\gamma$,
$
L\left(\gamma\right)=L( \Phi\circ \gamma).
$
\end{lemma}

\begin{lemma}\label{dpfjp}
Let   $\xi\in\mathbb{S}^{n-1}$ and  $r\in(0,\infty)$.
Let $\Gamma$ denote the upper half circle with center
$(r\xi,0)$ and radius $r$ in the plane
$\operatorname{span}\{\xi, e_{d+1}\}$ defined by setting
$$
\Gamma:= \left\{ (x\xi, y):(x-r)^2 +y^2 =r^2\
\text{and}\ x, y\in(0,\infty)\right\}.
$$
Then $\Phi(\Gamma)$ is a vertical line and,
for any $(x\xi, y)\in \Gamma$,
$\Phi((x\xi, y)) = (\frac \xi {2r}, \frac 1{2r}
\frac yx).$
\end{lemma}

A straightforward calculation also shows that,
for any $\mathbf{x}:=( x, x_{n+1}),\mathbf{y}:=
( y, y_{n+1}) \in \mathbb{R}_+^{n+1}$,
\begin{align}\label{2-1eq}
\rho(\mathbf x, \mathbf y) &=2\ln \left(\frac {
|\mathbf x-\mathbf y|  +|\mathbf x-\widetilde{
\mathbf y}|}{ 2 \sqrt{x_{n+1} y_{n+1}} }\right),
\end{align}
where $\widetilde{ \mathbf y} =(y, -y_{n+1})$.

\begin{proof}[Proof of Theorem \ref{thm-2-8}]
Let $\mathbf  p:=(p, p_{n+1})$ and $
\mathbf q=(q, q_{n+1})$.
We now consider the following two cases.
If $p=q$,
in this case, $\gamma(\mathbf  p, \mathbf q)$ is
a vertical line segment.
Without loss of generality, we may assume $p_{n+1}>q_{n+1}$.
Let $\gamma\in \Gamma(\mathbf p,\mathbf  q)$
be such that $\gamma(0)=\mathbf  p$ and $\gamma(1)=\mathbf q$. Then, by \eqref{2-1eq}, we obtain
\begin{align}\label{dsafsd}
L(\gamma)&=\int_0^1 \frac {|\gamma'(t)|}{\gamma_{n+1}(t)}\,
dt \ge \int_0^1 \frac {\gamma_{n+1}'(t)}{\gamma_{n+1}(t)} \,
dt
=\ln \frac {\gamma_{n+1}(1)}{\gamma_{n+1}(0)}=\ln \frac {
q_{n+1}}{p_{n+1}} =\rho(\mathbf  p,\mathbf q).
\end{align}
Moreover, for any $t\in[q_{n+1},p_{n+1}]$, let
$\gamma_{\mathbf p, \mathbf  q}(t):=( p, t).$
Then we have
\begin{align*}
L(\gamma_{(\mathbf  p, \mathbf q)})&=\int_{p_{n+1}
}^{q_{n+1}}
\frac {1}{t}\, dt =\ln \frac {  q_{n+1}}{p_{n+1}}
=\rho(\mathbf  p,\mathbf q),
\end{align*}
which, combined with \eqref{dsafsd}, further implies
the desired conclusion.

If $p\neq q$,
in this case, $\gamma(\mathbf  p,\mathbf q)$ has a
parametric representation:
for any $t\in[\alpha,\beta]$ with $0<\alpha< \beta<
\pi$,
$
\gamma_{\mathbf p, \mathbf  q}(t):=  \left( x_0+[r\cos t
] \xi, r\sin t\right),
$
where $x_0\in\mathbb{R}^n$, $\xi\in\mathbb{S}^{n-1}$,
 $r\in(0,\infty)$, $\gamma_{\mathbf p, \mathbf  q}(
 \alpha)=\mathbf  p$, and $\gamma_{\mathbf p,
 \mathbf  q}(\beta)=\mathbf  q$.
Without loss of generality, we may assume $x_0:=r\xi$.
Then, by Lemma \ref{dpfjp}, we conclude that, for any
$t\in[\alpha,\beta]$,
\begin{align*}
\Phi\circ \gamma_{\mathbf p,\mathbf q}(t)=\left(
\frac \xi {2r},
\frac {\sin t}{2r[1+\cos t]}\right),
\end{align*}
which, together with the fact that,  for any $t\in[\alpha,\beta]$,
$(\Phi\circ \gamma_{\mathbf p,\mathbf q})_{n+1}'(t)\in(0,\infty)$,
further implies that
\begin{align*}
L(\Phi\circ \gamma_{\mathbf p,\mathbf q}(t))&=\int_{
\alpha}^{\beta}
\frac {|(\Phi\circ \gamma_{\mathbf p,\mathbf q})'(t)|
}{(\Phi\circ \gamma_{\mathbf p,\mathbf q})_{n+1}(t)}\, dt =
\ln \frac {(\Phi\circ \gamma_{\mathbf p,\mathbf q})_{n+1}(\alpha)}{(\Phi\circ
\gamma_{\mathbf p,\mathbf q})_{n+1}(\beta)}=
\rho(\Phi(\mathbf p),\Phi(\mathbf q)).
\end{align*}
From this and Lemma \ref{dsfa}, we deduce that,
for any $\gamma\in\Gamma(\mathbf p,\mathbf q)$,
\begin{align*}
L(\gamma) =L(\Phi\circ \gamma) \ge \rho(\Phi(
\mathbf p),\Phi(\mathbf q))=
L( \Phi\circ \gamma_{\mathbf p,\mathbf q})=L(
\gamma_{\mathbf p,\mathbf q})
=\rho(\mathbf p,\mathbf q).
\end{align*}
This finishes the proof of Theorem \ref{thm-2-8}.
\end{proof}

By Theorem \ref{thm-2-8}, we have the following
result immediately
and we omit the
details of its proof.

\begin{corollary}\label{mosf}
The function $\rho$ defined in \eqref{metric} is
 a metric on $\mathbb{R}_+^{n+1}$.
\end{corollary}

Now, we establish the following properties about
the Poincar\'e hyperbolic metric.
For any $t\in(0,\infty)$ and $\mathbf{z}:=
(z,z_{n+1})\in\mathbb{R}_+^{n+1}$, let
$
B_\rho(\mathbf{z}, t):=\{ \mathbf{y}\in\mathbb{R
}_+^{n+1}:\rho(\mathbf{z}, \mathbf{y})< t\}.
$

\begin{lemma}\label{ddaf}
\begin{itemize}
\item[\rm(i)]
For any $t\in(0,1/4]$ and $\mathbf{z}\in\mathbb{
R}^{n+1}_+$,
\begin{equation}\label{dafg}
T\left(\mathbf{z}, \frac t4\right) \subset B_\rho(
\mathbf{z}, t)\subset
	T(\mathbf{z}, 2t),
\end{equation}
where
$
T(\mathbf{z}, t):= B(z, tz_{n+1})\times [(1-t) z_{
n+1},  (1+t) z_{n+1}].
$
\item[\rm(ii)]
Let $\mathbf{z}:=(z,z_{n+1})\in \mathbb{R}^{
n+1}_+$, $t\in(0,1/4]$, and $\mathbf{x}:=(x,
x_{n+1})\in B_\rho(\mathbf{z},t)$.
Then $z_{n+1}\sim x_{n+1}$ and
\begin{equation}\label{5-1-0}
|B_\rho(\mathbf{x},t) |\sim |B_\rho(\mathbf{z},t)|
\end{equation}
with positive equivalence constants independent of
$\mathbf{z},\mathbf{x}$, and $t$.
\item[\rm(iii)]
Let
$d\mu (x,x_{n+1}):=\frac {dx\,dx_{n+1}} {x_{n+1}}.$
Then, for any $\delta\in(0,1/4]$ and $\mathbf{z}
\in\mathbb{R}^{n+1}_+$,
$
\mu(B_\rho(\mathbf{z},\delta))\sim \delta(z_{n+1}
\delta)^n
$
with positive equivalence constants independent of
$\mathbf{z}$ and $t$.
\item[\rm(iv)]
Let $\delta\in(0,1/4]$
and $R\in(1,\infty)$.
Then, for any $\mathbf{z}\in\mathbb{R}^{n+1}_+$,
$
\mu(B_\rho(\mathbf{z},R))\lesssim\mu(B_\rho(\mathbf{z},\delta))
$
with the implicit positive constant independent of $\mathbf{z}$.
\end{itemize}
\end{lemma}
\begin{proof}
We first show (i).
We claim that,
for any
$\mathbf{x}\in\mathbb{R}^{n+1}_+$ with $\rho(
\mathbf{x},\mathbf{z})\le1/2$,
\begin{equation}\label{dafg1}
(1+2\rho(\mathbf{x},\mathbf{z}))^{-1}  x_{n+1} \leq z_{n+1}\leq
(1+2\rho(\mathbf{x},\mathbf{z}))  x_{n+1}
\end{equation}
and
\begin{equation}\label{dafg2}
\frac12 {|\mathbf{x}-\mathbf{z}| }\leq { x_{n+1}}
\rho(\mathbf{x},\mathbf{z}) \leq 2 |\mathbf{x}-\mathbf{z}|.
\end{equation}
Indeed, by the definition of $\rho$,
we conclude that, for any
$\mathbf{x}\in\mathbb{R}^{n+1}_+$,
$
\rho(\mathbf{x},\mathbf{z})
=\rho( \frac {\mathbf{x}} {x_{n+1}}, \frac {
\mathbf{z}} {x_{n+1}}),
$
which further implies that
\begin{align}\label{dafjp}
\left|\frac {\mathbf{x}} {x_{n+1}}- \frac {
\mathbf{z}} {x_{n+1}}\right|^2= 2 \frac {z_{n+1}}{x_{n+1}}
(\cosh \rho(\mathbf{x},\mathbf{z})-1).
\end{align}
Observe that, for any $t\in\mathbb{R}$ ,
\begin{align}\label{dafjp2}
\cosh t -1=\frac 12 e^t+\frac12 e^{-t} -1=
\sum_{n=1}^\infty \frac {t^{2n}}{(2n)!},
\end{align}
which further implies that, for any $t\in(0,
 \frac12]$,
$\frac {t^2} 2\leq  \cosh t-1\leq t^2$.
From this and \eqref{dafjp},
it follows  that
\begin{equation}\label{1-5b}
\frac {z_{n+1}}{x_{n+1}} \left[\rho(\mathbf{
x},\mathbf{z})
\right]^2 \leq \left| \frac {\mathbf{x}} {x_{
n+1}}-\frac {\mathbf{z}} {x_{n+1}}\right|^2\leq
2 \frac {z_{n+1}}{x_{n+1}} \left[\rho(\mathbf{x},
\mathbf{z})\right]^2,
\end{equation}
which further implies $|1-\frac {z_{n+1}}{x_{n+1}
}|^2\leq 2 \left[\rho(\mathbf{x},\mathbf{z})
\right]^2 \frac {z_{n+1}}{x_{n+1}}$.
Thus, after some simple computations, we find that
\begin{equation}\label{eqeq}
\frac 1 {1+2\rho(\mathbf{x},\mathbf{z})} \leq \frac {
z_{n+1}}{x_{n+1}}\leq 1+2\rho(\mathbf{x},\mathbf{z})
\end{equation}
and hence \eqref{dafg1} holds.
Using $\rho(\mathbf{x},\mathbf{z})\le1/2$,
\eqref{1-5b}, and \eqref{dafg1},
we obtain
\begin{align*}
\frac 12 \rho(\mathbf{x},\mathbf{z})\leq \frac {\rho(\mathbf{x},\mathbf{z})}{\sqrt{1+2\rho(\mathbf{
x},\mathbf{z})}}\leq \left| \frac {\mathbf{x}} {
x_{n+1}}-\frac {\mathbf{z}} {x_{n+1}}\right|
\leq \sqrt{2(1+2\rho(\mathbf{x},\mathbf{z}))} \rho(
\mathbf{x},\mathbf{z})\leq 2\rho(\mathbf{x},\mathbf{z}).
\end{align*}
This proves  \eqref{dafg2} and hence the above
claim holds.
By \eqref{dafg2}, we conclude that, for any $
\mathbf{x}\in B_\rho(\mathbf{z}, t)$,
$|\mathbf{x}-\mathbf{z}| \leq 2 x_{n+1} t.$
This further implies that
\begin{align}\label{fasldjf1}
B_\rho(\mathbf{z}, t)\subset
T(\mathbf{z}, 2t).
\end{align}
Now, we show that
\begin{align}\label{fasldjf}
T\left(\mathbf{z}, \frac t4\right) \subset B_\rho
(\mathbf{z}, t).
\end{align}
Observe that, for any $\mathbf{x}\in T(\mathbf{
z}, \frac t4)$,
$
|x-z|\leq \frac t4 z_{n+1}$ and $|x_{n+1}-z_{
n+1}|\leq \frac t4 z_{n+1},
 $
which, together with \eqref{dafg2}, further
implies that
$\rho(\mathbf{x},\mathbf{z}) \leq \frac 2{z_{
n+1}}|\mathbf{x}-\mathbf{z}| \leq t.$
Thus, \eqref{fasldjf} holds, which, combined with
\eqref{fasldjf1}, then completes the proof of (i).

As for (ii), from \eqref{dafg1} and (i), we infer
immediately that
(ii) holds.

Now, we prove (iii).
By (ii), we conclude that, for any $\mathbf{x}\in
B_\rho(\mathbf{z},\delta)$, $x_{n+1}\sim z_{n+1}$,
which further implies that
 \begin{align}\label{fasdo}
\mu \left( B_\rho(\mathbf{z}, \delta) \right)\leq
\int_{B_\rho(\mathbf{z}, \delta)}\,\frac
{dx\,dx_{n+1}}{x_{n+1}} \lesssim  z_{n+1}^{-1}
|B_\rho(\mathbf{z}, \delta)|\lesssim z_{n+1}^{-1
}|T(a,2\delta)|\lesssim \delta (z_{n+1}\delta)^{n}.
 \end{align}
From (i),
 we deduce that
 \begin{align*}
 \mu \left( B_\rho(\mathbf{z}, \delta) \right)\ge
 \mu \left( T(\mathbf{z}, \delta/4) \right)\ge
 \int_{B(z, \frac \delta4 z_{n+1})} \int_{(1-
 \frac \delta4) z_{n+1}}^{z_{n+1}} \,\frac {
 dx_{n+1}}{x_{n+1}}\, dx \sim \delta (z_{n+1} \delta)^n,
 \end{align*}
which, together with \eqref{fasdo}, further implies that
(iii) holds.

As for (iv), by \eqref{dafjp2}, we conclude that,
for any $\mathbf{x} \in B_\rho(\mathbf{z}, R)$,
$$
\left|\mathbf{z}-\mathbf{x}\right|^2=2 z_{n+1}x_{
n+1}\sum_{n=1}^\infty \frac {[\rho(\mathbf{x
},\mathbf{z})]^{2n}}{(2n)!},
$$
which further implies that
\begin{align}\label{fadsfb}
\left|z_{n+1}-x_{n+1}\right|^2\lesssim z_{n+1
}x_{n+1}
\end{align}
and
\begin{align}\label{fadsfb2}
|z-x|^2\lesssim z_{n+1}x_{n+1}.
\end{align}
From \eqref{fadsfb} and \eqref{fadsfb2}, it
follows that
there exists a positive constant $C_{(R)}$,
depending on $R$, such that
$\frac {x_{n+1}}{z_{n+1}}\leq C_{(R)}$
and $|z-x|\le C_{(R)}(z_{n+1}x_{n+1})^{\frac 12},$
which, combined with (iii), further implies that
\begin{align*}
\mu(B_\rho(\mathbf{z}, R)) &\leq \iint_{B_\rho(
\mathbf{z}, R)} \,\frac {dx\,dx_{n+1}}{x_{n+1}}\\
&\leq
\sum_{\{j\in\mathbb Z:2^j\leq C_{(R)}\}}\int_{
2^j z_{n+1}}^{2^{j+1} z_{n+1}}
\left[ \int_{|z-x|\leq C_{(R)} 2^{\frac{(j+1)}2} z_{
n+1}}\,dx \right]\,\frac {dx_{n+1}}{x_{n+1}}\\
&\lesssim z_{n+1}^n \sum_{\{j\in\mathbb Z:2^j \leq C_{
(R)}\}} 2^{\frac{jn}2}\lesssim z_{n+1}^n\lesssim \mu(
B_\rho(\mathbf{z}, \delta)).
\end{align*}
This finishes the proof of (iv) and hence
Lemma \ref{ddaf}.
\end{proof}

\begin{lemma}\label{fadsfb3}
For any given distinct points $\mathbf x,\mathbf y
\in \mathbb{R}_+^{n+1}$,
if $\mathbf p$ is a point on the arc $\gamma(
\mathbf x, \mathbf y)$, then
$$
\rho(\mathbf x, \mathbf y)=\rho(\mathbf x, \mathbf p
)+\rho(\mathbf p, \mathbf y).
$$
\end{lemma}

\begin{proof}
Let $\gamma(\mathbf x, \mathbf y)$ have a parametric
representation $
\gamma_{\mathbf x, \mathbf y}(t)$
for any $t\in[\alpha,\beta]$ with $0<\alpha< \beta<\pi$
satisfying
$\gamma_{\mathbf x, \mathbf y}(\alpha)=\mathbf x$,
$\gamma_{\mathbf x, \mathbf y}(\beta)=\mathbf  y$,
and
$\gamma_{\mathbf x, \mathbf y}(\theta)=\mathbf  p$.
From this and Theorem \ref{thm-2-8}, it follows that
\begin{align*}
\rho(\mathbf x, \mathbf y)&=
L( \gamma_{\mathbf x, \mathbf y}(t))=\int_{\alpha}^{\beta}
\frac {|( \gamma_{\mathbf x, \mathbf y})'(t)|}{( \gamma_{\mathbf x, \mathbf y})_{n+1}(t)}\, dt \\
&=
\int_{\alpha}^{\theta}
\frac {|( \gamma_{\mathbf x, \mathbf y})'(t)|}{( \gamma_{\mathbf x, \mathbf y})_{n+1}(t)}\, dt+\int_{\theta}^{\beta}
\frac {|( \gamma_{\mathbf x, \mathbf y})'(t)|}{( \gamma_{\mathbf x, \mathbf y})_{n+1}(t)}\, dt=
\rho(\mathbf x, \mathbf p)+\rho(\mathbf p,\mathbf y).
\end{align*}
This finishes the proof of
Lemma \ref{fadsfb3}.
\end{proof}

\begin{lemma}\label{dda2f}
\begin{itemize}
\item[\rm(i)]
If $\{A^\alpha\}_{\alpha\in I}$ is a collection of  subsets of $\mathbb{R}^{n+1}_+$  having the property \eqref{5-1-00},
then the union  $\bigcup_{\alpha\in I} A^\alpha$ satisfies \eqref{5-1-00}.
\item[\rm(ii)]
For any cube $I\in\mathbb{R}^n$, $T(I)$ satisfies \eqref{5-1-00}.
\item[\rm(iii)]
For any $S\subset\mathbb{R}^{n+1}_+$, $S_\delta$ satisfies \eqref{5-1-00}.
\end{itemize}
\end{lemma}

\begin{proof}
We first show (i).
Observe that,
for any $\mathbf{z}\in \bigcup_{\alpha\in I} A^\alpha$,
\begin{align*}
	\left|B_\rho(\mathbf{z}, \delta)\cap\left[\bigcup_{\alpha\in I} A^\alpha\right]\right|
=\left|\bigcup_{\alpha\in I}
B_\rho(\mathbf{z}, \delta)\cap A^{\alpha_0}\right|
	\ge |B_\rho(\mathbf{z},\delta)\cap A^{\alpha_0}|
\gtrsim|B_\rho(\mathbf{z}, \delta)|,
\end{align*}
where  $\alpha_0\in I$ satisfies $\mathbf{z}\in A^{\alpha_0}$.
This further implies that $\bigcup_{\alpha\in I} A^\alpha$ satisfies \eqref{5-1-00}.
Thus, (i) holds.

Next, we prove (ii).
By Lemma \ref{dafg}(i), we conclude that, for any $\mathrm{z}=(z, z_{n+1})\in T(I)$, $\frac 12 \ell(I) \leq z_{n+1}\leq \ell(I)$ and
$T(\mathrm{z}, \delta/2)\subset B_\rho(\mathrm{z}, \delta)$, which further implies that
\begin{align*}
|B_\rho(\mathrm{z}, \delta)\cap T(I)|&\ge |T(\mathrm{z}, \delta/2) \cap T(I)|\gtrsim |T(\mathrm{z}, \delta )|\gtrsim |B_\rho(\mathrm{z}, \delta)|.
\end{align*}
Therefore, (ii) holds.

By the fact $S_\delta=\bigcup_{\mathbf{p}\in S}B_\rho(\mathbf{p}, \delta)$, we
conclude that, to show (iii), it suffices to prove, for any $\mathbf{p}\in S $,
$B_\rho(\mathbf{p}, \delta)$ satisfies \eqref{5-1-0}.
Indeed, fix  $\mathbf{z}\in B_\rho(\mathbf{p}, \delta)$.
By Lemma \ref{fadsfb3}, we conclude that
there exists $\mathbf{q}\in \gamma(\mathbf{p}, \mathbf{z})$ such that
$\rho(\mathbf{p}, \mathbf{q})=\rho(\mathbf{q},\mathbf{z})=\frac12 \rho(\mathbf{p}, \mathbf{z})\leq \frac12\delta.$
Then we have
$B_\rho( \mathbf{q}, \frac 12\delta) \subset B_\rho(\mathbf{z},\delta) \cap B_\rho(\mathbf{p},\delta)$, which further implies that
\begin{align*}
|B_\rho(\mathbf{z},\delta) \cap B_\rho(\mathbf{p},\delta)|\ge
\left|B_\rho\left( \mathbf{q}, \frac{\delta}{2}\right)\right|	\ge  \left|	T\left( \mathbf{q}, \frac{\delta}{8}  \right)\right|
\sim (\delta q_{n+1})^{n+1}\sim|B_\rho(\mathbf{z},\delta)|.
\end{align*}
Therefore, (iii) holds.
This finishes the proof of Lemma \ref{dda2f}.
\end{proof}

\begin{lemma}\label{lem-3-2-0}
There
exists a positive constant $C$ such that, for any $R\in(1,\infty)$ and
$(x,y)\in{\mathbb R}_+^{n+1}$,
\begin{align*}
\left\{ (u,v)\in{\mathbb R}_+^{n+1}:|u-x|\leq   R y,\  R^{-1}y \leq v \leq Ry \right\}
\subset B_\rho\left( (x,y), CR\right).
\end{align*}
\end{lemma}

\begin{proof}
Let $(u,v)\in{\mathbb R}_+^{n+1}$ satisfy  $|u-x|\leq   R y$  and  $R^{-1}y \leq v \leq Ry$.
Observe that there
exists a positive constant $\widetilde{C}$ such that, for any $R\in(1,\infty)$,
$$
 \cosh\left(\widetilde{C} R\right)  -1\ge e^{\widetilde{C} R} \left(\frac{1- e^{- \widetilde{C} R}}2\right)  \ge  R^3.
$$
From this,
we deduce that
\begin{align*}
&|u-x|^2 +|v-y|^2 \leq 2 R^2 y^2 \leq 2 R^3 yv\leq 2 yv \left[\cosh \left(\widetilde{C} R\right)-1\right],
\end{align*}
which, combined with the definition of $\rho$, further implies that
$(u,v)\in B_\rho( (x,y), \widetilde{C} R).$
This finishes the proof of Lemma \ref{lem-3-2-0}.
\end{proof}

\begin{lemma}\label{lem-3-24}
Let $R\in(0,\infty)$.
Then there
exists a positive constant $C_{(R)}$, depending only on $R$, such that, for any
$(x,y)\in{\mathbb R}_+^{n+1}$,
\begin{align*}
B_\rho\left( (x,y), R\right)
\subset\left\{ (u,v)\in{\mathbb R}_+^{n+1}:|u-x|\leq   C_{(R)}y,\  C_{(R)}^{-1}y \leq v \leq C_{(R)}y \right\}.
\end{align*}
\end{lemma}

\begin{proof}
Let $\mathbf{x}:=(x,y)$.
By the definition of $\rho$,
we conclude that, for any
$\mathbf{z}:=(z,z_{n+1})\in\mathbb{R}^{n+1}_+$,
$$
\rho(\mathbf{x},\mathbf{z})=\rho\left( \frac {\mathbf{x}} {x}, \frac {\mathbf{z}} {x}\right),
$$
which further implies that
$|\frac {\mathbf{x}} {x}- \frac {\mathbf{z}} {x}|^2= 2 \frac {z_{n+1}}{x}
[\cosh \rho(\mathbf{x},\mathbf{z})-1].$
From this,
it follows  that, for any $\mathbf{z}\in B_\rho\left( (x,y), R\right)$,
\begin{equation}\label{1-5bs}
\left| \frac {\mathbf{x}} {x}-\frac {\mathbf{z}} {x}\right|^2\leq
2\frac {z_{n+1}}{x} A_{R},
\end{equation}
where $A_R:=\cosh R-1$. By  an argument
similar to that used in the estimation of \eqref{eqeq}, we obtain
\begin{equation*}
\frac 1 {1+2\sqrt{A_R}} \leq \frac {z_{n+1}}{x}\leq 1+2\sqrt{A_R},
\end{equation*}
which further implies that there exists a positive constant $C$ such that, for any
$\mathbf{z}\in\mathbb{R}^{n+1}_+$,
\begin{equation}\label{1-5bss}
C^{-1}x \leq z_{n+1}\leq Cx.
\end{equation}
From this and \eqref{1-5bs},
we infer that
$\left|\mathbf{x}-\mathbf{z}\right|^2\lesssim x z_{n+1}\lesssim x^2,$
which, together with \eqref{1-5bss}, then completes
the proof of Lemma \ref{lem-3-24}.
\end{proof}
\end{appendices}

\noindent\textbf{Acknowledgements}.\quad The authors would like
to thank Professor Hasi Wulan
for some helpful discussions on the topic of this article.

\bigskip

\smallskip

\noindent Feng Dai

\smallskip

\noindent Department of Mathematical and
Statistical Sciences, University of Alberta
Edmonton, Alberta T6G 2G1, Canada

\smallskip

\noindent {\it E-mail}: \texttt{fdai@ualberta.ca}

\bigskip

\noindent Eero Saksman

\smallskip

\noindent Department of Mathematics and Statistics, University of Helsinki,
P. O. Box 68, FIN-00014,
Helsinki, Finland

\smallskip

\noindent {\it E-mail}: \texttt{eero.saksman@helsinki.fi}

\bigskip

\noindent Dachun Yang, Wen Yuan and Yangyang Zhang

\smallskip

\noindent Laboratory of Mathematics and Complex Systems
(Ministry of Education of China),
School of Mathematical Sciences, Beijing Normal University,
Beijing 100875, The People's Republic of China

\smallskip

\noindent{\it E-mails:} \texttt{dcyang@bnu.edu.cn} (D. Yang)

\noindent\phantom{{\it E-mails:}} \texttt{wenyuan@bnu.edu.cn} (W. Yuan)

\noindent\phantom{{\it E-mails:}} \texttt{yangyzhang@bnu.edu.cn} (Y. Zhang)

\end{document}